\documentclass[EJP]{ejpecp} % replace ECP by EJP if needed.
\usepackage{amsmath,bm} 
\usepackage{tikz-cd}

\SHORTTITLE{The permuton limit of strong-Baxter and semi-Baxter permutations}

\TITLE{The permuton limit of strong-Baxter and semi-Baxter permutations is the skew Brownian permuton}
\AUTHORS{%
  Jacopo~Borga\footnote{Stanford University,
    \EMAIL{jborga@stanford.edu}}}
 % Separate items with ;

\AMSSUBJ{60C05; 60G50; 05A05; 34K50} % Edit. Separate items with ;
\SUBMITTED{February 28, 2022} % Edit.
\ACCEPTED{November 16, 2022} % Edit.

\VOLUME{0}
\YEAR{2020}
\PAPERNUM{0}
\DOI{10.1214/YY-TN}

\ABSTRACT{The \emph{skew Brownian permuton} is a new universal family of random permutons, depending on two parameters, which should describe the permuton limit of several models of pattern-avoiding permutations.
	For some specific choices of the parameters, the skew Brownian permuton coincides with some previously studied permutons: the \emph{biased Brownian separable permuton} and the \emph{Baxter permuton}. 
	The latter two permutons are degenerate cases of the skew Brownian permuton. 
	
	In the present paper we prove the first convergence result towards a non-degenerate skew Brownian permuton. Specifically, we prove that \emph{strong-Baxter permutations} converge in the permuton sense to the skew Brownian permuton for a non-degenerate choice of the two parameters. In order to do that, we develop a robust technique to prove convergence  towards the skew Brownian permuton for various families of random constrained permutations. This technique relies on generating trees for permutations, allowing an encoding of permutations with multi-dimensional walks in cones. We apply this technique also to \emph{semi-Baxter permutations}.}

\def\N{\mathbb{N}}
\def\R{\mathbb{R}}
\def\Z{\mathbb{Z}}
\def\E{\mathbb{E}}
\def\P{\mathbb{P}}

\DeclareMathOperator{\Var}{Var}

\DeclareMathOperator{\sgn}{sgn}
\DeclareMathOperator{\idf}{\mathds{1}}
\DeclareMathOperator{\Leb}{Leb}
\DeclareMathOperator{\Id}{Id}
\DeclareMathOperator{\mult}{mult}

\DeclareMathOperator{\CL}{CL}

\DeclareMathOperator{\Cor}{Cor}
\DeclareMathOperator{\wcp}{WC}
\DeclareMathOperator{\cpp}{CP}
\DeclareMathOperator{\pw}{PW}
\DeclareMathOperator{\indperm}{pat}
\DeclareMathOperator{\resc}{std}
\DeclareMathOperator{\FV}{FV}

\DeclareMathOperator{\cov}{cov}
\DeclareMathOperator{\dist}{dist}
\DeclareMathOperator{\Perm}{Perm}
\DeclareMathOperator{\stc}{stc}

\newcommand{\Coals}{\mathfrak C}
\newcommand{\sym}{\mathfrak S}

\newcommand{\eps}{\varepsilon}
\newcommand\indep{\perp\!\!\!\perp}
\newcommand{\CCCC}{\mathcal{C}}
\newcommand{\leqz}{\leq_Z}
\newcommand{\Ztp}{Z^{(t)}}
\newcommand{\wScbp}{\wcp_{Sb}}
\newcommand{\rews}{\wcp_{sb}}
\newcommand{\ovbmZ}{\overline{\bm Z}}
\newcommand{\ovconw}{\overline{\conti W}}
\newcommand{\wconz}{\widetilde{\conti Z}}
\newcommand{\ovconz}{\overline{\conti Z}}
\newcommand{\contzu}{\ovconz^{(u)}}
\newcommand{\zgz}{\Z_{>0}}
\newcommand{\intnui}{\lceil n \bm u_i \rceil}
\newcommand{\intnuj}{\lceil n \bm u_j \rceil}
\newcommand{\xraninf}{\xrightarrow[n\to\infty]}
\newcommand{\czord}{\mathcal C([0,1],\R^2)}
\newcommand{\obmx}{\overline{\bm X}}
\newcommand{\obmy}{\overline{\bm Y}}
\newcommand{\conti}[1]{{\bm{\mathcal #1}}}

\begin{document}

%%%%%%%%%%%%%%%%%%%%%%%%%%%%%%%%%%%%%%%%%%%%%%%%%%%%%%%%%%%%%%%%%%%
%%                                                               %%
%% No need for \maketitle.                                       %%
%%                                                               %%
%%%%%%%%%%%%%%%%%%%%%%%%%%%%%%%%%%%%%%%%%%%%%%%%%%%%%%%%%%%%%%%%%%%

%%%%%%%%%%%%%%%%%%%%%%%%%%%%%%%%%%%%%%%%%%%%%%%%%%%%%%%%%%%%%%%%%%%
%%                                                               %%
%% Please replace what follows by the body of your article       %%
%% (up to the bibliography):                                     %%
%%                                                               %%
%%%%%%%%%%%%%%%%%%%%%%%%%%%%%%%%%%%%%%%%%%%%%%%%%%%%%%%%%%%%%%%%%%%

\begin{figure}[htbp]
	\centering
	\includegraphics[width=0.245\linewidth]{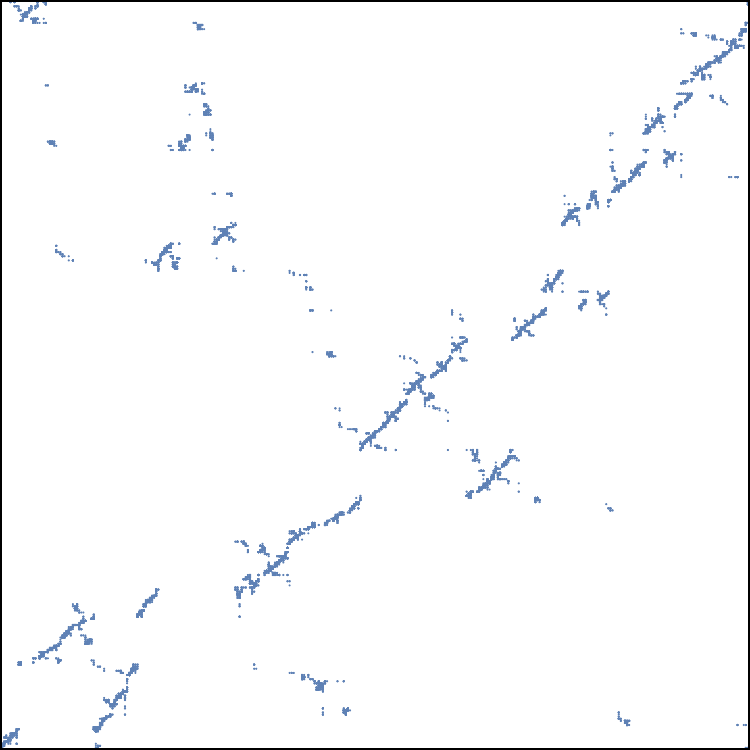}
	\includegraphics[width=0.245\linewidth]{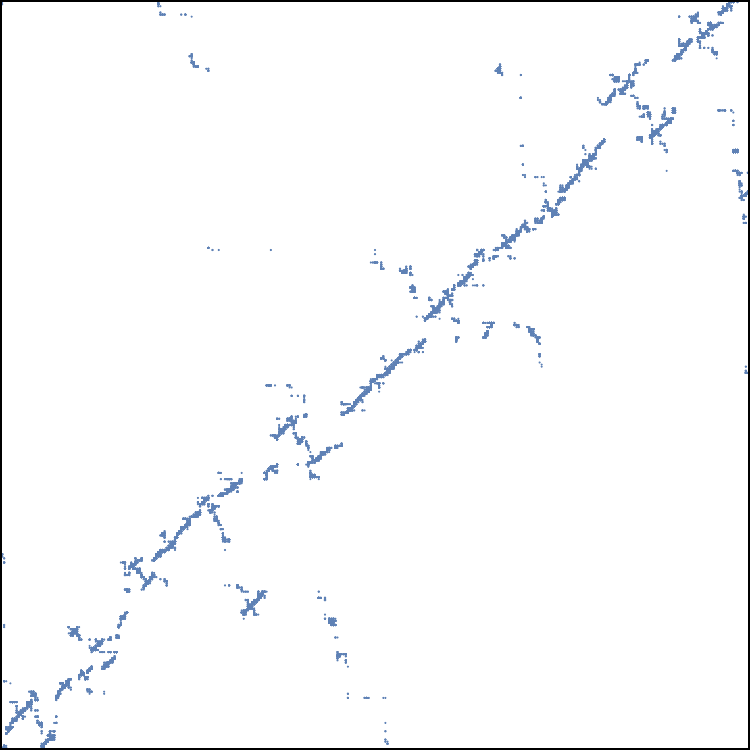}
	\includegraphics[width=0.245\linewidth]{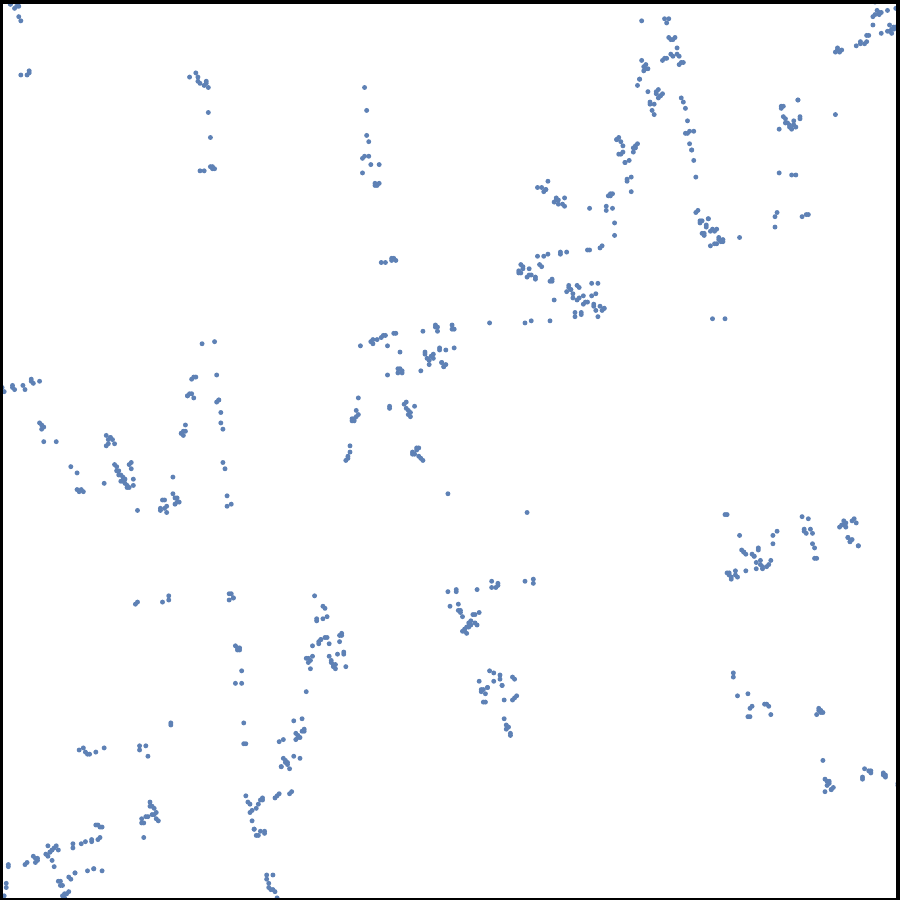}
	\includegraphics[width=0.245\linewidth]{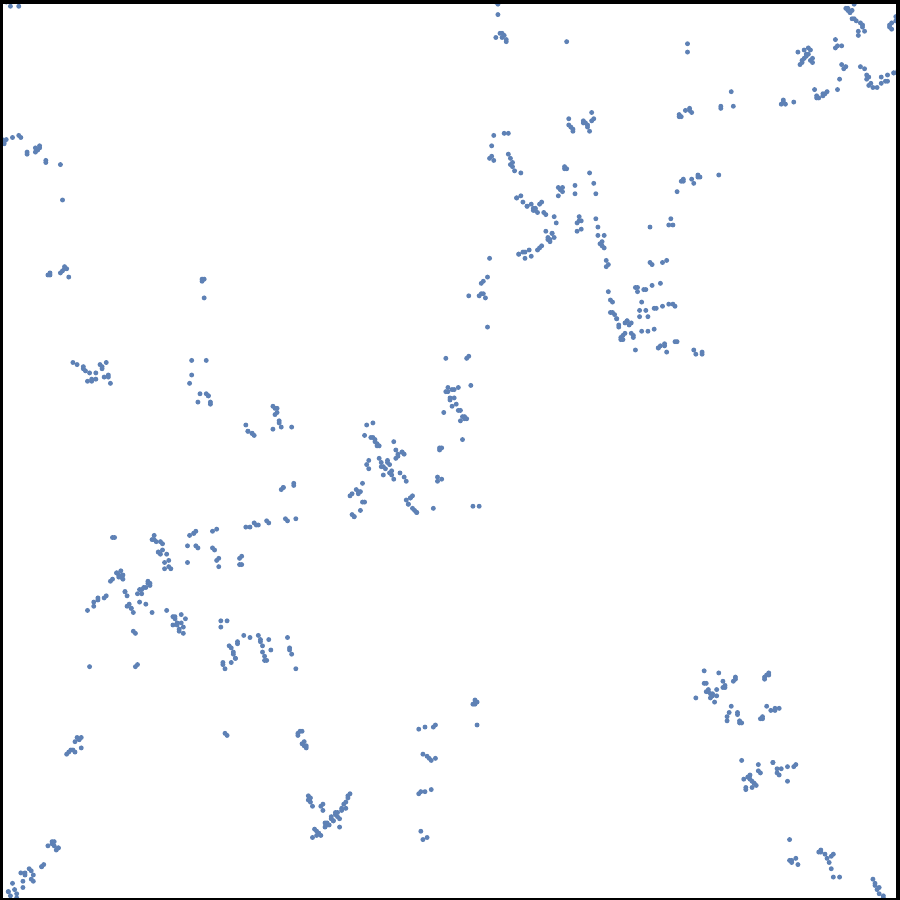}
	\caption{Simulations of the objects investigated in the present work. 
		\textbf{Left:} The diagrams of two uniform strong-Baxter permutations of size 17577 and 27574. 
		\textbf{Right:} The diagrams of two uniform semi-Baxter permutations of size 823 and 800. 
		\label{fig:Baxter_perm_3253}}
\end{figure}

%\newpage

%\tableofcontents

\section{Introduction}

Permutons are Borel probability measures on the unit square $[0,1]^2$ with uniform marginals and are the natural scaling limits of sequences of permutations. The \emph{skew Brownian permuton} is a new family of random permutons, recently introduced in \cite{borga2021skewperm}, describing the limit of various models of random constrained permutations, such as uniform separable permutations \cite{bassino2018brownian}, uniform permutations in substitution-closed classes \cite{bassino2017universal} and uniform Baxter permutations \cite{borga2020scaling}. As noticed in \cite{borga2021skewperm} (and also recalled later in this paper), for all these models, the permuton limit is a \emph{degenerate version} of the skew Brownian permuton.  

In the present paper, we show that there are natural models of random constrained permutations converging to the \emph{non-degenerate} skew Brownian permuton. To do that, we study the permuton limit of \emph{strong-Baxter permutations} and \emph{semi-Baxter permutations}.

This paper might be thought as a companion paper of \cite{borga2021skewperm}, but we tried to make the present work as much self-contained as we can. Nevertheless, we do not recall here all the motivations that justify the interest in studying the skew Brownian permuton and we refer the curious reader to \cite[Section 1.1]{borga2021skewperm}, where also a quick introduction to permuton convergence is provided. For a complete introduction to the theory of permutons, as well as bibliographic pointers, we refer to \cite[Section 2.1]{borga2021random}.

In the next sections, we first recall the construction of the skew Brownian permuton (see Section~\ref{sect:def_sbp}) and then we further motivate the interest in studying the permuton limit of \emph{strong-Baxter permutations} and \emph{semi-Baxter permutations} (see Section~\ref{sect:motivations}).

\subsection{The skew Brownian permuton}\label{sect:def_sbp}

We recall the construction of the skew Brownian permuton following \cite{borga2021skewperm}.

\medskip 

We start by recalling that a \emph{two-dimensional Brownian motion} \emph{of correlation} $\rho\in[-1,1]$, denoted\footnote{Here and throughout the paper we denote random quantities using \textbf{bold} characters.} $(\conti W_{\rho}(t))_{t\in \R_{\geq 0}}=(\conti X_{\rho}(t),\conti Y_{\rho}(t)))_{t\in \R_{\geq 0}}$, is a continuous two-dimensional Gaussian process such that the components $\conti X_{\rho}$ and $\conti Y_{\rho} $ are standard one-dimensional Brownian motions, and $\mathrm{Cov}(\conti X_{\rho}(t),\conti Y_{\rho}(s)) = \rho \cdot \min\{t,s\}$. We also recall that a  \emph{two-dimensional Brownian excursion} $(\conti E_{\rho}(t))_{t\in [0,1]}$ \emph{of correlation}\footnote{We remark that we do not include the case $\rho=-1$ in the definition of the two-dimensional Brownian excursion. Indeed, when $\rho=-1$ it is not meaningful to condition a two-dimensional Brownian motion of correlation $\rho=-1$ to stay in the non-negative quadrant.} $\rho\in(-1,1]$ \emph{in the non-negative quadrant} (here simply called a  \emph{two-dimensional Brownian excursion of correlation} $\rho$) is a 
%space
two-dimensional Brownian motion of correlation $\rho$ conditioned to stay in the non-negative quadrant $\R_{\geq 0}^2$ and to end at the origin, i.e.\ $\conti E_{\rho}(1)=(0,0)$. The latter 
%natural
process was formally constructed in various works (see for instance \cite[Section 3]{MR4010949} and \cite{MR4102254}).

\medskip

Let $(\conti E_{\rho}(t))_{t\in [0,1]}$ be a two-dimensional Brownian excursion of correlation $\rho\in(-1,1]$ and  let $q\in[0,1]$ be a further parameter.
Consider the solutions (which exist and are unique in the strong sense when $\rho\neq1$ thanks to \cite[Theorem 1.7]{borga2021skewperm}, and exist in the weak sense when $\rho=1$ thanks to \cite[Proposition 1.9]{borga2021skewperm}; see Remark~\ref{rem:technical_comments} for further details) of the following family of stochastic differential equations (SDEs) indexed by $u\in [0,1]$ and driven by $\conti E_{\rho} = (\conti X_{\rho},\conti Y_{\rho})$:
\begin{equation}\label{eq:flow_SDE_gen}
	\begin{cases}
		d\conti Z_{\rho,q}^{(u)}(t) 
		=
		\idf_{\{\conti Z_{\rho,q}^{(u)}(t)> 0\}} d\conti Y_{\rho}(t) - \idf_{\{\conti Z_{\rho,q}^{(u)}(t)< 0\}} 
		d \conti X_{\rho}(t)+(2q-1)\cdot d\conti L^{\conti Z_{\rho,q}^{(u)}}(t),& t\in(u,1),\\
		\conti Z_{\rho,q}^{(u)}(t)=0,&  t\in[0,u],
	\end{cases} 
\end{equation}
where $\conti L^{\conti Z_{\rho,q}^{(u)}}(t)$ is the symmetric local-time process at zero of $\conti Z_{\rho,q}^{(u)}$, i.e.\ 
$$\conti L^{\conti Z^{(u)}_{\rho,q}}(t)=\lim_{\varepsilon\to 0}\frac{1}{2\varepsilon}\int_0^t\idf_{\{\conti Z^{(u)}_{\rho,q}(s)\in[-\varepsilon,\varepsilon]\}}ds.$$
\begin{definition}\label{eq:defcoalproc}
	We call \emph{continuous coalescent-walk process driven by $(\conti E_{\rho},q)$} the collection of stochastic processes $\left\{\conti Z^{(u)}_{\rho,q}\right\}_{u\in[0,1]}$.
\end{definition}
We also consider the following stochastic process defined in terms of $\left\{\conti Z^{(u)}_{\rho,q}\right\}_{u\in[0,1]}$: 
\begin{equation}\label{eq:random_skew_function}
	\varphi_{\conti Z_{\rho,q}}(t)\coloneqq
	\Leb\left( \big\{x\in[0,t)|\conti Z_{\rho,q}^{(x)}(t)<0\big\} \cup \big\{x\in[t,1]|\conti Z_{\rho,q}^{(t)}(x)\geq0\big\} \right), \quad t\in[0,1].
\end{equation}

\begin{definition}\label{defn:ibfwebfowfpwenfiweop}
	Fix $\rho\in(-1,1]$ and $q\in[0,1]$. The \emph{skew Brownian permuton} of parameters $\rho, q$, denoted $\bm \mu_{\rho,q}$, is the push-forward of the Lebesgue measure on $[0,1]$ via the mapping $(\Id,\varphi_{\conti Z_{\rho,q}})$, that is
	\begin{equation}
		\bm \mu_{\rho,q}(\cdot)\coloneqq(\Id,\varphi_{\conti Z_{\rho,q}})_{*}\Leb (\cdot)= \Leb\left(\{t\in[0,1]|(t,\varphi_{\conti Z_{\rho,q}}(t))\in \cdot \,\}\right).
	\end{equation} 
\end{definition}

In \cite[Theorem 1.11]{borga2021skewperm} it was proved that the skew Brownian permuton $\bm \mu_{\rho,q}$ is well-defined for all $(\rho,q)\in(-1,1]\times[0,1]$.

\begin{remark}\label{rem:technical_comments}
	We point out that the continuous coalescent-walk process $\left\{\conti Z^{(u)}_{\rho,q}\right\}_{u\in[0,1]}$ in \eqref{eq:defcoalproc} is defined in the following sense: for almost every $\omega,$ $\conti Z^{(u)}_{\rho,q}$ is a solution for almost every $u\in[0,1]$.
	
	We also recall that in order to define the skew Brownian permuton $\bm \mu_{1,q}$ there are some additional technical difficulties due to the fact that the SDEs in \eqref{eq:flow_SDE_gen} when $\rho=1$ do not admit strong solutions. See \cite[Section 1.4.2]{borga2021skewperm} for further details. We  however note that such details are not needed in the present paper.
\end{remark}

\subsection{Motivations and statements of the main results}\label{sect:motivations}

We start by recalling that the Baxter permuton, introduced in \cite{borga2020scaling} as the permuton limit of uniform Baxter permutations, coincides with the skew Brownian permuton of parameters $\rho=-1/2$ and $q=1/2$. Therefore the corresponding SDEs in \eqref{eq:flow_SDE_gen} take the simplified form
\begin{equation}\label{eq:flow_SDE_erjiberoug}
	\begin{cases}
		d\conti Z^{(u)}(t) = \idf_{\{\conti Z^{(u)}(t)> 0\}} d\conti Y_{\rho}(t) - \idf_{\{\conti Z^{(u)}(t)< 0\}} d \conti X_{\rho}(t),& t\in(u,1),\\
		\conti Z^{(u)}(t)=0,&  t\in[0,u].
	\end{cases} 
\end{equation}

\medskip

We also recall that the biased Brownian separable permuton $\{\bm \mu^S_p\}_{p\in[0,1]}$, introduced in \cite{bassino2017universal} as the permuton limit of uniform permutations in substitution-closed classes (see also \cite{MR4115736}), is a one-parameter family of random permutons that can be constructed (see \cite{MR4079636}) from a one-dimensional Brownian excursion and a parameter $p\in[0,1]$. In \cite[Theorem 1.12]{borga2021skewperm}, we proved that for all $p\in[0,1]$, the biased Brownian separable permuton $\bm \mu^S_p$ has the same distribution as the skew Brownian permuton $\bm \mu_{1,1-p}$.

Note that when $\rho=1$ and $q=1-p$, the SDEs in \eqref{eq:flow_SDE_gen} take another simplified form. We denote by $(\bm e(t))_{t\in [0,1]}$ a one-dimensional Brownian excursion on $[0,1]$ and note that $\conti X_{\rho}(t)=\conti Y_{\rho}(t)=\bm e(t)$ when $\rho=1$. We also set $\sgn(x)\coloneqq\mathds{1}_{\{x>0\}}-\mathds{1}_{\{x\leq 0\}}$. Then the SDEs in \eqref{eq:flow_SDE_gen} rewrite as
\begin{equation}\label{eq:flow_SDE_gen_Tanaka}
	\begin{cases}
		d\conti Z_{1,1-p}^{(u)}(t) = \sgn\left(\conti Z_{1,1-p}^{(u)}(t)\right) d \bm e(t)+(1-2p)\cdot d\conti L^{\conti Z_{1,1-p}^{(u)}}(t),& t\in (u,1),\\
		\conti Z_{1,1-p}^{(u)}(t)=0,&  t\in [0,u].
	\end{cases} 
\end{equation}

\medskip

The unsatisfactory feature of the various instances of skew Brownian permutons above is that either $\rho=1$ or $q=1/2$, and in both cases the SDEs in \eqref{eq:flow_SDE_gen} take a simplified form: either the driving process is a one-dimensional Brownian excursion (see \eqref{eq:flow_SDE_gen_Tanaka}) or the local time term in \eqref{eq:flow_SDE_gen} cancels (see \eqref{eq:flow_SDE_erjiberoug}). As mentioned before, the main goal of this paper is to show that uniform \emph{strong-Baxter permutations} converge in the permuton sense to the skew Brownian permuton for a \emph{non-trivial} choice of the two parameters $\rho$ and $q$. This will give a new evidence that the skew Brownian permuton is a natural family of limiting random permutons for random constrained permutations\footnote{We remark that in \cite[Theorem 1.17]{borga2021skewperm} we showed that the skew Brownian permuton arise also from SLE-decorated Liouville quantum spheres. This result give an orthogonal motivation (compared to the ones coming from the study of random constrained permutations) for studying the skew Brownian permuton. See \cite{borga2022permutons,borga2022baxter} for recent progress in this direction.}. We start by defining strong-Baxter permutations.

\begin{definition}\label{defn:str-b}
	\emph{Strong-Baxter permutations} are permutations avoiding the three vincular patterns $2\underbracket[.5pt][1pt]{41}3$,  $3\underbracket[.5pt][1pt]{14}2$ and $3\underbracket[.5pt][1pt]{41}2$, i.e.\ permutations $\sigma$ such that there are no indices $1\leq i<j<k-1<n$ such that $\sigma(j+1) < \sigma(i) < \sigma(k) < \sigma(j)$ or $\sigma(j) < \sigma(k) < \sigma(i) < \sigma(j+1)$ or $\sigma(j+1) < \sigma(k) < \sigma(i) < \sigma(j)$.
\end{definition}

Strong-Baxter permutations were introduced in \cite{MR3882946} as a natural generalization of Baxter permutations. The main result in the current paper is the following one.

\begin{theorem}\label{thm:strong-baxter}
	Let $\bm \sigma_n$ be a uniform strong-Baxter permutation of size $n$. The following convergence in distribution in the permuton sense holds:
	\begin{equation}
		\bm \mu_{\bm \sigma_n}\xrightarrow{d}\bm \mu_{\rho,q},
	\end{equation}  
	where $\rho\approx-0.2151$ is the unique real solution of the polynomial 
	\begin{equation}
		1+6\rho+8\rho^2+8\rho^3,
	\end{equation} 
	and $q\approx0.3008$ is the unique real solution of the polynomial
	\begin{equation}
		-1+6q-11q^2+7q^3.
	\end{equation} 
\end{theorem}

As mentioned in the abstract, the techniques developed to prove Theorem~\ref{thm:strong-baxter} are quite robust (see Section~\ref{sect:techn} for further details). We apply them also to the case of \emph{semi-Baxter permutations}.

\begin{definition}\label{defn:semi-b}
	\emph{Semi-Baxter permutations} are permutations avoiding the vincular pattern $2\underbracket[.5pt][1pt]{41}3$, i.e.\ permutations $\sigma$ such that there are no indices $1\leq i<j<k-1<n$ such that $\sigma(j+1)<\sigma(i)<\sigma(k)<\sigma(j)$.
\end{definition}

We have the following second result.

\begin{theorem}\label{thm:semi-baxter}
	Let $\bm \sigma_n$ be a uniform semi-Baxter permutation of size $n$. The following convergence in distribution in the permuton sense holds:
	\begin{equation}
		\bm \mu_{\bm \sigma_n}\xrightarrow{d}\bm \mu_{\rho,q},
	\end{equation}  
	where
	\begin{equation}
		\rho=-\frac{1+\sqrt 5}{4}\approx -0.8090
		\quad\text{and}\quad
		q=\frac{1}{2}.
	\end{equation}  
\end{theorem}

In the next section we explain the techniques used to prove Theorems \ref{thm:strong-baxter} and \ref{thm:semi-baxter}. Doing that, we will also explain the organization of the remaining sections of the paper.

\subsection{Outline of the strategy for the proofs and generality of our techniques}\label{sect:techn}

Despite the strategies to prove Theorems \ref{thm:strong-baxter} and \ref{thm:semi-baxter} are quite involved, in this section we try to give a brief idea of the main steps in these proofs. In addition we highlight the substantial differences with the proofs in \cite{borga2020scaling}, where convergence in the permuton sense of Baxter permutations is established.

\bigskip

We start by briefly recalling how we proved in \cite{borga2020scaling} that the scaling limit of Baxter permutations is the skew Brownian permuton $\bm \mu_{-1/2,1/2}$. The starting point consisted in using two bijections already available in the literature, one between Baxter permutations and bipolar orientations \cite{MR2734180} and a second one between bipolar orientations and a family of two-dimensional random walks in the non-negative quadrant, called \emph{tandem walks} \cite{kenyon2015bipolar}. Composing these two bijections, we obtained a bijection between Baxter permutations and tandem walks.

To prove the convergence towards the skew Brownian permuton $\bm \mu_{-1/2,1/2}$, we first studied the scaling limit of the random discrete coalescent-walk process associated with a random tandem walk and we showed that it converges in law to the continuous random coalescent-walk process encoded by the SDEs in \eqref{eq:flow_SDE_erjiberoug}. This discrete coalescent-walk process is defined from the tandem walk in a manner similar to \eqref{eq:flow_SDE_erjiberoug}, but in the discrete setting (see Section~\ref{sect:coalproc} for a proper definition of discrete coalescent-walk processes). For the moment, the reader can simply think of this discrete coalescent-walk process as a collection of discrete one-dimensional walks driven by the increments of the corresponding tandem walk.

The last step in the proof was to transfer this coalescent-walk process convergence result to Baxter permutations using the bijection with tandem walks mentioned above.

\bigskip

We can now explain the techniques used in this paper to prove Theorems \ref{thm:strong-baxter} and \ref{thm:semi-baxter}. The first main difference between Baxter permutations and semi/strong-Baxter permutations is that bijections between the latter families of permutations and some families of walks are not available in the literature (as it was the case for Baxter permutations). Our first contribution in this paper is the following one: in Sections~\ref{sect:defn_discrete}~and~\ref{sect:disc_obj_strong} we develop a general technique to encode families of permutations enumerated using generating trees with two-dimensional labels (see Section~\ref{sect:gen_tree_for_perms} for an introduction to generating trees) first by some families of two-dimensional random walks\footnote{This technique builds on another recent work of the author, see \cite{borga2020asymptotic}.} and then by some specific discrete coalescent-walk processes. Our main combinatorial results are proved in Theorem~\ref{thm:The_diagram_commutes_strong} -- for strong-Baxter permutations -- and Theorem~\ref{thm:The_diagram_commutes} -- for semi-Baxter permutations.

\medskip

Building on these new combinatorial results, in order to prove Theorem~\ref{thm:strong-baxter} (resp.\ Theorem~\ref{thm:semi-baxter}), we will show in Section~\ref{sect:prob_part} (resp.\ Section~\ref{sect:semi_baxter_prob_res}) that the discrete coalescent-walk processes for strong-Baxter permutations (resp.\ semi-Baxter permutations) converge to the continuous coalescent-walk process introduced in \eqref{eq:defcoalproc} for the specific choices of the parameters stated in Theorem~\ref{thm:strong-baxter} (resp.\ Theorem~\ref{thm:semi-baxter}). These coalescent-walk process convergence results rely on proving the following two key results.
\begin{itemize}
	\item The first step is to show that a uniform two-dimensional walk encoding a uniform strong-Baxter permutation (resp.\ semi-Baxter permutation) converges to a two-dimensional Brownian excursion $\conti E_{\rho}$ in the non-negative quadrant with a specific correlation parameter $\rho$. This is proved in Proposition~\ref{prop:scaling_strong_walk} (resp.\ Proposition~\ref{prop:weiugfweougfwoeu}).
	\item The second step is to show that a one-dimensional walk in the discrete coalescent-walk process associated with a uniform strong-Baxter permutation (resp.\ semi-Baxter) converges to the process $\conti Z_{\rho,q}^{(u)}(t)$ defined by the SDEs in \eqref{eq:flow_SDE_gen} driven by the two-dimensional Brownian excursion $\conti E_{\rho}$ obtained in the previous step and with parameter $q$ as in Theorem~\ref{thm:strong-baxter} (resp.\ Theorem~\ref{thm:semi-baxter}). This is proved in Theorem~\ref{thm:fvwuofbgqfipqhfqfpq_strongb} (resp. Theorem~\ref{thm:fvwuofbgqfipqhfqfpq_semib}).
\end{itemize}

We now explain why in both of the two key steps above we need to develop some new strategies of proof (compared to the one used for Baxter permutations in \cite{borga2020scaling}).

\begin{itemize}
	\item The convergence of uniform tandem walks (the walks encoding Baxter permutations) towards the two-dimensional Brownian excursion $\conti E_{-1/2}$ was obtained in \cite{borga2020scaling} as an application of Donsker's theorem thanks to a remarkable (but very peculiar) property of uniform tandem walks (see \cite[Proposition 3.2]{borga2020scaling}). Such a property is not valid for the walks encoding semi/strong-Baxter permutations.
	Therefore the convergence of uniform two-dimensional walks encoding uniform semi/strong-Baxter permutations is less straightforward: we will need to use precise bounds on some probabilities related with random walks in cones (this bounds are obtained in Section~\ref{sect:prob_walk_cones} building on the results of Denisov and Wachtel \cite{MR3342657} and are applied in the proofs of  Propositions \ref{prop:scaling_strong_walk} and \ref{prop:weiugfweougfwoeu}).

	\item The convergence of a one-dimensional walk in the discrete coalescent-walk process associated with a uniform Baxter permutation towards the process $\conti Z^{(u)}(t)$ defined by the SDEs in \eqref{eq:flow_SDE_erjiberoug} was obtained in \cite{borga2020scaling} using standard techniques: indeed such walk turns out to be (surprisingly) a simple random walk (as proved in \cite[Proposition 3.3]{borga2020scaling}).
	
	A one-dimensional walk in the discrete coalescent-walk process associated with a uniform semi-Baxter or strong-Baxter permutation is \emph{not} a simple random walk. Therefore proving its convergence towards the process $\conti Z_{\rho,q}^{(u)}(t)$ defined by the SDEs in \eqref{eq:flow_SDE_gen} is more challenging and will force us to develop some new tools. For instance, in Proposition~\ref{prop:skew_part_strong}, building on some results of Ngo and Peign\'e \cite{MR4105264}, we prove that an unconditioned version of a one-dimensional walk in the discrete coalescent-walk process associated with a uniform strong-Baxter permutation converges towards a skew Brownian motion of parameter $q$, with $q$ as in Theorems \ref{thm:strong-baxter}. This will be the key-step to then prove convergence towards $\conti Z_{\rho,q}^{(u)}(t)$ in Theorem~\ref{thm:fvwuofbgqfipqhfqfpq_strongb}.
	
	Also the case of semi-Baxter permutations displays some remarkable phenomena that are too complicated to explain in this introduction. Therefore, for the moment, we just refer the reader to Remark~\ref{rem:traj}.
\end{itemize}

We conclude this section with a quick comment on the generality of the techniques used in this paper and a quick remark on the possible implications of our results on decorated planar maps. 

As we explained above, the starting point to obtain our results in Theorems \ref{thm:strong-baxter} and \ref{thm:semi-baxter} is the possibility to encode permutations with two-dimensional walks. This is done using generating trees with two-dimensional labels. There are several other families of permutations that are encoded by generating trees with two-dimensional labels, such as \emph{plane permutations} (see \cite[Proposition 5]{MR3882946}) or some particular subfamilies of Baxter permutations (see for instance \cite[Theorem 16]{MR3992288} and \cite{MR2028288}). We believe that the techniques developed in the present paper allow to prove convergence towards the skew Brownian permuton also for these other families of permutations.

We finally remark that this is the first paper where the parameter $q$ in Theorems \ref{thm:strong-baxter} and \ref{thm:semi-baxter} has been determined without using \emph{symmetry arguments}. We plan to use this new method also in a future work on decorated planar maps.
Indeed, the parameter $q$ of the skew Brownian permuton (implicitly) determines the angle between two SLE-curves decorating a Liouville quantum sphere, as explained in \cite[Theorem 1.17, Remark 1.18]{borga2021skewperm}. In the planar map literature there are only two models of decorated planar maps where convergence to SLE-decorated Liouville quantum spheres is proved: random bipolar orientations converge to a $\sqrt{4/3}$-Liouville quantum sphere decorated with two SLE$_{12}$-curves of angle $\pi/2$ (\cite{gwynne2016joint,borga2020scaling}) and random Schnyder woods converge to a $1$-Liouville quantum sphere decorated with two SLE$_{16}$-curves of angle $2\pi/3$ (\cite{li2017schnyder}). In both cases the angle was determined by symmetry arguments. We believe\footnote{See also the last open problem in \cite[Section 1.6]{borga2020scaling} for further explanations.} that the method developed in the present paper can be applied to some models of decorated planar maps, such as bipolar posets (see \cite{fusy2021enumeration}), in order to determine limiting angles that cannot be identified using symmetry arguments.

\subsection{Notation and basic definitions}

We denote by $\sym_n$ the set of permutations of size $n$ and  by $\sym$ the set of permutations of finite size.
Given a permutation $\sigma\in \sym_n$, its \emph{diagram} is the sets of points of
the Cartesian plane at coordinates $(i, \sigma(i))_{i\in[n]}$.

If $x_1,\dots ,x_n$ is a sequence of distinct numbers, let $\resc(x_1,\dots ,x_n)$ be the unique permutation $\pi$ in $\sym_n$ that is in the same relative order as $x_1,\dots ,x_n,$ \emph{i.e.}, $\pi(i)<\pi(j)$ if and only if $x_i<x_j.$
Given a permutation $\sigma\in\sym_n$ and a subset of indices $ I \subseteq [ n ] $, let $\indperm_I(\sigma)$ be the permutation induced by $(\sigma(i))_{i\in I},$ namely, $\indperm_I(\sigma)\coloneqq\resc\left((\sigma(i))_{i\in I}\right).$
For example, if $\sigma=57832164$ and $I=\{2,4,7\}$ then $\indperm_{\left\{2,4,7\right\}}(57832164)=\text{std}(736)=312$. If  $\indperm_I(\sigma)=\pi$, we then call the set of indices $I$ an \emph{occurrence} of the pattern $\pi$ in $\sigma$.

\medskip

Given a random variable $\bm X$ and an event $A$ we write $(\bm X;A)$ for the random variable $\bm X\mathds{1}_A$.

\section{Permutations, generating trees, two-dimensional walks, and coalescent-walk processes}\label{sect:defn_discrete}

The goal of this section is to introduce the definitions of generating tree and coalescent-walk process. Then we also recall the construction of two mappings, the first one defined between permutations and $d$-dimensional walks (introduced in \cite{borga2020asymptotic}) and the second one defined between coalescent-walk processes and permutations (introduced in \cite{borga2020scaling}).

\subsection{Generating trees for permutations}\label{sect:gen_tree_for_perms}

\subsubsection{Definitions}

We quickly summarize here all the definitions that we need about generating trees. For a more detailed and complete description we refer to \cite[Section 1.5]{borga2020asymptotic}.
We start with the following preliminary construction and some definitions.

\begin{definition}
	For a permutation $\sigma\in\sym^n$ and an integer $m\in [n+1],$ let $\sigma^{*m}$ denote the permutation
	\begin{equation}
		\sigma^{*m}\coloneqq\text{std}(\sigma(1),\dots,\sigma(n),m-1/2).
	\end{equation}
	In words, $\sigma^{*m}$ is obtained from $\sigma$ by \emph{appending a new final value} equal to $m$ and shifting by $+1$ all the other values larger than or equal to $m.$

	A family of permutations $\CCCC$ \emph{grows on the right} if for all permutations $\sigma\in\sym$ such that $\sigma^{*m}\in\CCCC$ for some $m\in [|\sigma|+1]$, it holds that $\sigma\in\CCCC.$
\end{definition}

\begin{definition}
	The \emph{generating tree} associated with a family of permutations $\CCCC$ that grows on the right is the infinite rooted tree whose vertices correspond to the permutations of $\CCCC$ (each appearing exactly once in the tree) and such that permutations of size $n$ are at level $n$.
	The $c(\sigma)$ children of a vertex corresponding to a permutation $\sigma\in\CCCC$ are the $c(\sigma)$ vertices $v_1(\sigma),\dots,v_{c(\sigma)}(\sigma)$ corresponding to the permutations obtained by appending a new final value to $\sigma$.
\end{definition}

\begin{definition}
	Consider a set $S$ and an $S$-valued statistic $\stc(\cdot)$ on $\CCCC$ whose value determines 
	\begin{itemize}
		\item the number of children in the generating tree of $\CCCC$ (in the sense that if $\sigma,\sigma'\in\CCCC$ and $\stc(\sigma)=\stc(\sigma')=k \in S$ then $c(\sigma)=c(\sigma')=:h(k)$);
		\item the values of such statistic for these children (in the sense that if $\sigma,\sigma'\in\CCCC$ and $\stc(\sigma)=\stc(\sigma')=k \in S$ then $\stc(v_i(\sigma))=\stc(v_i(\sigma'))=:e_{i}(k)$, for all $i\in[c(\sigma)]=[h(k)]$). 
	\end{itemize}
	Then we label each vertex of the generating tree with the value of the statistic taken by the permutation corresponding to the vertex. The associated \emph{succession rule} is determined by the label of the root $\lambda$ and, for every label $k$, by the labels $e_1(k),\dots,e_{h(k)}(k)$ of the $h(k)$ children of a vertex labeled by $k$. In full generality, the associated succession rule is represented by the formula
	\begin{equation}\label{eq:laruletadesuccession}
		\begin{cases} 
			\text{Root label}: (\lambda) \\
			(k)\to (e_1{\scriptstyle(k)}),\dots,(e_{h(k)}{\scriptstyle(k)}).
		\end{cases}
	\end{equation}
	We denote by $\mathcal{L}$ the set of all labels appearing in the generating tree and for all $k\in\mathcal{L},$ we denote by $\CL(k)$ the multiset of labels of the children $\{e_1(k),\dots,e_{h(k)}(k)\}$.
\end{definition}

An example of generating tree, in the case of the class of 123-avoiding\footnote{A permutation $\sigma$ avoids $\pi$ if there are no occurrences of the pattern $\pi$ in $\sigma$.} permutations, is given in Figure~\ref{fig:gen_tree_321}. Considering the $\zgz$-valued statistic that counts the number of possible insertions of a new point, then the succession rule is 
\begin{equation}
	\begin{cases} 
		\text{Root label}: (2) \\
		(k)\to (2),(3)\dots,(k+1).
	\end{cases}
\end{equation}

\begin{figure}[htbp]
	\begin{minipage}[c]{0.55\textwidth}
		\centering
		\includegraphics[scale=0.35]{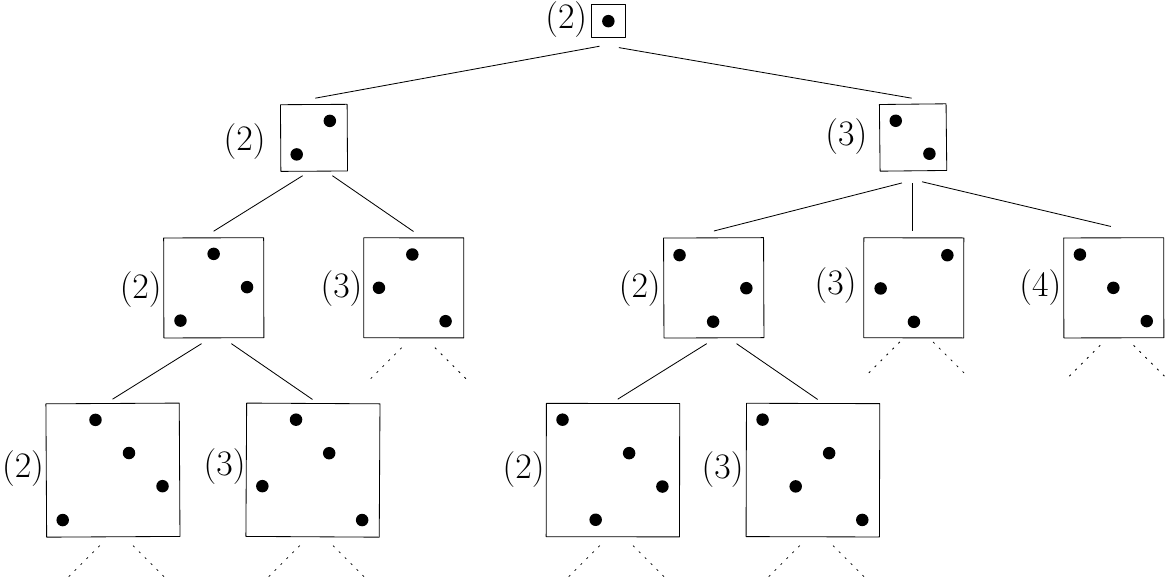}
	\end{minipage}
	\begin{minipage}[c]{0.45\textwidth}
		\caption{The generating tree for $123$-avoiding permutations. The children of a permutation are obtained by adding a new point on the right of the diagram in such a way that it does not create a pattern 123. 
			For each diagram we report (on the left) the label given by the statistic that counts the number of possible insertions of a new point. Note that every permutation with label $(k)$ has $k$ children with labels $(2),(3),\dots,(k+1)$. \label{fig:gen_tree_321}}
	\end{minipage}
\end{figure}

\subsubsection{A bijection between permutations encoded by generating trees and multi-dimensional walks}
\label{sect:bij_walk_perm}

We start by assuming that the children labels $e_1(k),\dots,e_{h(k)}(k)$ appearing in the succession rule in \eqref{eq:laruletadesuccession} are distinct\footnote{A general bijection can be constructed without assuming that the children labels $e_1(k),\dots,e_{h(k)}(k)$ are distinct, see \cite[Section 1.5.2]{borga2020asymptotic}. In this paper we only need the present simplified version.}. 

There is a simple bijection between a family of permutations encoded by its generating tree and the set of paths in the tree from the root. We can associate to the endpoint of each path the permutation corresponding to that vertex. Encoding each path in the tree with the list of labels appearing on the path, every permutation of size $n$ is bijectively encoded by a sequence of $n$ labels $(k_{1} = \lambda , k_{2} , \dots , k_{n} )$, where $k_{i}$ tracks the value of the label, such that every pair of two consecutive labels $( k_{i} , k_{i+1}),$ $1\leq i\leq n-1,$ is \emph{consistent with the succession rule} in \eqref{eq:laruletadesuccession}, i.e.\ for all $1\leq i\leq n-1,$ there exist $j= j({i}) \in [ 1 , h(k_{i}) ]$ such that $k_{i+1}=e_{j}(k_{i})$. We denote by $\pw$ this bijection between permutations in the generating tree and sequences of labels. 

Note that whenever	$\mathcal{L}$ is a subset of $\Z^n$ for some $n\in \zgz$ (this will be the case of all the permutation families considered in the article), then these sequences of labels can be naturally interpreted as $n$-dimensional walks. The bijection $\pw$ is then a map from permutations (P) to walks (W).

\subsection{Discrete coalescent-walk processes and permutations}\label{sect:coalproc}

\emph{Coalescent-walk processes} were introduced in \cite{borga2020scaling} (see in particular Section 2.4) to study the scaling limit of Baxter permutations. We now recall their definition and their relations with permutations. Then in Section~\ref{sect:coal_proc_for_strong_baxt} we show that there exists a natural coalescent-walk process encoding strong-Baxter permutations (and in Section~\ref{sect:coal_proc_for_semi_baxt} we will do the same for semi-Baxter permutations).

\begin{definition}
	Consider a (finite or infinite) interval $I$ of $\Z$. A \emph{coalescent-walk process} on $I$ is a family  of one-dimensional walks $\{(\Ztp_s)_{s\geq t, s\in I}\}_{t\in I}$ such that:
	\begin{itemize}
		\item $\Ztp_t=0$, for every $t\in I$;
		\item for $t'\geq t \in I$, if $\Ztp_k\geq Z^{(t')}_k$ (resp.\ $\Ztp_k\leq Z^{(t')}_k$) then $\Ztp_{k'}\geq Z^{(t')}_{k'}$ (resp.\ $\Ztp_{k'}\leq Z^{(t')}_{k'}$) for every $k'\geq k$.
	\end{itemize}
	The set of coalescent-walk processes on some interval $I$ is denoted by $\Coals(I)$.
\end{definition}

We highlight the following consequence of the definition above: if $\Ztp_k= Z^{(t')}_k$ for some time $k$, then $\Ztp_{k'}=Z^{(t')}_{k'}$ for all $k'\geq k$.
If this is the case, we say that the walks $\Ztp$ and $Z^{(t')}$ \emph{coalesce} and we call \emph{coalescent point} of $\Ztp$ and $Z^{(t')}$ the point $(s,\Ztp_s)$ such that  $s = \min \{ k\geq \max\{ t , t' \}| \Ztp_k=Z^{(t')}_k \}$.

\bigskip

For a coalescent-walk process $Z = \{\Ztp\}_{t\in I} \in \Coals(I)$ defined on a (finite or infinite) interval $I$, we define the following binary relation $\leqz$ on $I$:

\begin{equation}
	\begin{cases}
		i \leqz i,\\
		i\leqz j,&\text{ if } i<j \text{ and } Z^{(i)}_j\leq0,\\
		j\leqz i,&\text{ if } i<j \text{ and } Z^{(i)}_j>0.
	\end{cases}
\end{equation}

In \cite[Proposition 2.9]{borga2020scaling} it was shown that $\leqz$ defines a total order on $I$. This total order allows to associate a permutation with a coalescent-walk process on the interval $[n]$.

\begin{definition}\label{defn:coalandpermrel}
	Let $n\in\Z_{\geq 0}$ and $Z = \{\Ztp\}_{i\in [n]} \in \Coals([n])$ be a coalescent-walk process on $[n]$. We denote by $\cpp(Z)$ the only permutation $\sigma \in \sym_n$ such that for all $1\leq i, j\leq n$, 
	\begin{equation}
		\sigma(i)\leq\sigma(j) \iff i \leqz j.
	\end{equation}
	As a consequence, for all $j\in[n]$,
	$$\sigma(j)=\#\{i\in[n]|i\leqz j\}=\#\{i\in [j]:Z^{(i)}_j\leq 0\}+\#\{i\in [j+1,n]:Z^{(j)}_i> 0\}.$$
\end{definition}

In \cite[Proposition 2.11]{borga2020scaling} it was shown that pattern extraction in the permutation $\cpp(Z)$ depends only on a finite number of walks. This is a key tool to prove permuton convergence results.

\begin{proposition}[Proposition 2.11 in \cite{borga2020scaling}]
	\label{prop:patternextraction}
	Let $Z=\{\Ztp\}_{t\in[n]}$ be a coalescent-walk process and $\sigma=\cpp(Z)$ the corresponding permutation of size $n$. We fix a set of indexes $I=\{i_1<\dots<i_k\}\subseteq[n]$. Then $\indperm_I(\sigma)=\pi$ if and only if the following condition is satisfied: 
	for every $1\leq \ell< s \leq k,$
	\begin{equation}
		Z^{(i_\ell)}_{i_s} \geq 0   
		\iff  
		\pi(s)<\pi(\ell).
	\end{equation}
\end{proposition}

\section{Discrete objects associated with strong-Baxter permutations}\label{sect:disc_obj_strong}

Recall that strong-Baxter permutations were defined in Definition~\ref{defn:str-b}. 
These permutations have been enumerated using $\Z^2$-labeled generating trees and therefore they can be bijectively encoded by a specific family of two-dimensional walks.

The goal of this section is to specify the maps $\pw$ and $\cpp$ introduced in Section~\ref{sect:defn_discrete} in the specific case of strong-Baxter permutations ($\mathcal{S}_{Sb}$) and the corresponding families of walks $(\mathcal{W}_{Sb})$ and coalescent-walk processes ($\CCCC_{Sb}$). While doing that, we further introduce a third map $\wScbp$ between walks and coalescent-walk processes. Thus, we are going to properly define the following diagram:
\begin{equation}\label{eq:diagrm_strong_baxter}
	\begin{tikzcd}
		\mathcal{S}_{Sb} \arrow{r}{\pw}  & \mathcal{W}_{Sb} \arrow{d}{\wScbp} \\
		& \CCCC_{Sb} \arrow{ul}{\cpp}
	\end{tikzcd}.
\end{equation}
Actually we will show that this is a commutative diagram of bijections (see Theorem~\ref{thm:The_diagram_commutes_strong}).

\subsection{Succession rule for strong-Baxter permutations and the corresponding family of two-dimensional walks}\label{sect:strong_bac_obj}

We introduce some terminology.

\begin{definition}\label{defn:appending}
	Given a family of permutations $\CCCC,$ a permutation $\sigma\in\CCCC^n$ and an integer $m\in [n+1],$ we say that $m$ is an \emph{active site} of $\sigma$ if $\sigma^{*m}\in\CCCC.$ We denote by $\text{AS}(\sigma)$ the set of active sites of $\sigma$.
\end{definition}

We will adopt the following useful convention. Given a strong-Baxter permutation $\pi\in \mathcal{S}_{Sb}^n$ with $x+1$ active sites smaller than or equal to $\pi(n)$ and $y+1$ active sites greater than $\pi(n)$, we write
$$\text{AS}(\pi)=\{s_{-x}<\dots<s_0\}\cup\{s_{1}<\dots<s_{y+1}\},$$
where the first set corresponds to the $x+1$ active sites smaller than or equal to $\pi(n)$ and the second set corresponds to the $y+1$ active sites greater than $\pi(n)$.

\medskip

In \cite[Proposition 23]{MR3882946} it was shown that the generating tree for semi-Baxter permutations can be defined by the following succession rule\footnote{Note that the succession rule in our paper is obtained from the succession rule in \cite[Proposition 23]{MR3882946} by shifting all the labels by a factor $(-1,-1)$. This choice is more convenient for our purposes.}:
\begin{equation}\label{eq:efobwe9ubf0wif}
	\begin{cases} 
		\text{Root label}: (0,0) \\
		(h,k)\to\begin{cases}
			(0,k),(1,k),\dots,(h-1,k),\\
			(h,k+1),\\ 
			(h+1,0),(h+1,1)\dots,(h+1,k),
		\end{cases}
		\quad \text{for all }h,k\geq 0.
	\end{cases}
\end{equation}
In particular, the $\Z^2_{\geq 0}$-valued statistic that determines this succession rule is defined on every permutation $\sigma\in\mathcal{S}_{Sb}$ by
$$\Big(\# \{m\in\text{AS}(\sigma)|m\leq\sigma(n)\}-1\;,\;\# \{m\in\text{AS}(\sigma)|m>\sigma(n)\}-1\Big).$$

Using the strategy described in Section~\ref{sect:bij_walk_perm}, we can define a bijection $\pw$ between strong-Baxter permutations and the set of two-dimensional walks in the non-negative quadrant, starting at $(0,0)$, with increments in
\begin{equation}\label{eq:strong_bax_increm}
	I_{Sb}\coloneqq\{(-i,0): i\geq 1\}\cup\{(0,1)\}\cup\{(1,-i): i\geq 0\}.
\end{equation}
Note that the latter family is determined by the succession rule in \eqref{eq:efobwe9ubf0wif}.
We denote with $\mathcal{W}_{Sb}$ the set of two-dimensional walks in the non-negative quadrant, starting at $(0,0)$, with increment in $I_{Sb}$. We also denote with $\mathcal{W}_{Sb}^n$ the subset of walks in $\mathcal{W}_{Sb}$ of size $n$.

\medskip

We now want to investigate the relations between the increments of a walk $W\in\mathcal{W}^n_{Sb}$ and the active sites of the sequence of permutations:

$$\Big(\pw^{-1}((W_i)_{i\in[1]}),\pw^{-1}((W_i)_{i\in[2]}),\dots,\pw^{-1}((W_i)_{i\in[n]})\Big).$$

First of all, it holds that $\pw^{-1}((W_i)_{i\in[1]})$ is the unique permutation of size 1 and its active sites are 1 and 2. Now assume that for some $m<n$, $W_m=(x,y)\in\Z^2_{\geq 0}$ and  $\pw^{-1}((W_i)_{i\in[m]})=\pi$. By definition, $\pi$ has $x+1$ active sites smaller or equal to $\pi(n)$ and $y+1$ active sites greater than $\pi(n)$, i.e.,
$$\text{AS}(\pi)=\{s_{-x}<\dots<s_{0}\}\cup\{s_{1}<\dots<s_{y+1}\}.$$
We now distinguish three cases (for a proof of the following results see the proof of \cite[Proposition 23]{MR3882946}, compare also with Figure~\ref{fig:schema_perm_strong}):
\begin{itemize}
	\item \textbf{Case 1:} $W_{m+1}-W_m=(1,-i)$ for some $i\in \{0\}\cup[y]$. 
	
	\noindent In this case $\pw^{-1}((W_i)_{i\in[m+1]})=\pi^{*s_{i+1}}$ and the active sites of $\pi^{*s_{i+1}}$ are 
	$$\{s_{-x}<\dots<s_{0}<s_{i+1}\}\cup\{s_{i+1}+1,s_{i+2}+1<\dots<s_{y+1}+1\}.$$

	\item \textbf{Case 2:} $W_{m+1}-W_m=(0,1)$. 
	
	\noindent In this case $\pw^{-1}((W_i)_{i\in[m+1]})=\pi^{*s_{0}}$ and the active sites of $\pi^{*s_{0}}$ are 
	$$\{s_{-x}<\dots<s_{0}\}\cup\{s_{0}+1<s_{1}+1<\dots<s_{y+1}+1\}.$$
	
	\item \textbf{Case 3:} $W_{m+1}-W_m=(-i,0)$ for some $i\in [x]$. 
	
	\noindent In this case $\pw^{-1}((W_i)_{i\in[m+1]})=\pi^{*s_{-i}}$ and the active sites of $\pi^{*s_{-i}}$ are 
	$$\{s_{-x}<\dots<s_{-i}\}\cup\{s_{1}+1<\dots<s_{y+1}+1\}.$$
\end{itemize} 

An example of the various constructions above is given in Figure~\ref{fig:schema_perm_strong}.
We also point out the following important consequence of the discussion above.

\begin{remark}\label{obs:final_point_strong_bax}
	Let $\pi$ be a strong-Baxter permutation of size $n$ and assume that 
	$$\text{AS}(\pi)=\{s_{-x}<\dots<s_0\}\cup\{s_{1}<\dots<s_{y+1}\}.$$
	Then $\pi(n)=s_0$.
\end{remark}

\begin{figure}[htbp]
	\centering
	\includegraphics[width=\linewidth]{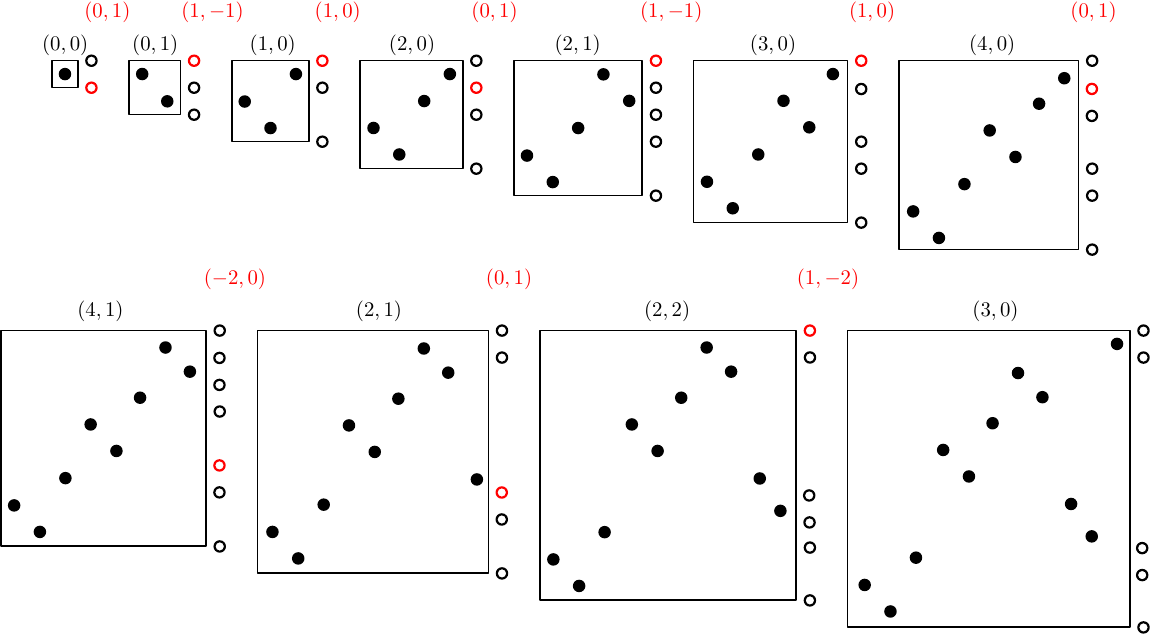}\\
	\caption{We consider the walk $W\in\mathcal{W}^{11}_{Sb}$ given by the eleven black $\Z^2_{\geq 0}$-labels in the picture. The increments $W_{m+1}-W_{m}$ of the walk $W$ are written in red between two consecutive diagrams. For each black label $W_m$ we draw the diagram of the corresponding permutation $\pw^{-1}((W_i)_{i\in[m]})$. On the right-hand side of each diagram we draw with small circles the active sites of the permutation and we highlight in red the site that will be activated by the corresponding red increment $W_{m+1}-W_m$.  \label{fig:schema_perm_strong}}
\end{figure}

\subsection{A coalescent-walk process for strong-Baxter permutations}\label{sect:coal_proc_for_strong_baxt}

We now define a family of coalescent-walk processes driven by a set of two-dimensional walks that contains $\mathcal{W}_{Sb}$. 
We consider a (finite or infinite) interval $I$ of $\Z$. 
Let $\mathfrak{W}_{Sb}(I)$ denote the set of two-dimensional walks indexed by $I$, with increments in $I_{Sb}$, and considered up to an additive constant.

\begin{definition}\label{eq:icre_coal_strong}
	Let $W\in\mathfrak{W}_{Sb}(I)$. The \emph{coalescent-walk process associated with} $W$
	is the family of walks $\wScbp(W) = \{\Ztp\}_{t\in I}$, defined for $t\in I$ by $\Ztp_t=0,$ and for all $\ell\geq t$ such that $\ell+1 \in I$,
	\begin{itemize}
		\item \textbf{Case 1:} $W_{\ell+1}-W_\ell=(1,-i)$ for some $i\geq 0$, then
		
		\begin{equation}
			\Ztp_{\ell+1}=
			\begin{cases}
				\Ztp_{\ell}-i, &\quad\text{if}\quad \Ztp_{\ell}-i>0 \text{ (and so } \Ztp_{\ell}>0),\\
				\Ztp_{\ell}-1, &\quad\text{if}\quad \Ztp_{\ell}\leq 0,\\
				-1,&\quad\text{otherwise}.
			\end{cases}
		\end{equation} 
		
		\item \textbf{Case 2:} $W_{\ell+1}-W_\ell=(0,1)$, then
		
		\begin{equation}
			\Ztp_{\ell+1}=
			\begin{cases}
				\Ztp_{\ell}+1, &\quad\text{if}\quad \Ztp_{\ell}\geq 0,\\
				\Ztp_{\ell},&\quad\text{otherwise}.
			\end{cases}
		\end{equation} 
		
		\item \textbf{Case 3:} $W_{\ell+1}-W_\ell=(-i,0)$ for some $i\geq 1$, then 
		
		\begin{equation}
			\Ztp_{\ell+1}=
			\begin{cases}
				\Ztp_{\ell},&\quad\text{if}\quad \Ztp_{\ell}\geq 0,\\
				\Ztp_{\ell}+i,&\quad\text{if}\quad \Ztp_{\ell}<0\text{ and }\Ztp_{\ell}+i\leq 0,\\
				0, &\quad\text{otherwise}.
			\end{cases}
		\end{equation} 
	\end{itemize} 
\end{definition}

Note that $\wScbp$ is a mapping form $\mathfrak{W}_{Sb}(I)$ to $\Coals(I)$ .
We set 
$\CCCC_{Sb}= \wScbp(\mathcal W_{Sb})$. For two examples, one for a walk in $\mathfrak{W}_{Sb}(I)$ and one for a walk in $\mathcal W_{Sb}$, the reader can look at Figure~\ref{fig:didwveiwvedbwedi_strong} and Figure~\ref{fig:ievdiuwbduwob_strong}. 

\begin{figure}[ht]
	\centering
	\includegraphics[width=\linewidth]{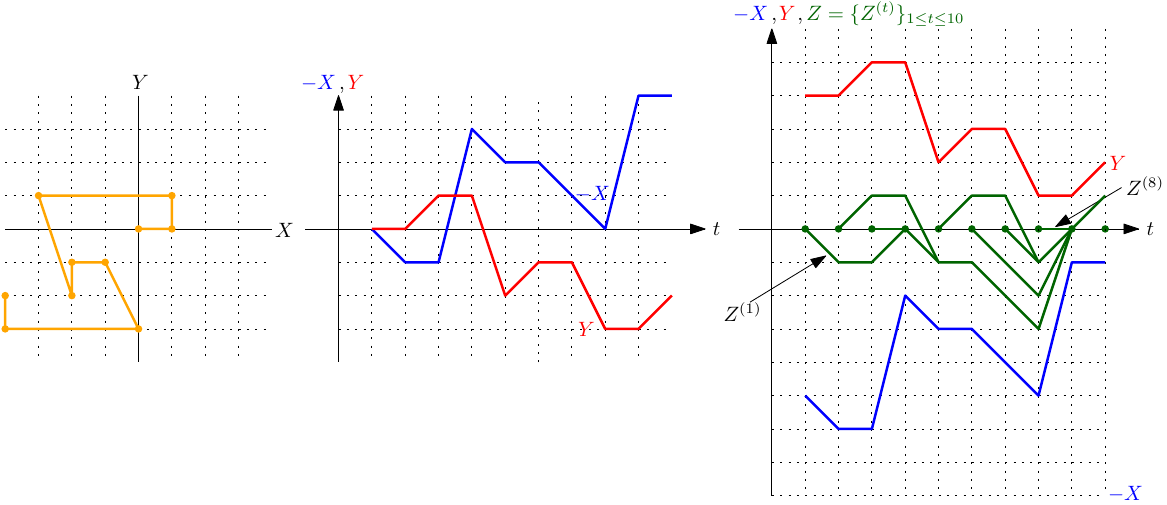}\\
	\caption{We explain the construction of a coalescent-walk process.
		\textbf{Left:} A two-dimensional walk $W=(W_t)_{t\in[10]}=(X_t,Y_t)_{t\in[10]}\in \mathfrak{W}_{Sb}([10])$.
		\textbf{Middle:} The two marginals $-X$ (in blue) and $Y$ (in red).
		\textbf{Right:} The two marginals are shifted and the ten walks of the coalescent-walk process are constructed in green. \label{fig:didwveiwvedbwedi_strong}}
\end{figure}

\begin{figure}[ht]
	\centering
	\includegraphics[scale=.6]{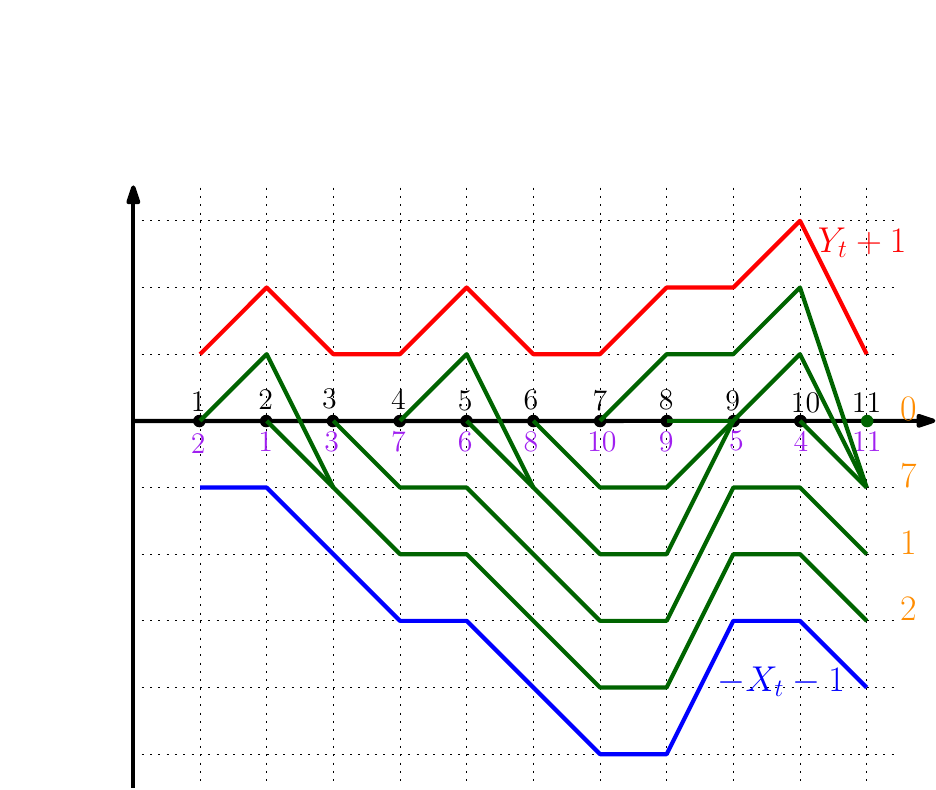}\\
	\caption{The coalescent walk process $\wScbp(W)$ for the walk $W$ considered in Figure~\ref{fig:schema_perm_strong}. In purple we plot the corresponding semi-Baxter permutation $\cpp((\wScbp(W)))$. Note that the latter permutation is equal to the permutation $\pw^{-1}(W)$ obtained in the last diagram in Figure~\ref{fig:schema_perm_strong}. In orange we plot the multiplicities of the final values of $\wScbp(W)$. \label{fig:ievdiuwbduwob_strong}}
\end{figure}

We give the following equivalent definition for later convenience.

\begin{definition}\label{defn:distrib_incr_coal_strong}
	Let $W\in\mathfrak{W}_{Sb}(I)$ and denote by $W_t = (X_t,Y_t)$ for $t\in I$. The \emph{coalescent-walk process associated with} $W$
	is the family of walks $\wScbp(W) = \{\Ztp\}_{t\in I}$, defined for $t\in I$ by $\Ztp_t=0,$ and for all $\ell\geq t$ such that $\ell+1 \in I$,
	\begin{equation}
		\Ztp_{\ell+1}=
		\begin{cases}
			\Ztp_{\ell}+(Y_{\ell+1}-Y_{\ell}), &\quad\text{if}\quad 
			\begin{cases}
				\Ztp_{\ell}= 0\text{ and }\Ztp_{\ell}-(X_{\ell+1}-X_{\ell})\geq 0,\\
				\Ztp_{\ell}>0\text{ and }\Ztp_{\ell}+(Y_{\ell+1}-Y_{\ell})>0,
			\end{cases}\\
			-1,&\quad\text{if}\quad 
			\Ztp_{\ell}> 0\text{ and }\Ztp_{\ell}+(Y_{\ell+1}-Y_{\ell})\leq0,\\
			\Ztp_{\ell}-(X_{\ell+1}-X_{\ell}),&\quad\text{if}\quad 
			\Ztp_{\ell}\leq 0\text{ and }\Ztp_{\ell}-(X_{\ell+1}-X_{\ell})<0,\\
			0,&\quad\text{if}\quad \Ztp_{\ell}< 0\text{ and }\Ztp_{\ell}-(X_{\ell+1}-X_{\ell})\geq 0.
		\end{cases}
	\end{equation} 
\end{definition}

\begin{remark}
	We note that the coalescent points of a coalescent-walk process obtained in this way have $y$-coordinates that are always equal either to 0 or to $-1$. 
\end{remark}

\begin{remark}\label{obs:alternating_excursions_strong}
	Note that every walk $\Ztp$ must pass through 0  between a non-positive and a strictly positive excursion. In addition the strictly positive excursions always start at one. See for instance at the walk $Z^{(1)}$ in Figure~\ref{fig:didwveiwvedbwedi_strong}.
	
	Another similar remarkable fact is that every walk $\Ztp$ must pass through $-1$ between a strictly positive  and a non-positive excursion. See for instance at the walk $Z^{(2)}$ in Figure~\ref{fig:didwveiwvedbwedi_strong}.
\end{remark}

The following definition and the consecutive lemma should clarify our definition of coalescent-walk process associated with a walk in $\mathcal{W}^n_{Sb}$ and the link with the corresponding strong-Baxter permutation.

\begin{definition}\label{defn:quantities_coal_walk_proc_strong}
	Let $W\in\mathcal{W}^n_{Sb}$ and consider the corresponding coalescent-walk process $\cpp(W)=\{\Ztp\}_{t\in [n]}\eqqcolon Z$. Assume that the set of final values $\{\Ztp_n\}_{t\in [n]}$ of the $n$ walks is equal to $$\FV(Z)\coloneqq\{f_{-x}<\dots<f_{-1}<f_{0}=0<f_1<\dots<f_{y}\},$$
	for some $x,y\in\Z_{\geq 0}$.
	For all $\ell\in[-x,y]$ we set $\mult(f_\ell)\coloneqq\#\{t\in [m]:\Ztp_n=f_{\ell}\}$ and by convention, we set $\mult(f_{\ell})=0$ for all $\ell\notin[-x,y]$.
\end{definition}

\begin{example}
	Consider the coalescent-walk process $Z$ in Figure~\ref{fig:ievdiuwbduwob_strong}. The set of final values $\FV(Z)$ is equal to $$\FV(Z)=\{-3,-2,-1,0\}.$$
	Moreover, $\mult(-3)=2$, $\mult(-2)=1$, $\mult(-1)=7$, and $\mult(0)=1$ because there are 2 green walks ending at -3, 1 green walk ending at -2, 7 green walks ending at -1, and 1 green walk ending at 0.
\end{example}

\begin{lemma}\label{lem:techn_for_comm_diagram_strong}
	Let $W\in\mathcal{W}^n_{Sb}$. Fix $m\in[n]$, and consider the corresponding coalescent-walk process $\cpp(W_{|_{[m]}})=\{\Ztp\}_{t\in [m]}\eqqcolon Z$ and the corresponding strong-Baxter permutation $\pw^{-1}(W_{|_{[m]}})=\pi$. Assume that $W_m=(x,y)\in\Z^2_{\geq 0}$, i.e.\ $\pi$ has $x+1$ active sites smaller than or equal to $\pi(m)$ and $y+1$ active sites greater than $\pi(m)$,  denoted by
	$$\{s_{-x}<\dots<s_0\}\cup\{s_{1}<\dots<s_{y+1}\}.$$ 
	Then
	$$\FV(Z)=\{f_{-x}<\dots<f_{-1}<f_{0}=0<f_1<\dots<f_{y}\}=[-x,y],$$
	and in particular, $f_\ell-f_{\ell-1}=1$, for all $\ell\in[-x+1,y]$.
	Moreover, it holds that 
	\begin{equation}\label{eq:rel_active_sites_strong}
		s_\ell=1+\sum_{j\leq \ell-1}\mult(f_j),\quad\text{for all}\quad \ell\in[-x,y+1].
	\end{equation}
\end{lemma}

\begin{proof}
	We prove the statement by induction over $m$.
	
	For $m=1$ then $x=0$, $y=0$, $\FV(Z)=\{0\}=\{f_0\}$ and $\mult(f_0)=1$. On the other hand, $\pi=1$ and the set of active sites is given by $\{s_0=1,s_1=2\}$. Note that  \eqref{eq:rel_active_sites_strong} holds.
	
	Now assume that $1\leq m<n$ and that $Z$ and $\pi$ verify the statement of the lemma. We are going to show that also $\cpp(W_{|_{[m+1]}})=\{{Z'}^{(t)}\}_{t\in [m+1]}\eqqcolon Z'$ and the corresponding strong-Baxter permutation $\pw^{-1}(W_{|_{[m+1]}})=\pi'$ also verify the statement of the lemma.  We distinguish three cases:
	\begin{itemize}
		\item \textbf{Case 1:} $W_{m+1}-W_m=(1,-i)$ for some $i\in \{0\}\cup[y]$ (see the left-hand side of Figure~\ref{fig:example_for_proof_strong}). 
		
		\noindent As explained in Section~\ref{sect:strong_bac_obj}, in this case $\pi'=\pi^{*s_{i+1}}$ and its active sites of $\pi'$ are 
		$$\{s'_{-x-1}<\dots <s'_0\}\cup\{s'_1<\dots<s'_{y-i+1}\}.$$
		
		where $s'_\ell=s_{\ell+1}$ for $\ell\in[-x-1,-1]$, $s'_0=s_{i+1}$, and $s'_\ell=s_{\ell+i}+1$ for $\ell\in[1,y-i+1]$.
		
		On the other hand, looking at Case 1 in \eqref{eq:icre_coal_strong}, we immediately have that 
		$$\FV(Z')=\{f'_{-x-1}<\dots<f'_{-1}<f'_{0}=0<f'_1<\dots<f'_{y-i}\}=[-x-1,y-i],$$
		and $\mult(f'_\ell)=\mult(f_{\ell+1})$ for all $\ell\in[-x-1,-2]$, $\mult(f'_{-1})=\sum_{\ell=0}^{i}\mult(f_{\ell})$, $\mult(f'_0)=1$, and $\mult(f'_\ell)=\mult(f_{\ell+i})$ for all $\ell\in[1,y-i]$.

		\item \textbf{Case 2:} $W_{m+1}-W_m=(0,1)$ (see the left-hand side of Figure~\ref{fig:example_for_proof_strong}). 
		
		\noindent As explained in Section~\ref{sect:strong_bac_obj}, in this case $\pi'=\pi^{*s_{0}}$ and its active sites of $\pi'$ are 
		$$\{s'_{-x}<\dots <s'_0\}\cup\{s'_1<\dots<s'_{y+2}\}.$$
		
		where $s'_\ell=s_{\ell}$ for $\ell\in[-x,0]$, and $s'_\ell=s_{\ell-1}+1$ for $\ell\in[1,y+2]$.
		
		On the other hand, looking at Case 2 in \eqref{eq:icre_coal_strong}, we immediately have that 
		$$\FV(Z')=\{f'_{-x}<\dots<f'_{-1}<f'_{0}=0<f'_1<\dots<f'_{y+1}\}=[-x,y+1],$$
		and $\mult(f'_\ell)=\mult(f_{\ell})$ for all $\ell\in[-x,-1]$, $\mult(f'_0)=1$, and $\mult(f'_\ell)=\mult(f_{\ell-1})$ for all $\ell\in[1,y+2]$.

		\item \textbf{Case 3:} $W_{m+1}-W_m=(-i,0)$ for some $i\in [x]$ (see the right-hand side of Figure~\ref{fig:example_for_proof_strong}). 
		
		\noindent  As explained in Section~\ref{sect:strong_bac_obj}, in this case $\pi'=\pi^{*s_{-i}}$ and its active sites of $\pi'$ are 
		$$\{s'_{-x+i}<\dots <s'_0\}\cup\{s'_1<\dots<s'_{y+1}\},$$
		with $s'_{\ell}=s_{\ell-i}$ for all $\ell\in[-x+i,0]$, and $s'_{\ell}=s_{\ell}+1$ for all $\ell\in[1,y+1]$.
		
		On the other hand, looking at Case 3 in \eqref{eq:icre_coal_strong}, we immediately have that 
		$$\FV(Z')=\{f'_{-x+i}<\dots<f'_{-1}<f'_{0}=0<f'_1<\dots<f'_{y}\}=[-x+i,y],$$
		and $\mult(f'_\ell)=\mult(f_{\ell-i})$ for all $\ell\in[-x+i,-1]$,
		$\mult(f'_0)=1+\sum_{\ell=-i}^{0}\mult(f_{\ell})$, and $\mult(f'_\ell)=\mult(f_{\ell})$ for all $\ell\in[1,y]$.
	\end{itemize} 
	With a straightforward computation, based on the expressions of the $s'_\ell$ and $\mult(f'_\ell)$ in terms of $s_\ell$ and $\mult(f_\ell)$, it can be checked that \eqref{eq:rel_active_sites} holds.
\end{proof}

\begin{figure}[htbp]
	\centering
	\includegraphics[width=\linewidth]{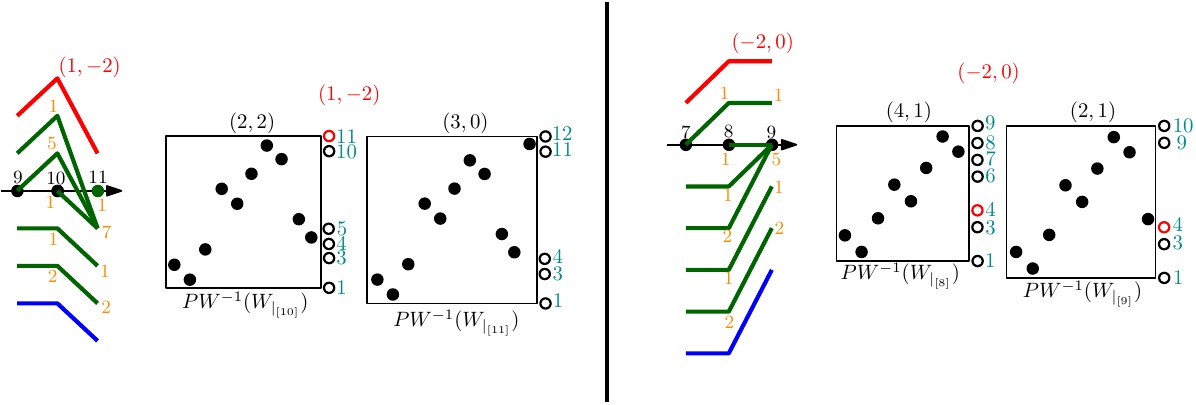}\\
	\caption{\textbf{Left:} The final steps of the coalescent-walk process $\wScbp(W_{|_{[11]}})$ in Figure~\ref{fig:ievdiuwbduwob_strong} and the corresponding permutations $\pw^{-1}(W_{|_{[10]}})$ and $\pw^{-1}(W_{|_{[11]}})$ from Figure~\ref{fig:schema_perm_strong} with the values of the active sites highlighted in cyan. We have that $W_{11}-W_{10}=(1,-2)$. Note that $\FV(\wScbp(W_{|_{[10]}}))=[-2,2]$ and from Figure~\ref{fig:ievdiuwbduwob_strong}, we can determine that $\mult(-2)=2$, $\mult(-1)=\mult(0)=\mult(2)=1$ and $\mult(1)=5$ (these numbers are plotted in orange close to the final values of the various walks). Note also that $\FV(\wScbp(W_{|_{[11]}}))=[-3,0]$ and we have that $\mult(-3)=2$, $\mult(-2)=\mult(0)=1$, and $\mult(-1)=7$.
		\textbf{Right: }The final steps of the coalescent-walk process $\wScbp(W_{|_{[9]}})$ in Figure~\ref{fig:ievdiuwbduwob_strong} and the corresponding permutations $\pw^{-1}(W_{|_{[8]}})$ and $\pw^{-1}(W_{|_{[9]}})$ from Figure~\ref{fig:schema_perm_strong}. We have that $W_9-W_8=(-2,0)$. Here, $\FV(\wScbp(W_{|_{[8]}}))=[-4,1]$ and $\mult(-3)=\mult(-1)=\mult(0)=\mult(1)=1$, and  $\mult(-4)=\mult(-2)=2$. Note also that $\FV(\wScbp(W_{|_{[9]}}))=[-2,1]$ and we have that $\mult(-2)=2$, $\mult(-1)=\mult(1)=1$, and $\mult(0)=5$. One can check, comparing the orange and cyan numbers, that in both cases \eqref{eq:rel_active_sites_strong} holds.
		\label{fig:example_for_proof_strong}}
\end{figure}

\subsection{The diagram commutes}

The goal of this section is to prove the following result.

\begin{theorem}\label{thm:The_diagram_commutes_strong}
	The diagram in \eqref{eq:diagrm_strong_baxter} commutes.
\end{theorem}

\begin{proof}
	We show that $\pw^{-1}=\cpp\circ\wScbp$. Fix $n\in\N$ and let $W\in\mathcal{W}^n_{Sb}$. We are going to prove the following.

	\noindent\textbf{Claim.} Assume that $m<n$ and $\pw^{-1}((W_i)_{i\in[m]})=\cpp\circ\wScbp((W_i)_{i\in[m]})$. 
	Then $$\pw^{-1}((W_i)_{i\in[m+1]})=\cpp\circ\wScbp((W_i)_{i\in[m+1]}).$$
	
	\medskip
	
	\noindent\textbf{Proof of the claim.} Since by assumption $\pw^{-1}((W_i)_{i\in[m]})=\cpp\circ\wScbp((W_i)_{i\in[m]})$ it is enough to show that $\pw^{-1}((W_i)_{i\in[m+1]})(m+1)=\cpp\circ\wScbp((W_i)_{i\in[m+1]})(m+1)$. By Remark~\ref{obs:final_point_strong_bax} and assuming that 
	$$\text{AS}(\pw^{-1}((W_i)_{i\in[m+1]}))=\{s_{-x}<\dots<s_0\}\cup\{s_{1}<\dots<s_{y+1}\},$$
	we have that  $\pw^{-1}((W_i)_{i\in[m+1]})(m+1)=s_0$. 
	
	On the other hand, setting $\wScbp((W_i)_{i\in[m+1]})=\{\Ztp\}_{t\in [m+1]}=Z$ and using Lemma~\ref{lem:techn_for_comm_diagram_strong}, we have that
	$$\FV(Z)=\{f_{-x}<\dots<f_{-1}<f_{0}=0<f_1<\dots<f_{y}\}=[-x,y].$$ 
	Then by Definition~\ref{defn:coalandpermrel} and Definition~\ref{defn:quantities_coal_walk_proc_strong} we have that 
	\begin{multline}
		\cpp\circ\wScbp((W_i)_{i\in[m+1]})(m+1)=\#\{t\in [m+1]:\Ztp_{m+1}\leq 0\}\\
		=\sum_{\ell\leq 0}\mult(f_\ell)=1+\sum_{\ell\leq -1}^{}\mult(f_\ell).
	\end{multline}
	
	Using \eqref{eq:rel_active_sites_strong} in Lemma~\ref{lem:techn_for_comm_diagram_strong} we can conclude that $1+\sum_{\ell\leq -1}^{}\mult(f_\ell)=s_0$, concluding the proof of the claim.
\end{proof}

\section{Probabilistic results for strong-Baxter permutations}\label{sect:prob_part}

The main goal of this section is to prove Theorem~\ref{thm:strong-baxter}. The proof of this result is divided in the following steps:

\begin{itemize}
	\item In Section~\ref{sect:strong_bax_walks} we explain how to sample a uniform strong-Baxter permutation as a two-dimensional walk conditioned to stay in a cone.
	
	\item Then, in Section~\ref{sect:den_wash_scal}, we prove a scaling limit result for this conditioned two-dimensional walk.
	
	\item In Section~\ref{sect:coal_scaling}, we prove a scaling limit result for the coalescent-walk process associated with this conditioned two-dimensional walk.
	
	\item Finally, in Section~\ref{sect:permlim} we explain how to deduce permuton convergence for strong-Baxter permutations (Theorem~\ref{thm:strong-baxter}) from the latter result.
\end{itemize}

\subsection{Sampling a uniform strong-Baxter permutation as a conditioned two-dimensional walk}\label{sect:strong_bax_walks}

Since the map $\pw:\mathcal{S}_{Sb}\to \mathcal{W}_{Sb}$ is a size-preserving bijection and using the definition of $\mathcal{W}_{Sb}$, in order to sample a uniform strong-Baxter permutation of size $n$, it is enough to sample a uniform two-dimensional walk in the non-negative quadrant of size $n$, starting at $(0,0)$, with increments in $I_{Sb}$.

Consider the following probability measure on  $I_{Sb}$ (see also the left-hand side of Figure~\ref{fig:jumps_strong_baxter}):
\begin{equation}\label{eq:distristep_strongbaxter}
	\mu_{Sb}=\sum_{i=1}^{\infty}\alpha\gamma^{i}\cdot \delta_{(-i,0)}
	+\alpha\theta^{-1}\cdot \delta_{(0,1)}+\sum_{i=0}^{\infty}\alpha\gamma^{-1}\theta^{i}\cdot\delta_{(1,-i)},
\end{equation}
where $\delta$ denotes the delta-Dirac measure and $\alpha,\theta,\gamma$ are the unique solutions of the following system of equations:
\begin{equation}\label{eq:system_strong}
	\begin{cases}
		\frac{1}{\gamma(1-\theta)}=\frac{\gamma}{(1-\gamma)^2},\\
		\frac{1}{\theta}=\frac{\theta}{\gamma(1-\theta)^2},\\
		\alpha=\frac{1}{\frac{1}{\theta}+\frac{\gamma}{1-\gamma}+\frac{1}{\gamma(1-\theta)}},\\
		\alpha>0,\theta>0,\gamma>0,\alpha\theta^{-1}\leq 1.
	\end{cases}
\end{equation}
With standard computations, we get that $\gamma$ is the unique real root of the polynomial $-1+2\gamma-\gamma^2+\gamma^3$, $\theta=-7+18\gamma-14\gamma^2+11\gamma^3-3\gamma^4$, and $\alpha=\frac{36}{11}-\frac{83}{11}\gamma+\frac{61}{11}\gamma^2-4\gamma^3+\frac{12}{11}\gamma^4$.
With a computer one can estimate that $\theta\approx0.43016$ , $\gamma\approx0.56984$, and $\alpha\approx0.14861$.

\begin{remark}
	We explain how one can find the distribution in \eqref{eq:distristep_strongbaxter} and what is the meaning of each equation in the system in \eqref{eq:system_strong}. The general idea is to first look for a distribution $\mu_{Sb}$ of geometric type (this is useful for consecutive computations such as the computations in \eqref{eq:fivbuiwebgf}). That is, one imposes that every step of type $(i,j)$ receives probability $\alpha\gamma^i\theta^j$. Then one also imposes three additional conditions (corresponding to the three equations in the system in \eqref{eq:system_strong}): the first condition is that $\mu_{Sb}$ is a probability measure (\emph{i.e.}, the total sum is equal to one); the second and the third conditions ensure that $\mu_{Sb}$ is a centered distribution (\emph{i.e.}, its expectation is $(0,0)$).
\end{remark}

Let $(\bm X,\bm Y)$ be a random variable such that $\mathcal{L}aw(\bm X,\bm Y)=\mu_{Sb}$.  With standard computations we have that:
\begin{align}\label{eq:computations_parm_strong}
	\E[\bm X]=\E[\bm Y]=0,&\qquad\E[\bm X\bm Y]=-\frac{\alpha\theta}{\gamma(1-\theta)^2}, \\
	\E[\bm X^2]=\alpha\left(\frac{1}{\gamma(1-\theta)}+\frac{\gamma(1+\gamma)}{(1-\gamma^3)}\right)\eqqcolon\sigma^2,&\qquad
	\E[\bm Y^2]=\alpha\left(\frac{1}{\theta}+\frac{\theta(1+\theta)}{\gamma(1-\theta^3)}\right)\eqqcolon{\sigma'}^2.
\end{align}
Therefore
\begin{equation}\label{eq:cov_strong-baxter}
	\Var((\bm X,\bm Y))=
	\alpha
	\begin{pmatrix}
		\frac{1}{\gamma(1-\theta)}+\frac{\gamma(1+\gamma)}{(1-\gamma^3)} & -\frac{\theta}{\gamma(1-\theta)^2} \\
		-\frac{\theta}{\gamma(1-\theta)^2} & \frac{1}{\theta}+\frac{\theta(1+\theta)}{\gamma(1-\theta^3)}
	\end{pmatrix},
\end{equation}
and so 
\begin{equation}\label{eq:corr_strong_baxt}
	\rho=\Cor((\bm X,\bm Y))=
	-\frac{\theta}{(1-\theta)^2\gamma\sqrt{\left(\frac{1}{\gamma(1-\theta)}+\frac{\gamma(1+\gamma)}{(1-\gamma^3)}  \right)\left(\frac{1}{\theta}+\frac{\theta(1+\theta)}{\gamma(1-\theta^3)}\right)}}\approx-0.21508.
\end{equation}
Equivalently, with some standard computations, we obtain that $\rho$ is the unique real solution of the polynomial $1+6\rho+8\rho^2+8\rho^3$.

\medskip

We now denote by
\begin{equation}
	\stackrel{\leftarrow}{I_{Sb}}\coloneqq\{(i,0): i\geq 1\}\cup\{(0,-1)\}\cup\{(-1,i): i\geq 0\},
\end{equation}
i.e.\ the set of ``reversed" increments (recall the definition of the set $I_{Sb}$ in \eqref{eq:strong_bax_increm}). We further denote by 
\begin{equation}
	\stackrel{\leftarrow}{\mu_{Sb}}=\sum_{i=1}^{\infty}\alpha\gamma^{i}\cdot \delta_{(i,0)}
	+\alpha\theta^{-1}\cdot \delta_{(0,-1)}+\sum_{i=0}^{\infty}\alpha\gamma^{-1}\theta^{i}\cdot\delta_{(-1,i)},
\end{equation}
the ``reversed" distribution on $\stackrel{\leftarrow}{I_{Sb}}$ induced by $\mu_{Sb}$ (see also the right-hand side of Figure~\ref{fig:jumps_strong_baxter}).

\begin{figure}[t]
	\centering
	\hspace{3.1cm}\includegraphics[scale=0.7]{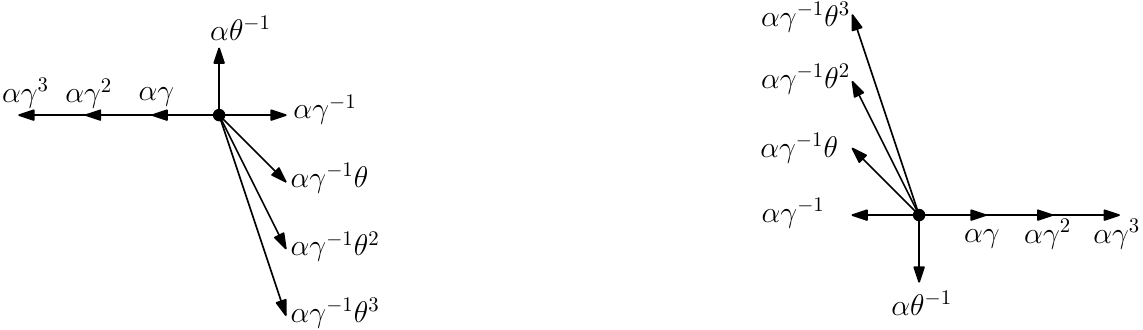}
	\caption{\textbf{Left:} Some of the increments in the set $I_{Sb}$ are plotted together with the corresponding probability weights given by $\mu_{Sb}$. \textbf{Right:} Some increments in the set $\stackrel{\leftarrow}{I_{Sb}}$ are plotted together with the corresponding probability weights given by $\stackrel{\leftarrow}{\mu_{Sb}}$.\label{fig:jumps_strong_baxter}}
\end{figure}

For all $n\in\zgz$, we define the following additional probability measure
\begin{equation}\label{eq:starting_prop_strong}
	\nu_{Sb}^n=\frac{1}{Z_n}\sum_{(h,\ell)\in\mathcal{L}^n_{Sb}}\gamma^{h}\theta^{\ell} \delta_{(h,\ell)},
\end{equation}
where 
$
\mathcal{L}_{Sb}^n\coloneqq\{\text{Labels at level $n$ in the generating tree for strong-Baxter permutations}\}
$
and the normalizing constant satisfies
$
Z_n=\sum_{(h,\ell)\in\mathcal{L}^n_{Sb}}\gamma^{h}\theta^{\ell}.
$

Let $(\stackrel{\leftarrow}{\bm W_n}(i))_{i\geq 1}$ be a two-dimensional random walk with increments distributed as $\stackrel{\leftarrow}{\mu_{Sb}}$ and starting probability $\nu_{Sb}^n(h,\ell)=\P(\stackrel{\leftarrow}{\bm W_n}(1)=(h,\ell))$.
Denote by $\stackrel{\leftarrow}{\mathcal{W}_{Sb}^n}$ the set of two-dimensional walks $(x_i)_{i\in[n]}$ in the non-negative quadrant with increments in $\stackrel{\leftarrow}{I_{Sb}}$ and such that $x_n=(0,0)$.
\begin{proposition}\label{prop:the_walk_is_uniform_strong}
	Conditioning on the event $\left\{(\stackrel{\leftarrow}{\bm W_n}(i))_{i\in [n]}\in\stackrel{\leftarrow}{\mathcal{W}_{Sb}^n}\right\}$, the walk $(\stackrel{\leftarrow}{\bm W_n}(i))_{i\in [n]}$ is a uniform walk in $\stackrel{\leftarrow}{\mathcal{W}_{Sb}^n}$.
\end{proposition}

\begin{proof}
	Fix $(x_i)_{i\in[n]}\in\stackrel{\leftarrow}{\mathcal{W}_{Sb}^n}$. It is enough to show that $\P\left((\stackrel{\leftarrow}{\bm W_n}(i))_{i\in [n]}=(x_i)_{i\in[n]}\right)$ is independent of the choice of $(x_i)_{i\in[n]}$. To do that, assume that $x_1=(h,\ell)$ and recall that $x_n=(0,0)$. By definition of $\stackrel{\leftarrow}{\mu_{Sb}}$ and $\nu_{Sb}^n$ we have that
	\begin{equation}\label{eq:fivbuiwebgf}
		\P\left((\stackrel{\leftarrow}{\bm W_n}(i))_{i\in [n]}=(x_i)_{i\in[n]}\right)=\nu_{Sb}^n(h,\ell)\cdot\alpha^{n-1}\gamma^{-h}\theta^{-\ell}=\frac{\gamma^{h}\theta^{\ell}}{Z_n}\cdot\alpha^{n-1}\gamma^{-h}\theta^{-\ell}=\frac{\alpha^{n-1}}{Z_n},
	\end{equation}
	and this concludes the proof.
\end{proof}

Let $(\bm W_n(i))_{i\geq 1}$ be the reversed walk obtained from $(\stackrel{\leftarrow}{\bm W_n}(i))_{i\geq 1}$. A consequence of Proposition~\ref{prop:the_walk_is_uniform_strong} is the following.

\begin{corollary}\label{cor:the_walk_is_uniform_strong}
	Conditioning on the event $\left\{(\bm W_n(i))_{i\in[n]}\in\mathcal{W}_{Sb}^n\right\}$, the walk $(\bm W_n(i))_{i\in [n]}$ is a uniform walk in $\mathcal{W}_{Sb}^n$.
\end{corollary}

\subsection{Scaling limit of the conditioned two-dimensional walks for strong-Baxter permutations}\label{sect:den_wash_scal}

We define a rescaled version of the walk $(\bm W_n(i))_{i\geq 1}=(\bm X_n(i),\bm Y_n(i))_{i\geq 1}$: for all $n\geq 1,$ let $\conti W_n:[0,1]\to \R^2$ be the continuous function defined by linearly interpolating the points
\begin{equation}\label{eq:resc_walk}
	\conti W_n\left(\frac kn\right) =  \left(\frac {\bm X_n(k)} {\sigma\sqrt{n}},\frac {\bm Y_n(k)} {\sigma'\sqrt{n}}\right),\quad\text{ for all } k\in [n],
\end{equation}
where $\sigma$ and $\sigma'$ were defined in \eqref{eq:computations_parm_strong}.

\medskip

All the spaces of continuous functions considered below are implicitly endowed with the
topology of uniform convergence on every compact set.

\begin{proposition}\label{prop:scaling_strong_walk}
	Let $\rho$ be the unique real solution of the polynomial 
	\begin{equation}
		1+6\rho+8\rho^2+8\rho^3.
	\end{equation}
	Conditioning on the event $\left\{(\bm W_n(i))_{i\in[n]}\in\mathcal{W}_{Sb}^n\right\}$, we have the following convergence in $\czord$, 
	\begin{equation}\label{eq:conv_to_brow_excurs_strong}
		\conti W_n
		\xraninf{d}
		\conti E_{\rho},
	\end{equation}
	where we recall that $\conti E_{\rho}=(\conti E_{\rho}(t))_{t\in [0,1]}$ denotes a two-dimensional Brownian excursion of correlation $\rho$ in the non-negative quadrant.
\end{proposition}

\begin{proof}
	An immediate consequence of \cite[Theorem 6]{MR4102254} is that, for every choice of $(h,\ell)\in\Z^2_{\geq0}$, \eqref{eq:conv_to_brow_excurs_strong} holds conditioning on 
	$$\left\{(\bm W_n(i))_{i\in[n]}\in\mathcal{W}_{Sb}^n\right\}\cap\left\{\bm W_n(n)=(h,\ell)\right\}.$$
	Therefore in order to prove the proposition it is enough to show that 
	\begin{equation}\label{eq:goal_of_proof_strong}
		\sum_{(h,\ell)\in\Z^2_{\geq 0}}\liminf_{n\to\infty}\P\left(\bm W_n(n)=(h,\ell)\middle|(\bm W_n(i))_{i\in[n]}\in\mathcal{W}_{Sb}^n\right)=1.
	\end{equation}
	Note that
	\begin{multline}\label{eq:fbbfowenfwe}
		\P\left(\bm W_n(n)=(h,\ell)\middle|(\bm W_n(i))_{i\in[n]}\in\mathcal{W}_{Sb}^n\right)\\
		=
		\frac{\P\left((\bm W_n(i))_{i\in[n]}\in\mathcal{W}_{Sb}^n\middle|\bm W_n(n)=(h,\ell)\right)\cdot \P\left(\bm W_n(n)=(h,\ell)\right)}{\P\left((\bm W_n(i))_{i\in[n]}\in\mathcal{W}_{Sb}^n\right)}.
	\end{multline}
	From Corollary~\ref{cor:upper_bound_walks_strong}, there exist a constant $C>0$ independent of $h,\ell,n$ and a parameter $p>0$ such that
	\begin{equation}\label{eq:fbuofbwoif}
		\P\left((\bm W_n(i))_{i\in[n]}\in\mathcal{W}_{Sb}^n\middle|\bm W_n(n)=(h,\ell)\right)\leq  C(1+ |(h,\ell)|^p) n^{-p-1}, \quad\text{for all}\quad n,h,\ell\in \Z_{\geq 0},
	\end{equation}
	where $|(h,\ell)|$ denotes the Euclidean norm of the vector $(h,\ell)$.
	Moreover, using the definition of the measure $\nu_{Sb}^n$ in \eqref{eq:starting_prop_strong}, we have that
	\begin{equation}
		\P\left((\bm W_n(i))_{i\in[n]}\in\mathcal{W}_{Sb}^n\right)=\frac{1}{Z_n}\sum_{(h,\ell)\in\mathcal{L}^n_{Sb}}\P\left((\bm W_n(i))_{i\in[n]}\in\mathcal{W}_{Sb}^n\middle|\bm W_n(n)=(h,\ell)\right)\gamma^{h}\theta^{\ell}.
	\end{equation} 
	Noting that $(1,0)\in\mathcal{L}^n_{Sb}$ for all $n\in\N$, and that $Z_n\leq \frac{1}{(1-\gamma)(1-\theta)}$ for all $n\in\N$, we have for all $n\in\N$,
	\begin{equation}\label{eq:ivbfweubfewnfoew}
		\P\left((\bm W_n(i))_{i\in[n]}\in\mathcal{W}_{Sb}^n\right)\geq\gamma(1-\gamma)(1-\theta)\P\left((\bm W_n(i))_{i\in[n]}\in\mathcal{W}_{Sb}^n\middle|\bm W_n(n)=(1,0)\right)\geq C'n^{-p-1},
	\end{equation}
	where in the last inequality we used Corollary~\ref{cor:upper_bound_walks_strong2}.
	Substituting in \eqref{eq:fbbfowenfwe} the bounds obtained in \eqref{eq:fbuofbwoif} and \eqref{eq:ivbfweubfewnfoew}, together with the bound $\P\left(\bm W_n(n)=(h,\ell)\right)\leq C''\gamma^{h}\theta^{\ell}$, we can conclude that
	\begin{equation}
		\P\left(\bm W_n(n)=(h,\ell)\middle|(\bm W_n(i))_{i\in[n]}\in\mathcal{W}_{Sb}^n\right)
		\leq
		C'''(1+ |(h,\ell)|^p)\gamma^{h}\theta^{\ell}\quad\text{for all}\quad n,h,\ell\in \Z_{\geq 0}.
	\end{equation}
	Therefore for every $\varepsilon>0$ there exists a compact set $K_{\varepsilon}\subseteq \Z^2$ such that 
	\begin{equation}
		\sum_{(h,\ell)\in\Z^2_{\geq 0}\setminus K_{\varepsilon}}\P\left(\bm W_n(n)=(h,\ell)\middle|(\bm W_n(i))_{i\in[n]}\in\mathcal{W}_{Sb}^n\right)
		\leq \varepsilon, \quad\text{for all}\quad n \in \Z_{\geq 0}.
	\end{equation}
	As a consequence, for every $\varepsilon>0$,
	\begin{multline}
		\sum_{(h,\ell)\in\Z^2_{\geq 0}}\liminf_{n\to\infty}\P\left(\bm W_n(n)=(h,\ell)\middle|(\bm W_n(i))_{i\in[n]}\in\mathcal{W}_{Sb}^n\right)\\
		\geq 
		\liminf_{n\to\infty}
		\sum_{(h,\ell)\in K_{\varepsilon}}\P\left(\bm W_n(n)=(h,\ell)\middle|(\bm W_n(i))_{i\in[n]}\in\mathcal{W}_{Sb}^n\right)\\
		\geq 1-
		\limsup_{n\to\infty}\sum_{(h,\ell)\in \Z^2_{\geq 0}\setminus K_{\varepsilon}}\P\left(\bm W_n(n)=(h,\ell)\middle|(\bm W_n(i))_{i\in[n]}\in\mathcal{W}_{Sb}^n\right)\geq 1-\varepsilon.
	\end{multline}
	This proves \eqref{eq:goal_of_proof_strong} and concludes the proof of the proposition.
\end{proof}

\subsection{Scaling limit of coalescent-walk processes for strong-Baxter permutations}
\label{sect:coal_scaling}
In this section we first prove a scaling limit result for coalescent-walk processes associated with strong-Baxter permutation, both in the unconditioned (see Theorem~\ref{thm:fvwuofbgqfipqhfqfpq_strongb}) and conditioned (see Proposition~\ref{prop:jfwvbvouvwefuow} and Proposition~\ref{prop:jfbeouwfboufwe}) case. Then, in Section~\ref{sect:permlim}, we explain how to deduce permuton convergence for strong-Baxter permutations (Theorem~\ref{thm:strong-baxter}) from Proposition~\ref{prop:jfbeouwfboufwe}.

\subsubsection{The unconditioned scaling limit}

Let $\overline{\bm W} = (\obmx,\obmy) =(\obmx(k),\obmy(k))_{k\in \Z}$ be a random bi-infinite two-dimensional walk with step distribution $\mu_{Sb}$ (defined in \eqref{eq:distristep_strongbaxter}), and let $\ovbmZ = \wScbp(\overline{\bm W})$ be the corresponding discrete coalescent-walk process. For convenience, we set $\ovbmZ^{(j)}_{i} = 0$ for $i,j \in \Z$, $i < j$.

We introduce the following rescaled processes: for all $n\geq 1, u\in \R$, let $\ovconw_n:\R\to \R^2$, and $\contzu_n:\R\to\R$ be the continuous processes that interpolate the following points:
\begin{equation}\label{eq:rescaled_version_strong}
	\ovconw_n\left(\frac kn\right) = \left(\frac {\obmx(k)} {\sigma\sqrt {n}},\frac {\obmy(k)} {\sigma'\sqrt {n}}\right), \quad\text{for all}\quad k\in \Z,
\end{equation}
where $\sigma$ and $\sigma'$ are defined in \eqref{eq:computations_parm_strong}, and
\begin{equation}\label{eq:rescaled_version2_strong}
	\contzu_{n}\left(\frac kn\right) =
	\begin{cases}
		\frac {\ovbmZ^{(\lceil nu\rceil)}_{k}} { \sigma'\sqrt {n}},\quad&\text{when}\quad\ovbmZ^{(\lceil nu\rceil)}_{k}\geq 0,\\
		\frac {\ovbmZ^{(\lceil nu\rceil)}_{k}} {\sigma\sqrt{ n}},\quad&\text{when}\quad\ovbmZ^{(\lceil nu\rceil)}_{k}< 0,
	\end{cases}
	\quad\text{for all}\quad k\in \Z.
\end{equation}

We also introduce the potential limiting processes. Fix $\rho\in[-1,1]$ and $q\in[0,1]$. Consider the solutions (that exist and are unique thanks to \cite[Theorem 2.1]{borga2021skewperm}) of the following family\footnote{Note that the SDEs in \eqref{eq:flow_SDE_inf_vol_strong} are the unconditioned version of the SDEs in \eqref{eq:flow_SDE_gen}, p.\ \pageref{eq:flow_SDE_gen}.} of SDEs indexed by $u\in \R$ and driven by a two-dimensional Brownian motion $\ovconw_{\rho} = (\overline{\conti X}_{\rho},\overline{\conti Y}_{\rho})$ of correlation $\rho$:
\begin{equation}\label{eq:flow_SDE_inf_vol_strong}
	\begin{cases}
		d\ovconz_{\rho,q}^{(u)}(t) = \idf_{\{\ovconz_{\rho,q}^{(u)}(t)> 0\}} d\overline{\conti Y}_{\rho}(t) - \idf_{\{\ovconz_{\rho,q}^{(u)}(t)< 0\}} d \overline{\conti X}_{\rho}(t)+(2q-1)\cdot d\conti L^{\ovconz_{\rho,q}^{(u)}}(t),& t\geq u,\\
		\ovconz_{\rho,q}^{(u)}(t)=0,&  t\leq u,
	\end{cases} 
\end{equation}
where we recall that $\conti L^{\ovconz_{\rho,q}^{(u)}}(t)$ is the symmetric local-time process at zero of $\ovconz_{\rho,q}^{(u)}$. We remark that, as stated in \cite[Theorem 2.1]{borga2021skewperm}, the processes $\left\{{(\ovconz}_{\rho,q}^{(u)}(t))_{t\geq u}\right\}_{u\in\R}$ are skew Brownian motions of parameter $q.$ We recall that a \emph{skew Brownian motion} of parameter $q\in[0,1]$ is a standard one-dimensional Brownian motion where each excursion is flipped independently to the positive side with probability $q$ (see for instance \cite[Theorem 6]{MR2280299}).

\medskip

We now prove a scaling limit result for a walk of the coalescent-walk process.

\begin{theorem}\label{thm:fvwuofbgqfipqhfqfpq_strongb}
	Let $u\in \R$. 
	The following joint convergence holds in the space $\mathcal C(\R,\R)^{3}$: 
	\begin{equation}
		\label{eq:fvwuofbgqfipqhfqfpq_strongb}
		\left(\ovconw_n,\contzu_n\right) 
		\xraninf{d}
		\left(\ovconw_{\rho},\ovconz_{\rho,q}^{(u)}\right),
	\end{equation}
	where $\rho$ is the unique real solution of the polynomial 
	\begin{equation}
		1+6\rho+8\rho^2+8\rho^3,
	\end{equation} 
	and $q$ is the unique real solution of the polynomial
	\begin{equation}\label{eq:param_q_strong}
		-1+6q-11q^2+7q^3.
	\end{equation} 
\end{theorem}

In the remaining part of this section we give the proof of Theorem~\ref{thm:fvwuofbgqfipqhfqfpq_strongb}. We start by stating the following key proposition whose proof is postponed to the end of the section.

\begin{proposition}\label{prop:skew_part_strong}
	Fix $u\in \R$. 
	We have the following convergence in $\mathcal C(\R,\R)$: 
	\begin{equation}
		\contzu_n 
		\xraninf{d} \overline{\conti B}_{q}^{(u)},
	\end{equation}
	where $\overline{\conti B}_{q}^{(u)}(t)=0$ for $t<u$ and $\overline{\conti B}_{q}^{(u)}(t)$ is a skew Brownian motion of parameter
	$q$ (defined in \eqref{eq:param_q_strong}) for $t\geq u$.
\end{proposition}

The proof of Theorem~\ref{thm:fvwuofbgqfipqhfqfpq_strongb} is in some steps similar to the proof of \cite[Theorem 4.5]{borga2020scaling}, the main difference being that here we are dealing with a skew Brownian motion instead of a classical Brownian motion. Since the proof is quite short, we include all the details for the sake of completeness.
\begin{proof}[Proof of Theorem~\ref{thm:fvwuofbgqfipqhfqfpq_strongb}]
	We want to show the following joint convergence in $\mathcal C(\R,\R)^{3}$:
	\begin{equation}\label{eq:fvwuofbgqfipqhfqfpq}
		\left(\ovconw_n,\contzu_n\right) 
		\xraninf{d}
		\left(\ovconw_{\rho},\ovconz_{\rho,q}^{(u)}\right).
	\end{equation}
	By Donsker's theorem and using the expression for the correlation $\rho$ of the distribution $\mu_{Sb}$ given in \eqref{eq:corr_strong_baxt},
	we have that $\ovconw_n$ converges in distribution to $\ovconw_{\rho}$. The convergence of second component follows from Proposition~\ref{prop:skew_part_strong} since $\ovconz_{\rho,q}^{(u)}$ is a skew Brownian motion of parameter $q$ started at time $u$.
	These results prove component-wise convergence in \eqref{eq:fvwuofbgqfipqhfqfpq}.
	
	\medskip
	
	We now establish joint convergence. Using Prokhorov's theorem, the marginals $\ovconw_n$ and  $\contzu_n$ are tight and so the left-hand side of \eqref{eq:fvwuofbgqfipqhfqfpq} is a tight sequence. Using a second time Prokhorov's theorem, in order to prove \eqref{eq:fvwuofbgqfipqhfqfpq} it is sufficient to show that all joint subsequential limits have the same distribution. Assume now that on a subsequence, we have 
	\begin{equation}\label{eq:fbibnfipqefipepf}
		\left(\ovconw_n,\ovconz^{( u )}_n\right) 
		\xraninf{d}
		\left(\ovconw_{\rho},\widetilde {\conti Z}\right),
	\end{equation}
	where $\widetilde {\conti Z}$ is a skew Brownian motion of parameter $q$ started at time $u$. We want to show that $\widetilde {\conti Z} = \ovconz_{\rho,q}^{(u)}$ a.s.
	Thanks to Skorokhod's theorem, we can also assume that \eqref{eq:fbibnfipqefipepf} is in fact an almost sure convergence.
	
	Let $(\mathcal G_t)_t$ be the completion of $\sigma(\ovconw_{\rho}(s),\wconz(s), s\leq t)$ by the negligible events. Then the processes $\ovconw_{\rho}$ and $\wconz$ are $\mathcal G_t$-adapted.
	We are going to show that $\ovconw_{\rho}$ is a $(\mathcal G_t)_t$-Brownian motion, i.e., 
	\begin{equation}
		\left(\ovconw_{\rho}(t+s)-\ovconw_{\rho}(t)\right) \indep \mathcal G_t,\quad\text{for}\quad t\in \R, s\in \R_{\geq 0}.
	\end{equation}
	For fixed $n$, $\left(\ovconw_n({t+s}) - \ovconw_n(t)\right) \indep \sigma\left(\overline{\bm W}(k), k\leq \lfloor n t \rfloor\right)$ and so
	$$
	\left(\ovconw_n({t+s})-\ovconw_n(t)\right)  \indep \left(\ovconw_n(r),\contzu_n(r)\right)_{r \leq n^{-1} \lfloor n  t \rfloor
	}.
	$$
	Since $\ovconw_n$ converges to $\ovconw$, we have that $\left(\ovconw_{\rho}(t+s)-\ovconw_{\rho}(t)\right)
	\indep \left(\ovconw_{\rho}(r),\wconz(r)\right)_{r\leq t}$. Thus, $\ovconw_{\rho}$ is a $(\mathcal G_t)_t$-Brownian motion.
	
	Now fix $\eps>0, \eps\in\mathbb{Q}$ and an open interval $(\bm a,\bm b)$ with $\bm a,\bm b\in\mathbb{Q}$ on which $\wconz(t)>\eps$. We remark that $(\bm a,\bm b)$ depends on $\wconz(t)$. By a.s.\ convergence, there exists $N_0$ such that for all $n\geq N_0$, $\contzu_n>\eps/2$ on $(\bm a,\bm b)$. 
	The process $(\contzu_n - \overline{\conti Y}_n)|_{(\bm a+1/n,\bm b)}$ is constant by Definition~\ref{defn:distrib_incr_coal_strong} and so its limit $(\wconz - \overline{\conti Y}_{\rho})|_{(\bm a,\bm b)}$ is constant too a.s. We have shown that a.s.\ $\wconz - \overline{\conti Y}_{\rho}$ is locally constant on the set $\{t>u:\wconz(t)>\eps\}$. Thus we have that a.s.
	\begin{equation}
		\int_{u}^t \idf_{\{\wconz(r)>\eps\}} d\wconz(r) = \int_u^t \idf_{\{\wconz(r)>\eps\}} d\overline{\conti Y}_{\rho}(r), \quad t\geq u. 
	\end{equation}
	The two stochastic integrals above are well-defined: for the first one it is enough to consider the filtration of $\widetilde {\conti Z}$, for the second one, the filtration $(\mathcal G_t)_t$.	
	Using similar arguments for negative values, we obtain that 
	\begin{equation}
		\int_u^t \idf_{\{|\wconz(r)|>\eps\}} d\wconz(r)  =  \int_{u}^t \idf_{\{\wconz(r)>\eps\}} d\overline{\conti Y}_{\rho}(r) -  \int_{u}^t \idf_{\{\wconz(r)<-\eps\}} d\overline{\conti X}_{\rho}(r).
	\end{equation}
	By stochastic dominated convergence theorem, we are allowed to take the limit as $\eps\to 0$, \cite[Thm. IV.2.12]{revuz2013continuous}, and deduce that 
	\begin{equation}
		\int_u^t \idf_{\{\wconz(r)\neq0\}}d\wconz(r)  =  \int_{u}^t \idf_{\{\wconz(r)>0\}} d\overline{\conti Y}_{\rho}(r) -  \int_{u}^t \idf_{\{\wconz(r)<0\}} d\overline{\conti X}_{\rho}(r).
	\end{equation}
	Since $\wconz$ is a skew Brownian motion of parameter $q$, we have that (see for instance \cite[Eq.(38-39)]{MR2280299})
	$$\int_u^t \idf_{\{\wconz(r)=0\}}d\wconz(r) = (2q-1) \int_u^td\conti L^{\wconz}(r),$$
	where we recall that $\conti L^{\wconz}$ denotes the symmetric local-time process at zero of $\wconz$.
	Since 
	$$\int_u^t \idf_{\{\wconz(r)\neq0\}}d\wconz(r)+\int_u^t \idf_{\{\wconz(r)=0\}}d\wconz(r)=\wconz(t),$$ 
	using the two displayed equations above, we obtain that
	\begin{equation}
		\wconz(t)=\int_{u}^t \idf_{\{\wconz(r)>0\}} d\overline{\conti Y}_{\rho}(r) -  \int_{u}^t \idf_{\{\wconz(r)<0\}} d\overline{\conti X}_{\rho}(r)+(2q-1) \int_u^td\conti L^{\wconz}(r),
	\end{equation}
	As a result $\widetilde {\conti Z}$ verifies \eqref{eq:flow_SDE_inf_vol_strong} almost surely and we can apply pathwise uniqueness (\cite[Theorem 2.1]{borga2021skewperm}) to complete the proof that $\widetilde {\conti Z} = \ovconz_{\rho,q}^{(u)}$ a.s.
\end{proof}

It remains to prove Proposition~\ref{prop:skew_part_strong}. We recall that (see for instance \cite[Equation (17)]{MR2280299}) for a skew Brownian motion $(\overline{\conti B}_{q}(t))_{t\geq 0}$ of parameter $q$ it holds that for any non-negative and continuous function $\varphi$ with compact support,
\begin{equation}
	\E\left[\varphi(\overline{\conti B}_{q}(t))\right]=q\int_{0}^{+\infty}\varphi(y)\frac{2e^{-y^2/2t}}{\sqrt{2\pi t}}dy+(1-q)\int_{-\infty}^{0}\varphi(y)\frac{2e^{-y^2/2t}}{\sqrt{2\pi t}}dy.
\end{equation}

\begin{proof}[Proof of Proposition~\ref{prop:skew_part_strong}]
	In order to prove our result it is enough to show:
	\begin{itemize}
		\item convergence of one-dimensional marginal distributions;
		\item convergence of finite-dimensional marginal distributions;
		\item tightness.
	\end{itemize}
	
	This \emph{three-steps standard approach} is also used in \cite{MR4105264} to show convergence of some specific random walks to a skew Brownian permuton. Unfortunately the models of random walks considered in \cite{MR4105264} do not include our walks.
	Here we show all the details of the convergence of one-dimensional marginal distributions for the walk $\contzu_n$, then convergence of finite-dimensional marginal distributions and tightness follows as in \cite{MR4105264}, using the same kind of modifications presented here for the convergence of one-dimensional marginal distributions.
	
	We consider the case $u=0$, the general proof being similar. We recall that $\overline {\conti Z}_n\coloneqq\overline {\conti Z}^{(0)}_n:\R\to\R$ is the continuous process defined by linearly interpolating the following points:
	\begin{equation}
		\ovconz_{n}\left(\frac kn\right) =
		\begin{cases}
			\frac {\ovbmZ_{k}} {\sigma'\sqrt {n}},\quad&\text{when}\quad\ovbmZ_{k}\geq 0,\\
			\frac {\ovbmZ_{k}} {\sigma\sqrt{n}},\quad&\text{when}\quad\ovbmZ_{k}< 0,
		\end{cases}
		\quad\text{for all}\quad k\in \Z,
	\end{equation}
	where ${\ovbmZ_{k}}\coloneqq {\ovbmZ^{(0)}_{k}}$. 
	For the rest of the proof we fix a non-negative and Lipschitz continuous function $\varphi$ with compact support.
	We want to show that for all $t\geq 0$,
	\begin{equation}
		\E\left[\varphi\left(\overline {\conti Z}_n(t)\right)\right]\to
		q\int_{0}^{+\infty}\varphi(y)\frac{2e^{-y^2/2t}}{\sqrt{2\pi t}}dy+(1-q)\int_{-\infty}^{0}\varphi(y)\frac{2e^{-y^2/2t}}{\sqrt{2\pi t}}dy.
	\end{equation}
	Since $\varphi$ is Lipschitz continuous, as shown in the first lines of the proof of \cite[Lemma 3.1]{MR4105264}, it is enough to prove that
	\begin{equation}\label{eq:jbfgweiufbowfwfnoiweb_strong}
		\E\left[\varphi\left(\frac {\ovbmZ_{\lfloor nt\rfloor}} {\sigma'\sqrt  n}\right);\ovbmZ_{\lfloor nt\rfloor}> 0\right]\to q\int_{0}^{+\infty}\varphi(y)\frac{2e^{-y^2/2t}}{\sqrt{2\pi t}}dy,
	\end{equation}
	and that
	\begin{equation}\label{eq:bkjfkwebfoewnfpoew}
		\E\left[\varphi\left(\frac {\ovbmZ_{\lfloor nt\rfloor}} {\sigma\sqrt  n}\right);\ovbmZ_{\lfloor nt\rfloor}< 0\right]\to (1-q)\int_{-\infty}^{0}\varphi(y)\frac{2e^{-y^2/2t}}{\sqrt{2\pi t}}dy.
	\end{equation}
	We start with the first expectation. It can be decomposed as 
	\begin{equation}
		\E\left[\varphi\left(\frac {\ovbmZ_{\lfloor nt\rfloor}} {\sigma'\sqrt  n}\right);\ovbmZ_{\lfloor nt\rfloor}> 0\right]=
		\sum_{k=0}^{\lfloor nt\rfloor}\sum_{\ell=0}^{+\infty}\E\left[\varphi\left(\frac {\ovbmZ_{\lfloor nt\rfloor}} {\sigma'\sqrt  n}\right);\bm \tau^{+}_{\ell}=k,(\ovbmZ_{i})_{i\in[k+1,\lfloor nt\rfloor]}> 0\right],
	\end{equation}
	where $\bm \tau^{+}_{0}=0$ and $\bm \tau^{+}_{\ell+1}=\inf\{i\geq\bm \tau^{+}_{\ell}:\ovbmZ_{i}=0\}$ for all $\ell\in\Z_{\geq 0}$. Note that for all $\ell\in\Z_{\geq 0}$, using Definition~\ref{defn:distrib_incr_coal_strong} we have that
	$$\left(\ovbmZ_{\lfloor nt\rfloor};\bm \tau^{+}_{\ell}=k,(\ovbmZ_{i})_{i\in[k+1,\lfloor nt\rfloor]}> 0\right)\stackrel{d}{=}\left(\bm S_{\lfloor nt\rfloor-k};\tau^{+}_{\ell}=k,(\bm S_{i})_{i\in[\lfloor nt\rfloor-k]}> 0\right),$$ 
	where $(\bm S_i)_{i\geq 0}$ denotes a random walk started at zero at time zero, with step distribution equal to the distribution of $\obmy(1)-\obmy(0)$, and independent of $\bm \tau^{+}_{\ell}$. Therefore we can write
	\begin{multline}\label{eq:webhfyuwheribfiuewbfoweuf_strong}
		\E\left[\varphi\left(\frac {\ovbmZ_{\lfloor nt\rfloor}} {\sigma'\sqrt  n}\right);\ovbmZ_{\lfloor nt\rfloor}>0\right]=
		\sum_{k=0}^{\lfloor nt\rfloor}\left(\sum_{\ell=0}^{+\infty}\P(\bm \tau^{+}_{\ell}=k)\right)\E\left[\varphi\left(\frac {\bm S_{\lfloor nt\rfloor-k}} {\sigma'\sqrt  n}\right);(\bm S_{i})_{i\in[\lfloor nt\rfloor-k]}> 0\right]\\
		=\sum_{k=0}^{\lfloor nt\rfloor}\left(\sum_{\ell=0}^{+\infty}\P(\bm \tau^{+}_{\ell}=k)\right)\E\left[\varphi\left(\frac {\bm S_{\lfloor nt\rfloor-k}} {\sigma'\sqrt  n}\right)\middle |(\bm S_{i})_{i\in[\lfloor nt\rfloor-k]}> 0\right]\P\left((\bm S_{i})_{i\in[\lfloor nt\rfloor-k]}> 0\right).
	\end{multline}
	We now focus on studying $\P(\bm \tau^{+}_{\ell}=k)$. We have the following result.
	
	\begin{lemma}\label{lem:ijrnfr9fnr39}
		As $k\to\infty$,
		\begin{equation}
			\P(\bm \tau^{+}_{1}=k)\sim\frac{\beta}{k^{3/2}},
		\end{equation}
		where
		\begin{equation}\label{eq:fbouewfboiwenfewf}
			\beta=\frac{1}{\sqrt{2\pi}}\left(\frac{1}{\sigma'}\frac{\alpha\theta^{-1}}{1-\theta}+\frac{1}{\sigma}\left(\alpha\theta^{-1}+\frac{\alpha\gamma^{-1}}{1-\theta}\right)\frac{1}{1-\gamma}\right)\approx 0.730268.
		\end{equation}
	\end{lemma}
	\begin{proof}
		Recall that $\ovbmZ_{0}=0$. Using Observation Remark~\ref{obs:alternating_excursions_strong}, we can write (see also Figure~\ref{fig:schema_prob_event_strong})
		\begin{align}\label{eq:fgfgfdagfdgagagfggd}
			&\P(\bm \tau^{+}_{1}=k)=\P(\ovbmZ_{k}=0,\ovbmZ_{1}=1,(\ovbmZ_{i})_{i\in[1,k-1]}\neq 0)+\P(\ovbmZ_{k}=0,(\ovbmZ_{i})_{i\in[1,k-1]}< 0)=\notag\\
			&\P(\ovbmZ_{1}=1)\sum_{s=1}^{k-2}\sum_{\substack{y\in\zgz\\
					y'\in\Z_{<0}}}\P(\ovbmZ_{s}=y,(\ovbmZ_{i})_{i\in[s]}> 0|\ovbmZ_{1}=1)
			\P(\ovbmZ_{s+1}-\ovbmZ_{s}=-y-1)\notag\\
			&\cdot \P(\ovbmZ_{k-1}=y',(\ovbmZ_{i})_{i\in[s+1,k-1]}< 0|\ovbmZ_{s+1}=-1)
			\P(\ovbmZ_{k}-\ovbmZ_{k-1}=-y')\notag\\
			&+\P(\ovbmZ_{k}=0,(\ovbmZ_{i})_{i\in[k-1]}< 0).
		\end{align}
		\begin{figure}[htbp]
			\centering
			\includegraphics[scale=1]{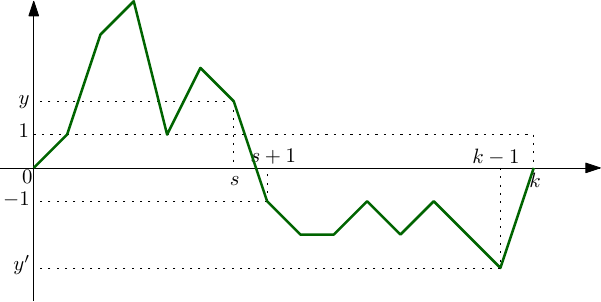}
			\caption{A schema for the event $\{\ovbmZ_{k}=0,\ovbmZ_{1}=1,(\ovbmZ_{i})_{i\in[1,k-1]}\neq 0\}$.\label{fig:schema_prob_event_strong}}
		\end{figure}

		Using Definition~\ref{defn:distrib_incr_coal_strong} and letting:
		\begin{itemize}
			\item  $(\bm S_i)_{i\geq 0}$ be a random walk started at zero at time zero, with step distribution equal to the distribution of $\obmy(1)-\obmy(0)$,
			\item  $(\bm S'_i)_{i\geq 0}$ be a random walk started at zero at time zero, with step distribution equal to the distribution of $-(\obmx(1)-\obmx(0))$,
			\item $(\bm S_i)_{i\geq 0}$ is independent of $(\bm S'_i)_{i\geq 0}$,
		\end{itemize}
		we can rewrite the previous expression as
		\begin{align}
			\P&(\bm \tau^{+}_{1}=k)\notag\\
			=&\P(\obmy(1)-\obmy(0)=1)\sum_{s=0}^{k-3}\sum_{\substack{y\in\zgz\\
					y'\in\Z_{<0}}}\P(1+\bm S_{s}=y,(1+\bm S_{i})_{i\in[s]}> 0)
			\P(\obmy(1)-\obmy(0)\leq -y)\notag\\
			\cdot\;&\P(-1+\bm S'_{k-2}=y',(-1+\bm S'_{i})_{i\in[s+1,k-2]}< 0)
			\P(-(\obmx(1)-\obmx(0))\geq -y')\label{eq:wefiuwebfgrtgrtgfewbfqewbofq}\notag\\
			+&\sum_{y'\in\Z_{<0}}\P(\bm S'_{k-2}=y',(\bm S'_{i})_{i\in[k-2]}< 0)\P(-(\obmx(1)-\obmx(0))\geq -y').
		\end{align}
		We now focus on
		$\sum_{s=0}^{k-3}\P(1+\bm S_{s}=y,(1+\bm S_{i})_{i\in[s]}> 0)
		\P(-1+\bm S'_{k-2}=y',(-1+\bm S'_{i})_{i\in[s+1,k-2]}< 0)$. From Lemma~\ref{lem:local_est_one_dim} we know that as $m\to \infty$,
		\begin{align}\label{eq:ewvifewvifewnfieowpnfpewi}
			&\P(1+\bm S_{m}=y,(1+\bm S_{i})_{i\in[m]}> 0)\sim\frac{c_1}{\sigma'\sqrt{2\pi}}\frac{\tilde{h}(y)}{m^{3/2}},\\
			&\P(-1+\bm S'_{m}=y',(-1+\bm S'_{i})_{i\in[m]}< 0)\sim\frac{c_2}{\sigma\sqrt{2\pi}}\frac{\tilde{h'}(y')}{m^{3/2}},\notag\\
			&\P(1+\bm S_{m}=y,(1+\bm S_{i})_{i\in[m]}> 0)\leq C\frac{\tilde{h}(y)}{m^{3/2}},\quad\text{for all}\quad y\in\Z_{> 0},\notag\\
			&\P(-1+\bm S'_{m}=y',(-1+\bm S'_{i})_{i\in[m]}< 0)\leq C\frac{\tilde{h'}(y')}{m^{3/2}},\quad\text{for all}\quad y'\in\Z_{< 0},\notag
		\end{align}
		where $\tilde{h}$ and $\tilde{h'}$ are the functions defined in \eqref{eq:efdbuebfwebfowe} for the walks $-\bm S_{m}$ and $-\bm S'_{m}$ respectively and
		\begin{equation}\label{eq:corrconstants}
			c_1=\frac{\E[-\bm{S}_{\bm N}]}{\sum_{z\in\zgz} \tilde{h}(z)\P(\bm{S}_1\leq -z)}\quad\text{ and }\quad c_2=\frac{\E[\bm{S}'_{\bm N'}]}{\sum_{z\in\Z_{<0}} \tilde{h}'(z)\P(\bm{S}'_1\geq -z)},
		\end{equation}
		with $\bm N\coloneqq\inf\{n>0|\bm{S}_n< 0\}$ and $\bm N'\coloneqq\inf\{n>0|\bm{S'}_n> 0\}$.
		Using Proposition~\ref{prop:asympt_estim_sum} and the estimates above, we have, as $k\to \infty$,
		\begin{multline}
			\sum_{s=0}^{k-3}\P(1+\bm S_{s}=y,(1+\bm S_{i})_{i\in[s]}> 0)
			\P(-1+\bm S'_{k-2}=y',(-1+\bm S'_{i})_{i\in[s+1,k-2]}< 0)\\
			\sim\frac{1}{\sqrt{2\pi}k^{3/2}}C(y,y'),
		\end{multline}
		where $C(y,y')$ is equal to
		\begin{multline}
			\frac{c_1\cdot\tilde{h}(y)}{\sigma'}\left(\sum_{s=0}^{\infty}\P(-1+\bm S'_{s}=y',(-1+\bm S'_{i})_{i\in[s]}< 0)\right)\\
			+\frac{c_2\cdot\tilde{h'}(y')}{\sigma}\left(\sum_{s=0}^{\infty}\P(1+\bm S_{s}=y,(1+\bm S_{i})_{i\in[s]}> 0)\right).
		\end{multline}
		Therefore from the uniform bounds in \eqref{eq:ewvifewvifewnfieowpnfpewi} and the fact that (we are using again Lemma~\ref{lem:local_est_one_dim})
		\begin{align}
			&\P(\bm S'_{k}=y',(\bm S'_{i})_{i\in[k-1]}< 0)\sim\frac{\P(-(\obmx(1)-\obmx(0))=-1)}{\sigma\sqrt{2\pi}}\frac{\tilde{h'}(y')}{k^{3/2}},\notag\\
			&\P(\bm S'_{k}=y',(\bm S'_{i})_{i\in[k-1]}< 0)\leq C \frac{\tilde{h'}(y')}{k^{3/2}},\quad\text{for all}\quad y'\in\Z_{< 0},
		\end{align}
		we can conclude from \eqref{eq:wefiuwebfgrtgrtgfewbfqewbofq} that
		\begin{multline}
			\P(\bm \tau^{+}_{1}=k)
			\sim\frac{1}{k^{3/2}}
			\frac{1}{\sqrt{2\pi}}\\
			\cdot\Bigg(\P(\obmy(1)-\obmy(0)=1)\sum_{\substack{y\in\zgz\\
					y'\in\Z_{< 0}}}
			\P(\obmy(1)-\obmy(0)\leq-y)
			\P(-(\obmx(1)-\obmx(0))\geq -y')C(y,y')\\
			+\P(-(\obmx(1)-\obmx(0))=-1)\sum_{y'\in\Z_{<0}}\frac{\tilde{h'}(y')}{\sigma}\P(-(\obmx(1)-\obmx(0))\geq -y')\Bigg)=\frac{\beta}{k^{3/2}}.
		\end{multline}
		We finally simplify the expression for the coefficient $\beta$ above, finding the formula in \eqref{eq:fbouewfboiwenfewf}. Note that 
		\begin{multline}
			\sum_{y\in\zgz}\sum_{s=0}^{\infty}\P(1+\bm S_{s}=y,(1+\bm S_{i})_{i\in[s]}> 0)\P(\obmy(1)-\obmy(0)\leq-y)\\
			=\sum_{s=0}^{\infty}\P(1+\bm S_{s+1}\leq 0,(1+\bm S_{i})_{i\in[s]}> 0)=1,
		\end{multline}
		because a symmetric random walk started at $1$ becomes eventually non-positive almost surely.
		Similarly,
		\begin{equation}
			\sum_{y'\in\Z_{<0}}\sum_{s=0}^{\infty}\P(-1+\bm S'_{s}=y',(-1+\bm S'_{i})_{i\in[s]}< 0)\P(-(\obmx(1)-\obmx(0))\geq-y')=1.
		\end{equation}
		Therefore, using the expressions of the constants $c_1$ and $c_2$ in \eqref{eq:corrconstants}, we obtain that
		\begin{multline}
			\beta=\frac{1}{\sqrt{2\pi}}
			\Bigg(\frac{\P(\obmy(1)-\obmy(0)=1)+\P(-(\obmx(1)-\obmx(0))=-1)}{\sigma}\E[\bm{S}'_{\bm N'}]\\
			+\frac{\P(\obmy(1)-\obmy(0)=1)}{\sigma'}\E[-\bm{S}_{\bm N}]\Bigg).
		\end{multline}
		Recalling that (see \eqref{eq:distristep_strongbaxter})
		\begin{align}
			&\P(\obmy(1)-\obmy(0)=1)=\alpha\theta^{-1},\quad\P(\obmy(1)-\obmy(0)=y)=\alpha\gamma^{-1}\theta^{-y},\quad\text{for all } y\in\Z_{<0}\notag\\
			&\P(-(\obmx(1)-\obmx(0))=-1)=\frac{\alpha\gamma^{-1}}{1-\theta},\quad\P(-(\obmx(1)-\obmx(0))=y')=\alpha\gamma^{y'},\quad\text{for all } y'\in\zgz,
		\end{align}
		and using Lemma~\ref{lem:local_est_one_dim_expect}, we obtain that $\E\left[-\bm S_{\bm N}\right]=\frac{1}{1-\theta}$ and $\E\left[\bm S'_{\bm N'}\right]=\frac{1}{1-\gamma}$ and so
		\begin{equation}
			\beta
			=\frac{1}{\sqrt{2\pi}}
			\left(\frac{1}{\sigma}\left(\alpha\theta^{-1}+\frac{\alpha\gamma^{-1}}{1-\theta}\right)\frac{1}{1-\gamma}
			+\frac{1}{\sigma'}\frac{\alpha\theta^{-1}}{1-\theta}\right)
		\end{equation}
		as claimed in \eqref{eq:fbouewfboiwenfewf}.
	\end{proof}

	\begin{corollary}\label{cor:asympt_estim_rew_strong}
		As $k\to+\infty$,
		\begin{equation}
			\sum_{\ell=0}^{+\infty}\P(\bm \tau^{+}_{\ell}=k)\sim\frac{1}{\beta}\frac{1}{2\pi\sqrt k}\;.
		\end{equation}
	\end{corollary}
	\begin{proof}
		Thanks to\footnote{This corresponds to \cite[Theorem B]{MR1440141} where the term in the right-hand side of Equation (1.10) should be replaced by $\frac{1}{\Gamma(1-\alpha)\Gamma(\alpha)}$} \cite[Theorem B]{MR1440141}, it is enough have to check that
		$$\sup_{n\in\Z_{\geq 0}}\frac{n\cdot \P(\bm \tau^{+}_{1}=n)}{\P(\bm \tau^{+}_{1}>n)}<\infty,$$
		and this follows from Lemma~\ref{lem:ijrnfr9fnr39}.
	\end{proof}
	
	With this result in our hands we can now continue to estimate \eqref{eq:webhfyuwheribfiuewbfoweuf_strong}.
	We start by defining for $1\leq k\leq \lfloor nt\rfloor-2$ and $s\in[\frac k n,\frac{k+1}{n})$,
	\begin{equation}
		f^+_n(s)\coloneqq n\left(\sum_{\ell=0}^{+\infty}\P(\bm \tau^{+}_{\ell}=k)\right)\E\left[\varphi\left(\frac {\bm S_{\lfloor nt\rfloor-k}} {\sigma'\sqrt  n}\right)\middle|(\bm S_{i})_{i\in[\lfloor nt\rfloor-k]}> 0\right]\P\left((\bm S_{i})_{i\in[\lfloor nt\rfloor-k]}> 0\right),
	\end{equation}
	and $f^+_n(s)\coloneqq 0$ on $(0,\frac 1 n)\cup[\frac{\lfloor nt\rfloor-1}{n},t)$. We have that	
	\begin{multline}\label{eq:fbwebfowenfpwenmfpew}
		\E\left[\varphi\left(\frac {\ovbmZ_{\lfloor nt\rfloor}} {\sigma'\sqrt  n}\right);\ovbmZ_{\lfloor nt\rfloor}>0\right]\\
		=\sum_{k=0}^{\lfloor nt\rfloor}\left(\sum_{\ell=0}^{+\infty}\P(\bm \tau^{+}_{\ell}=k)\right)\E\left[\varphi\left(\frac {\bm S_{\lfloor nt\rfloor-k}} {\sigma'\sqrt  n}\right)\middle |(\bm S_{i})_{i\in[\lfloor nt\rfloor-k]}> 0\right]\P\left((\bm S_{i})_{i\in[\lfloor nt\rfloor-k]}> 0\right)\\
		=\int_0^tf_n^+(s) ds+O(1/\sqrt n).
	\end{multline}
	Indeed each term in the sum in \eqref{eq:fbwebfowenfpwenmfpew} is $O(1/\sqrt n)$; this follows from Lemma~\ref{lem:local_est_one_dim}, Corollary~\ref{cor:asympt_estim_rew_strong}, and the fact that\footnote{As pointed out in \cite[Section 2.2]{MR4105264} the sequence $(\tau^{+}_{\ell})_{\ell\in\Z_{\geq 0}}$ is a strictly increasing random walk on $\Z_{\geq 0}$ with i.i.d.\ increments distributed as $\tau^{+}_{1}$. Thus, its potential $\sum_{\ell=0}^{+\infty}\P(\bm \tau^{+}_{\ell}=k)$ is finite for all $k\in\Z_{\geq 0}$.} $\sum_{\ell=0}^{+\infty}\P(\bm \tau^{+}_{\ell}=k)<\infty$ for all $k\in\Z_{\geq 0}$.
	Since the walk $\frac {\bm S_{\lfloor nt\rfloor}} {\sigma'\sqrt  n}$ converges in distribution to a Brownian meander (see for instance \cite{MR803677}), we have that
	\begin{equation}\label{eq:buibdqweobdpqwdbnwq}
		\E\left[\varphi\left(\frac {\bm S_{\lfloor nt\rfloor}} {\sigma'\sqrt  n}\right)\middle|(\bm S_{i})_{i\in[\lfloor nt\rfloor]}> 0\right]\sim\frac{1}{t}\int_0^{+\infty}\varphi(u)ue^{-u^2/2t}du,
	\end{equation}
	Using the latter estimate and the fact that from Lemma~\ref{lem:local_est_one_dim},
	\begin{equation}
		\P\left((\bm S_{i})_{i\in[n]}> 0\right)\sim 2\frac{\P(\obmy(1)-\obmy(0)=1)\E[-\bm{S}_{\bm N}]}{\sqrt{2\pi}\sigma'} \frac{1}{\sqrt{n}}=2\frac{\alpha\theta^{-1}}{\sqrt{2\pi}\sigma'(1-\theta)} \frac{1}{\sqrt{n}}, 
	\end{equation}
	together with Corollary~\ref{cor:asympt_estim_rew_strong}, we obtain that for any $s\in(0,t)$,
	\begin{multline}
		f^+_n(s)\sim n\frac{1}{\beta}\frac{1}{2\pi\sqrt n\sqrt s}\frac{1}{t-s}\int_0^{+\infty}\varphi(u)ue^{-u^2/2(t-s)}du\cdot 2\frac{\alpha\theta^{-1}}{\sqrt{2\pi}\sigma'(1-\theta)}\frac{1}{\sqrt n\sqrt{t-s}}\\
		=\frac{\alpha\theta^{-1}}{(1-\theta)\sqrt{2\pi}\sigma'\beta}\frac{1}{\pi\sqrt{s(t-s)}(t-s)}\int_0^{+\infty}\varphi(u)ue^{-u^2/2(t-s)}du.
	\end{multline}
	Moreover, recalling that $\varphi$ is Lipschitz continuous with compact support, we have for $1\leq k\leq \lfloor nt\rfloor-2$ and $s\in[\frac k n,\frac{k+1}{n})$ the bound
	\begin{equation}
		\left|f^+_n(s)\right|\leq \frac{Cn}{\sqrt{\lfloor ns\rfloor\cdot (\lfloor nt\rfloor-\lfloor ns\rfloor)}}\leq \frac{C'}{\sqrt{s(t-s)}}.
	\end{equation}
	Therefore by dominated convergence we can conclude that
	\begin{multline}\label{eq:weijbfuiwefbewonfpewinf}
		\E\left[\varphi\left(\frac {\ovbmZ_{\lfloor nt\rfloor}} {\sigma'\sqrt  n}\right);\ovbmZ_{\lfloor nt\rfloor}>0\right]\\
		\xrightarrow{n\to\infty}
		\frac{\alpha\theta^{-1}}{(1-\theta)\sqrt{2\pi}\sigma'\beta}\int_{0}^{+\infty}\frac{1}{\pi\sqrt{s(t-s)}(t-s)}\int_0^{+\infty}\varphi(u)ue^{-u^2/2(t-s)}\;du\;ds\\
		=\frac{\alpha\theta^{-1}}{(1-\theta)\sqrt{2\pi}\sigma'\beta}\int_{0}^{+\infty}\varphi(y)\frac{2e^{-y^2/2t}}{\sqrt{2\pi t}}dy,
	\end{multline}
	where in the last equality we used some standard integral computations. Using the expression for $\beta$ in \eqref{eq:fbouewfboiwenfewf} we conclude that
	\begin{equation}
		q=\frac{\alpha\theta^{-1}}{(1-\theta)\sqrt{2\pi}\sigma'\beta}
		=\frac{\frac{\alpha\theta^{-1}}{(1-\theta)\sigma'}}{\frac{1}{\sigma'}\frac{\alpha\theta^{-1}}{1-\theta}+\frac{1}{\sigma}\left(\alpha\theta^{-1}+\frac{\alpha\gamma^{-1}}{1-\theta}\right)\frac{1}{1-\gamma}}
		=\frac{1}{1+\frac{\sigma'}{\sigma}\left(\frac{\theta+\gamma-\gamma\theta}{\gamma(1-\gamma)}\right)}.
	\end{equation}
	The formula for $q$ given in \eqref{eq:param_q_strong} follows substituting the expressions for $\sigma$ and $\sigma'$ given in \eqref{eq:computations_parm_strong} and using the relations for the parameters $\theta$ and $\gamma$ given in \eqref{eq:system_strong}.
	
	\medskip
	
	The convergence for the second expectation in \eqref{eq:bkjfkwebfoewnfpoew} follows using the same type of estimates and therefore we omit the details. 
\end{proof}

\subsubsection{The conditioned scaling limit}
\label{sec:cond_conv}
We now look at the conditioned case. The arguments in this section are quite similar to the ones in \cite[Section 4.3]{borga2020scaling} and so we postpone the proofs of Propositions \ref{prop:jfwvbvouvwefuow} and \ref{prop:jfbeouwfboufwe} to Section~\ref{sect:techproofs}.

\medskip

For all $n\in\Z_{> 0}$, let $\bm W_n=(\bm W_n(i))_{i\geq 1}=(\bm X_n(i),\bm Y_n(i))_{i\geq 1}$ be a uniform walk in $\mathcal W^n_{Sb}$ and\footnote{For convenience, we denote here the time variable of the walks of coalescent-walk processes between parenthesis.} $\bm Z_n =\{\bm \Ztp_n(\cdot)\}_{t\in [n]} =\wScbp(\bm W_n)$ be the corresponding uniform coalescent-walk process in $\CCCC_{Sb}$. Also here, for convenience, we set $\ovbmZ^{(j)}_n(i) = 0$ for $i,j \in \zgz$, $i<j$.

For all $n\geq 1, u\in (0,1)$, recall the function ${\conti W}_n:[0,1]\to \R^2$ defined in \eqref{eq:resc_walk} and let $\conti{Z}^{(u)}_n:(0,1)\to\R$ be the continuous process that interpolates the following points defined for all $k\in [n]$,
\begin{equation}\label{eq:wifbjwbfowebn}
	\conti{Z}^{(u)}_{n}\left(\frac kn\right) =
	\begin{cases}
		\frac {{\bm Z}^{(\lceil nu\rceil)}_{n}(k)} { \sigma'\sqrt {n}},\quad&\text{when}\quad{\bm Z}^{(\lceil nu\rceil)}_{n}(k)\geq 0,\\
		\frac {{\bm Z}^{(\lceil nu\rceil)}_{n}(k)} {\sigma\sqrt{ n}},\quad&\text{when}\quad{\bm Z}^{(\lceil nu\rceil)}_{n}(k)< 0,
	\end{cases}
\end{equation}
where $\sigma$ and $\sigma'$ are defined in \eqref{eq:computations_parm_strong}.

\medskip

In the next proposition we state a scaling limit result for a one-dimensional walk in the conditioned coalescent-walk processes introduced in \eqref{eq:wifbjwbfowebn}. This is a consequence of Proposition~\ref{prop:scaling_strong_walk}, Theorem~\ref{thm:fvwuofbgqfipqhfqfpq_strongb}, and absolute continuity arguments between correlated Brownian excursions and correlated Brownian motions.

\begin{proposition}\label{prop:jfwvbvouvwefuow}
	Let $u\in (0,1)$. 
	The following joint convergence in $C([0,1],\R)^2\times C([0,1),\R)$ holds: 
	\begin{equation}
		\label{eq:coal_con_cond}
		\left({\conti W}_n,\conti{Z}^{(u)}_n \right)
		\xraninf{d}
		\left(\conti E_{\rho},\conti Z_{\rho,q}^{(u)}\right),
	\end{equation}
	where $\conti Z_{\rho,q}^{(u)}$ is the strong solution of the SDE in \eqref{eq:random_skew_function} and the parameters $\rho,q$ are the same as in Theorem~\ref{thm:fvwuofbgqfipqhfqfpq_strongb}.
\end{proposition}

We also generalize the previous result for a countable number of one-dimensional walks in the conditioned coalescent-walk processes introduced in \eqref{eq:wifbjwbfowebn}, chosen with uniform starting points. This result is the key-step for proving Theorem~\ref{thm:strong-baxter} in the next section.

\begin{proposition}
	\label{prop:jfbeouwfboufwe}
	Consider a sequence $(\bm u_i)_{i\in\zgz}$ of i.i.d.\ uniform random variables on $[0,1]$, independent of all other variables below. The following joint convergence in the space $\czord\times \mathcal C([0,1),\R)^{\mathbb Z_{>0}}$ holds:
	\begin{equation}
		\left(
		{\conti W}_n,\left({\conti Z}^{(\bm u_i)}_n\right)_{i\in\zgz}
		\right)
		\xraninf{d}
		\left( \conti E_{\rho},\left(\conti Z_{\rho,q}^{(\bm u_i)}\right)_{i\in\zgz}
		\right).
	\end{equation}
\end{proposition}

\subsection{The permuton limit of strong-Baxter permutations}\label{sect:permlim}
We now prove Theorem~\ref{thm:strong-baxter}, which follows from Proposition~\ref{prop:jfbeouwfboufwe}.

\medskip

We recall some notions and results related to permuton limits that we need in this section\footnote{Recall that a complete introduction to the theory of permutons can be found in \cite[Section 2.1]{borga2021random}.}. The space of permutons $\mathcal M$, endowed with the topology of weak convergence of measures, is a compact space. Furthermore, $\mathcal M$ is metrizable by the metric $d_\square$ that is defined as follows: 
\begin{equation}
	d_\square(\mu,\mu')
	\coloneqq
	\sup_{R\in\mathcal R}| \mu(R) - \mu'(R) |,\quad\text{for every pair of permutons } (\mu,\mu'),
\end{equation}
where $\mathcal R$ denotes the set of rectangles contained in $[0,1]^2.$

We also define the permutation induced by a collection of $k$ points in the square $[0,1]^2$. Consider a sequence of $k$ points $(X,Y)=((x_{1},y_{1}),\dots, (x_{k},y_{k}))$ in $[0,1]^2$ and assume that the $x$ and $y$ coordinates are distinct. 
We denote by $\left((x_{(1)},y_{(1)}),\dots, (x_{(k)},y_{(k)})\right)$ the \emph{$x$-reordering} of $(X,Y)$,
that is the only reordering of $((x_{1},y_{1}),\dots, (x_{k},y_{k}))$ such that
$x_{(1)} < \cdots < x_{(k)}$.
In this way, the values $(y_{(1)} , \ldots , y_{(k)})$ have the same
relative order as the values of a unique permutation of size $k$. We call this permutation the \emph{permutation induced by} the sequence $(X,Y)$.

Let $\mu\in \mathcal M$ be a permuton and $((\bm X_{i},\bm Y_{i}))_{i\in \zgz}$ be an i.i.d.\ sequence with distribution $ \mu$. We denote by $\Perm_{k}(\mu)$ the random permutation induced by the sequence $((\bm X_{i},\bm Y_{i}))_{i \in [k]}$.

\begin{lemma}[{\cite[Lemma 2.3]{bassino2017universal}}]
	\label{lem:fibwebvfoubqe}
	There exists  a constant $k_0$ such that whenever $k>k_0$, 
	\begin{equation}
		\P\left(
		d_\square(\mu_{\Perm_k(\bm{\nu})},\bm{\nu})
		\geq 16 k^{-1/4}\right)
		\leq \frac{1}{2} e^{-\sqrt{k}}, \quad \text{for every random permuton $\bm{\nu}$.}
	\end{equation}
\end{lemma}

We can now prove Theorem~\ref{thm:strong-baxter}.

\begin{proof}[Proof of Theorem~\ref{thm:strong-baxter}]
	Since from Proposition~\ref{prop:jfbeouwfboufwe} we know that $\left({\conti W}_n,\left({\conti Z}^{(\bm u_i)}_n \right)_{i\in\zgz}\right)$ is a convergent sequence of random variables converging to $\left( \conti{E}_\rho,\left({\conti Z}^{(\bm u_i)}_{\rho,q}\right)_{i\in\zgz}\right)$ and the space of permutons $\mathcal M$ is compact, using Prokhorov's theorem, both $\left({\conti W}_n,\left({\conti Z}^{(\bm u_i)}_n\right)_{i\in\zgz}\right)$ and $\mu_{\bm \sigma_n}$ are tight sequences of random variables.
	Assume now that on a subsequence it holds that
	\begin{equation}\label{eq:wofhweihwep}
		\left(
		{\conti W}_n,\left({\conti Z}^{(\bm u_i)}_n\right)_{i\in\zgz},\mu_{\bm \sigma_n}
		\right)
		\xraninf{d}
		\left( \conti{E}_\rho,\left({\conti Z}^{(\bm u_i)}_{\rho,q}\right)_{i\in\zgz},\bm \nu
		\right).
	\end{equation}
	In order to complete the proof we need to show that $\bm \nu\stackrel{d}{=}\bm \mu_{\rho,q}$.
	Using Skorokhod's theorem, we can assume that the convergence in \eqref{eq:wofhweihwep} holds almost surely. In particular, almost surely, $\mu_{\bm \sigma_n} \to \bm \nu$ in the space of permutons $\mathcal M$, and for every $i\in\zgz$, $\conti Z_{n}^{(\bm u_i)} \to \conti Z_{\rho,q}^{(\bm u_i )}$ uniformly on $[0,1)$, where we recall that $(\bm u_i)_{i\in\zgz}$ are i.i.d.\ uniform random variables on $[0,1]$.
	
	We now fix $k\in\zgz$ and denote by $\bm \rho_n^k$ the pattern induced by the permutation $\bm \sigma_n$ on the set of indices $\left\{ \lceil n \bm u_{1} \rceil , \ldots , \lceil n \bm u_{k} \rceil\right\}$ ($\bm \rho_n^k$ remains undefined whenever two indices are equal; this event has probability tending to zero). Thanks to the uniform convergence mentioned above, and recalling the definition of $\conti{Z}^{(u)}_{n}\left(\frac kn\right)$ given in \eqref{eq:wifbjwbfowebn}, for all $1 \leq i < j \leq k$ it holds that
	\begin{equation}
		\sgn\left({\bm Z}_n^{( \intnui \wedge \intnuj)}(\intnui \vee \intnuj)\right)
		\xraninf{}
		\sgn(\conti Z_{\rho,q}^{(\bm u_i \wedge \bm u_j)}(\bm u_j \vee \bm u_i))\quad \text{ a.s.}
	\end{equation}
	We highlight that $\sgn(\cdot)$ is not continuous at zero, but from \cite[Lemma 3.2]{borga2021skewperm} we know that $\conti Z_{\rho,q}^{(\bm u_i \wedge \bm u_j)}(\bm u_j \vee \bm u_i)$ is almost surely nonzero and so a continuity point of $\sgn(\cdot)$.
	By Proposition~\ref{prop:patternextraction}  and \cite[Lemma 3.2]{borga2021skewperm}, this means that $\bm \rho_n^k \xraninf{} \Perm_k(\bm \mu_{\rho,q})$, where we recall that $\Perm_k(\bm \mu_{\rho,q})$ denotes the permutation induced by $(\bm u_i, \varphi_{\conti Z_{\rho,q}}(\bm u_i))_{i \in [k]}$ (recall the definition of $\bm \mu_{\rho,q}$ given in Definition~\ref{defn:ibfwebfowfpwenfiweop}).
	Now, from Lemma~\ref{lem:fibwebvfoubqe}, for $k$ large enough it holds that
	\begin{equation}\label{eq:first_perm_bound}
		\P\left(
		d_{\square}(\mu_{\bm \rho_n^k},\mu_{\bm \sigma_n})
		> 16k^{-1/4}\right)
		\leq \frac{1}{2} e^{-\sqrt k} + O(n^{-1}),
	\end{equation}
	where the $O(n^{-1})$ term is a consequence of the fact that $\bm \rho^k_n$ might be undefined.
	Since we know that $\bm \rho_n^k \xraninf{} \Perm_k(\bm \mu_{\rho,q})$ and also that $\mu_{\bm \sigma_n} \to \bm \nu$, then we have that
	\begin{equation}
		\P\left(
		d_{\square}(\mu_{\Perm_k(\bm \mu_{\rho,q})},\widetilde {\bm \mu})
		> 16k^{-1/4}\right)
		\leq \frac{1}{2} e^{-\sqrt k}.
	\end{equation}
	As a consequence, $\mu_{\Perm_k(\bm \mu_{\rho,q})}\xrightarrow[k\to\infty]{P}\widetilde {\bm \mu}$. Using again Lemma~\ref{lem:fibwebvfoubqe}, we have that $\mu_{\Perm_k(\bm \mu_{\rho,q})} \xrightarrow[k\to\infty]{P}\bm  \mu_{\rho,q}$. The latter two limits imply that $\widetilde {\bm \mu} \stackrel{a.s.}{=} \bm \mu_{\rho,q}$ and this is enough to complete the proof.
\end{proof}

\section{The case of semi-Baxter permutations}

In this section we show that both the combinatorial constructions developed in Section~\ref{sect:disc_obj_strong} and the probabilistic techniques used in Section~\ref{sect:prob_part} are robust: we apply them to another family of permutations, called semi-Baxter permutations, proving Theorem~\ref{thm:semi-baxter}.

\subsection{Discrete objects associated with semi-Baxter permutations}\label{sect:jiewvbfiuwebvof}

Recall the definition of semi-Baxter permutations from Definition~\ref{defn:semi-b}.
Also these permutations have been enumerated using $\Z^2$-labeled generating trees and therefore they can be bijectively encoded by a specific family of two-dimensional walks.

The goal of this section, as done in Section~\ref{sect:disc_obj_strong} in the case of strong-Baxter permutations, is to specify the maps $\pw$ and $\cpp$ introduced in Section~\ref{sect:defn_discrete} to semi-Baxter permutations ($\mathcal{S}_{sb}$) and the corresponding families of walks $(\mathcal{W}_{sb})$ and coalescent-walk processes ($\CCCC_{sb}$). While doing that, we further introduce a third map $\rews$ (different from $\wScbp$) between walks and coalescent-walk processes. Thus, we are going to properly define the following new diagram
\begin{equation}\label{eq:diagrm_semi_baxter}
	\begin{tikzcd}
		\mathcal{S}_{sb} \arrow{r}{\pw}  & \mathcal{W}_{sb} \arrow{d}{\rews} \\
		& \CCCC_{sb} \arrow{ul}{\cpp}
	\end{tikzcd}
\end{equation}
and again to show that this is a commutative diagram of bijections (see Theorem~\ref{thm:The_diagram_commutes} below).

\bigskip

We highlight that the strategy used in this section is step-by-step the same as the one used for strong-Baxter permutations. Nevertheless, we need to repeat almost all the combinatorial constructions given in Section~\ref{sect:disc_obj_strong} and to keep track of all the various little modifications compared to the case of strong-Baxter permutations. These little modifications will then be fundamental in Section~\ref{sect:prob_part} to determine the correct values of the parameters $\rho$ and $q$ for the limiting skew Brownian permuton determined by semi-Baxter permutations. As already mentioned in the introduction, semi-Baxter permutations will present a different interesting phenomenon compared to the case of strong-Baxter permutations. This phenomenon is explained in Remark~\ref{rem:traj}.

\subsubsection{Succession rule for semi-Baxter permutations and a corresponding family of two-dimensional walks}\label{sect:semi_bac_obj}

Recall Definition~\ref{defn:appending}. We will adopt again the following useful convention. Given a semi-Baxter permutation $\pi\in \mathcal{S}_{sb}^n$ with $x+1$ active sites smaller than or equal to $\pi(n)$ and $y+1$ active sites greater than $\pi(n)$, we write
$$\text{AS}(\pi)=\{s_{-x}<\dots<s_0\}\cup\{s_{1}<\dots<s_{y+1}\},$$
where the first set corresponds to the $x+1$ active sites smaller than or equal to $\pi(n)$ and the second set corresponds to the $y+1$ active sites greater than $\pi(n)$.

\medskip

In \cite[Proposition 3]{MR3882946} it was shown that the generating tree for semi-Baxter permutations can be defined by the following succession rule\footnote{Note that the succession rule in our paper is obtained from the succession rule in \cite[Proposition 3]{MR3882946} by shifting all the labels by a factor $(-1,-1)$. This choice is more convenient for our purposes.}:
\begin{equation}
	\begin{cases} 
		\text{Root label}: (0,0) \\
		(h,k)\to\begin{cases}
			(0,k+1),(1,k+1),\dots,(h,k+1),\\ 
			(h+k+1,0),(h+k,1)\dots,(h+1,k),
		\end{cases}
		\quad \text{for all }h,k\geq 0.
	\end{cases}
	\label{eq:laruletadesuccession_1423_4123}
\end{equation}
As in the case of strong-Baxter permutations, the $\Z^2_{\geq 0}$-valued statistic that determines this succession rule is defined on every permutation $\sigma\in\mathcal{S}_{sb}$ by
$$\Big(\# \{m\in\text{AS}(\sigma)|m\leq\sigma(n)\}-1\;,\;\# \{m\in\text{AS}(\sigma)|m>\sigma(n)\}-1\Big).$$

Using the strategy described in Section~\ref{sect:bij_walk_perm}, we can define a bijection $\pw$ between semi-Baxter permutations and the set of two-dimensional walks in the non-negative quadrant, starting at $(0,0)$, with increments in
\begin{equation}\label{eq:semi_bax_increm}
	I_{sb}\coloneqq\{(-i,1): i\geq 0\}\cup\{(i,-i+1): i\geq 1\}.
\end{equation}
We denote with $\mathcal{W}_{sb}$ the set of two-dimensional walks in the non-negative quadrant, starting at $(0,0)$, with increment in $I_{sb}$.

\medskip

We now investigate the relations between the increments of a walk $W\in\mathcal{W}^n_{sb}$ and the active sites of the sequence of permutations:

$$\Big(\pw^{-1}((W_i)_{i\in[1]}),\pw^{-1}((W_i)_{i\in[2]}),\dots,\pw^{-1}((W_i)_{i\in[n]})\Big).$$

First of all, it holds that $\pw^{-1}((W_i)_{i\in[1]})$ is the unique permutation of size 1 and its active sites are 1 and 2. Now assume that for some $m<n$, $W_m=(x,y)\in\Z^2_{\geq 0}$ and  $\pw^{-1}((W_i)_{i\in[m]})=\pi$. By definition, $\pi$ has $x+1$ active sites smaller or equal to $\pi(n)$ and $y+1$ active sites greater than $\pi(n)$, i.e.,
$$\text{AS}(\pi)=\{s_{-x}<\dots<s_0\}\cup\{s_{1}<\dots<s_{y+1}\}.$$
We now distinguish two cases (for a proof of the following results see the proof of \cite[Proposition 3]{MR3882946}, compare also with Figure~\ref{fig:schema_perm}):
\begin{itemize}
	\item \textbf{Case 1:} $W_{m+1}-W_m=(i,-i+1)$ for some $i\in [y+1]$. 
	
	\noindent In this case $\pw^{-1}((W_i)_{i\in[m+1]})=\pi^{*s_{i}}$ and the active sites of $\pi^{*s_{i}}$ are 
	$$\{s_{-x}<\dots<s_0<s_{1}<\dots <s_{i}\}\cup\{s_{i}+1,s_{i+1}+1<\dots<s_{y+1}+1\}.$$

	\item \textbf{Case 2:} $W_{m+1}-W_m=(-i,1)$ for some $i\in \{0\}\cup[x]$. 
	
	\noindent In this case $\pw^{-1}((W_i)_{i\in[m+1]})=\pi^{*s_{-i}}$ and the active sites of $\pi^{*s_{-i}}$ are 
	$$\{s_{-x}<\dots<s_{-i}\}\cup\{s_{-i}+1<s_{1}+1<\dots<s_{y+1}+1\}.$$
\end{itemize} 

An example of the construction above is given in Figure~\ref{fig:schema_perm}.

\begin{figure}[htbp]
	\centering
	\includegraphics[width=\linewidth]{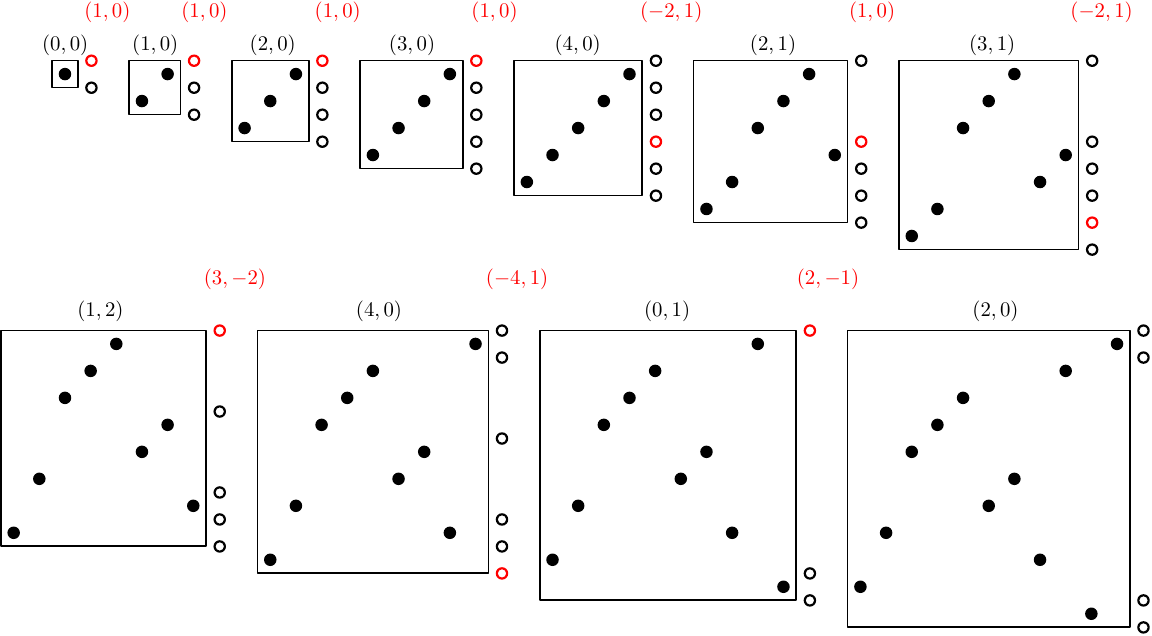}\\
	\caption{We consider the walk $W\in\mathcal{W}^{11}_{sb}$ given by the eleven black $\Z^2_{\geq 0}$-labels in the picture. The increments $W_{m+1}-W_{m}$ of the walk $W$ are written in red between two consecutive diagrams. For each black label $W_m$ we draw the diagram of the corresponding permutation $\pw^{-1}((W_i)_{i\in[m]})$. On the right-hand side of each diagram we draw with small circles the active sites of the permutation and we highlight in red the site that will be activated by the corresponding red increment $W_{m+1}-W_m$.  \label{fig:schema_perm}}
\end{figure}

\subsubsection{A coalescent-walk process for semi-Baxter permutations}\label{sect:coal_proc_for_semi_baxt}

We define a family of coalescent-walk processes driven by a set of two-dimensional walks that contains $\mathcal{W}_{sb}$. 
We fix a (finite or infinite) interval $I$ of $\Z$. 
Let $\mathfrak{W}_{sb}(I)$ denote the set of two-dimensional walks indexed by $I$, with increments in $I_{sb}$, and considered up to an additive constant.

\begin{definition}\label{eq:distrib_incr_coal_2}
	Let $W\in\mathfrak{W}_{sb}(I)$. The \emph{coalescent-walk process associated with} $W$
	is the family of walks $\rews(W) = \{\Ztp\}_{t\in I}$, defined for $t\in I$ by $\Ztp_t=0,$ and for all $\ell\geq t$ such that $\ell+1 \in I$,
	\begin{itemize}
		\item \textbf{Case 1:} $W_{\ell+1}-W_\ell=(i,-i+1)$ for some $i\geq 1$. 
		
		\begin{equation}
			\Ztp_{\ell+1}=
			\begin{cases}
				\Ztp_{\ell}-i+1, &\quad\text{if}\quad \Ztp_{\ell}-i+1>0 \text{ (and so } \Ztp_{\ell}\geq0),\\
				\Ztp_{\ell}-i,&\quad\text{otherwise}.
			\end{cases}
		\end{equation} 
		
		\item \textbf{Case 2:} $W_{\ell+1}-W_\ell=(-i,1)$ for some $i\geq 0$. 
		
		\begin{equation}
			\Ztp_{\ell+1}=
			\begin{cases}
				\Ztp_{\ell}+i,&\quad\text{if}\quad \Ztp_{\ell}<0\text{ and }\Ztp_{\ell}+i<0,\\
				1,&\quad\text{if}\quad \Ztp_{\ell}<0\text{ and }\Ztp_{\ell}+i\geq 0,\\
				\Ztp_{\ell}+1, &\quad\text{otherwise}.
			\end{cases}
		\end{equation} 
	\end{itemize} 
\end{definition}

$\rews$ is a mapping form $\mathfrak{W}_{sb}(I)$ to $\Coals(I)$ .We set $\CCCC_{sb}= \rews(\mathcal W_{sb})$. Two examples, one for a walk in $\mathfrak{W}_{sb}(I)$ and one for a walk in $\mathcal W_{sb}$, can be found in Figure~\ref{fig:didwveiwvedbwedi} and Figure~\ref{fig:ievdiuwbduwob}. We also give the following equivalent definition for later convenience.

\begin{figure}[ht]
	\centering
	\includegraphics[width=\linewidth]{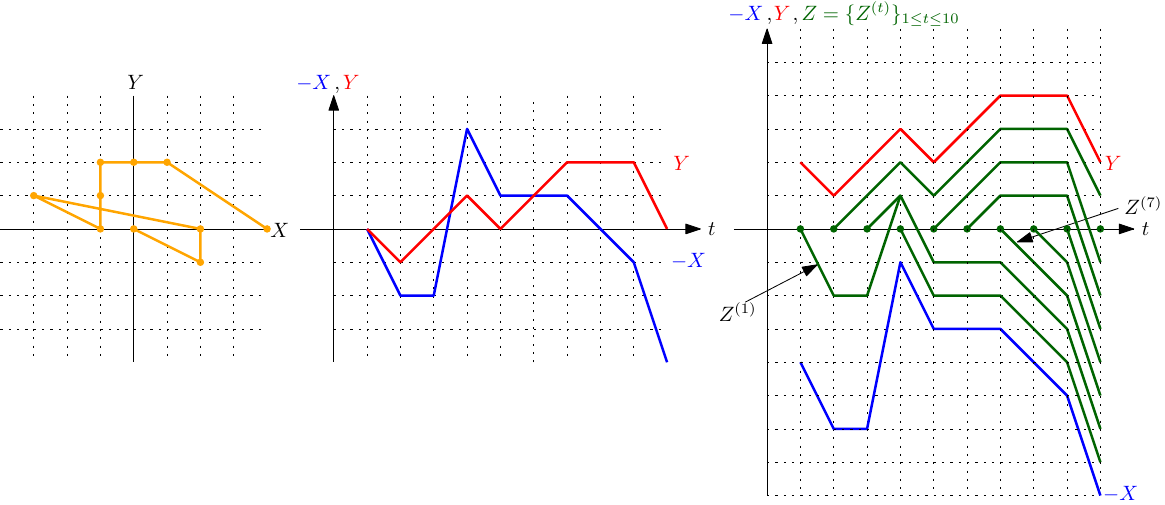}\\
	\caption{We explain the construction of a coalescent-walk process.
		\textbf{Left:} A two-dimensional walk $W=(W_t)_{t\in[10]}=(X_t,Y_t)_{t\in[10]}\in \mathfrak{W}_{sb}([10])$.
		\textbf{Middle:} The two marginals $-X$ (in blue) and $Y$ (in red).
		\textbf{Right:} The two marginals are shifted and the ten walks of the coalescent-walk process are constructed in green. \label{fig:didwveiwvedbwedi}}
\end{figure}

\begin{figure}[ht]
	\centering
	\includegraphics[scale=.7]{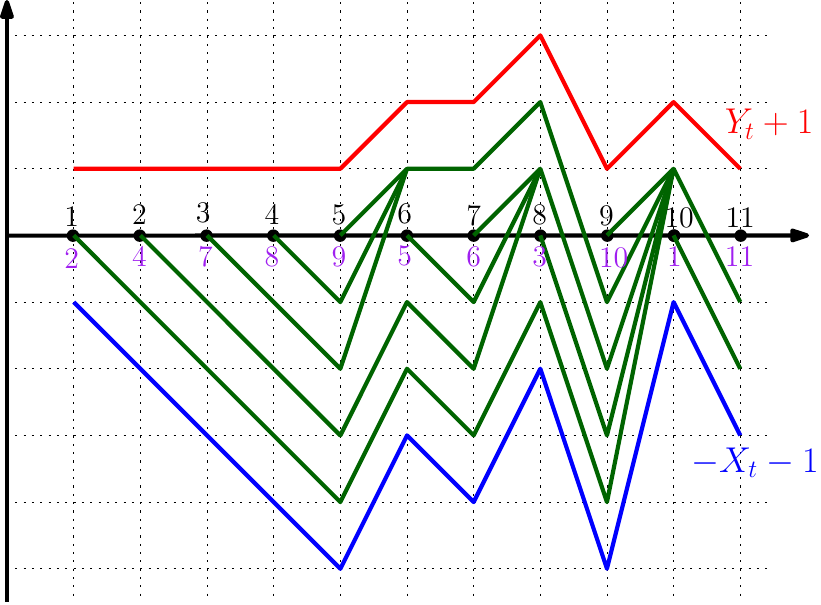}\\
	\caption{The coalescent walk process $\rews(W)$ for the walk $W$ considered in Figure~\ref{fig:schema_perm}. In purple we plot the corresponding semi-Baxter permutation $\cpp((\wScbp(W)))$. Note that the latter permutation is equal to the permutation $\pw^{-1}(W)$ obtained in the last diagram in Figure~\ref{fig:schema_perm}.
		\label{fig:ievdiuwbduwob}}
\end{figure}

\begin{definition}\label{eq:distrib_incr_coal}
	Let $W\in\mathfrak{W}_{sb}(I)$ and denote by $W_t = (X_t,Y_t)$ for $t\in I$. The \emph{coalescent-walk process associated with} $W$
	is the family of walks $\rews(W) = \{\Ztp\}_{t\in I}$, defined for $t\in I$ by $\Ztp_t=0,$ and for all $\ell\geq t$ such that $\ell+1 \in I$,
	\begin{equation}\label{eq:refrom_coal_semib}
		\Ztp_{\ell+1}=
		\begin{cases}
			\Ztp_{\ell}+(Y_{\ell+1}-Y_{\ell}), &\quad\text{if}\quad \Ztp_{\ell}\geq0\text{ and }\Ztp_{\ell}+(Y_{\ell+1}-Y_{\ell})>0,\\
			\Ztp_{\ell}-(X_{\ell+1}-X_{\ell}),&\quad\text{if}\quad 
			\begin{cases}
				\Ztp_{\ell}\geq0\text{ and }\Ztp_{\ell}+(Y_{\ell+1}-Y_{\ell})\leq0,\\
				\Ztp_{\ell}<0\text{ and }\Ztp_{\ell}-(X_{\ell+1}-X_{\ell})<0,
			\end{cases}\\
			1,&\quad\text{if}\quad \Ztp_{\ell}<0\text{ and }\Ztp_{\ell}-(X_{\ell+1}-X_{\ell})\geq 0.
		\end{cases}
	\end{equation} 
\end{definition}

\begin{remark}\label{obs:never_zero}
	By definition, for all $t\in I$, $\Ztp_{\ell}\neq 0$ for all $\ell>t$.
\end{remark}

\begin{remark}\label{obs:alternating_excursions}
	By definition, the process $(\Ztp_{\ell})_{\ell\geq t}$ starts at zero at time $t$ and then it can be  decomposed in alternating strictly positive and strictly negative excursions (after a strictly positive excursion there is always a strictly negative excursion and vice-versa). Note also that when the process passes from a strictly negative excursion to a strictly positive excursion, the first value of the process in the new strictly positive excursion is always 1.
\end{remark}

\begin{remark}
	The coalescent points of a coalescent-walk process obtained in this way have $y$-coordinates that are always equal to 1. 
\end{remark}

Note that Definition~\ref{defn:quantities_coal_walk_proc_strong} can be extended to semi-Baxter walks.
We have the following analogue of Lemma~\ref{lem:techn_for_comm_diagram_strong}.

\begin{lemma}\label{lem:techn_for_comm_diagram}
	Let $W\in\mathcal{W}^n_{sb}$. Fix $m\in[n]$, and consider the corresponding coalescent-walk process $\cpp(W_{|_{[m]}})=\{\Ztp\}_{t\in [m]}\eqqcolon Z$ and the corresponding semi-Baxter permutation $\pw^{-1}(W_{|_{[m]}})=\pi$. Assume that $W_m=(x,y)\in\Z^2_{\geq 0}$, i.e.\ $\pi$ has $x+1$ active sites smaller than or equal to $\pi(m)$ and $y+1$ active sites greater than $\pi(m)$,  denoted by
	$$\{s_{-x}<\dots<s_0\}\cup\{s_{1}<\dots<s_{y+1}\}.$$ 
	Then
	$$\FV(Z)=\{f_{-x}<\dots<f_{-1}<f_{0}=0<f_1<\dots<f_{y}\}=[-x,y],$$
	and in particular, $f_\ell-f_{\ell-1}=1$, for all $\ell\in[-x+1,y]$.
	Moreover, it holds that 
	\begin{equation}\label{eq:rel_active_sites}
		s_\ell=1+\sum_{j\leq \ell-1}\mult(f_j),\quad\text{for all}\quad \ell\in[-x,y+1].
	\end{equation}
\end{lemma}

\begin{proof}
	We prove the statement by induction over $m$.
	
	For $m=1$ then $x=0$, $y=0$, $\FV(Z)=\{0\}=\{f_0\}$ and $\mult(f_0)=1$. On the other hand, $\pi=1$ and the set of active sites is given by $\{s_0=1,s_1=2\}$. Note that  \eqref{eq:rel_active_sites} holds.
	
	Now assume that $0\leq m<n$ and that $Z$ and $\pi$ verify the statement of the lemma. We are going to show that also $\cpp(W_{|_{[m+1]}})=\{{Z'}^{(t)}\}_{t\in [m+1]}\eqqcolon Z'$ and the corresponding semi-Baxter permutation $\pw^{-1}(W_{|_{[m+1]}})=\pi'$ also verify the statement of the lemma.  We distinguish two cases:
	\begin{itemize}
		\item \textbf{Case 1:} $W_{m+1}-W_m=(i,-i+1)$ for some $i\in [y+1]$ (see the right-hand side of Figure~\ref{fig:example_for_proof}). 
		
		\noindent As explained in Section~\ref{sect:semi_bac_obj}, in this case $\pi'=\pi^{*s_{i}}$ and its active sites are 
		$$\{s'_{-x-i}<\dots <s'_0\}\cup\{s'_1<\dots<s'_{y-i+2}\}.$$
		
		where $s'_\ell=s_{\ell+i}$ for $\ell\in[-x-i,0]$ and $s'_\ell=s_{\ell+i-1}+1$ for $\ell\in[1,y-i+2]$.
		
		On the other hand, looking at Case 1 in \eqref{eq:distrib_incr_coal_2}, we immediately have that 
		$$\FV(Z')=\{f'_{-x-i}<\dots<f'_{-1}<f'_{0}=0<f'_1<\dots<f'_{y-i+1}\}=[-x-i,y-i+1],$$
		and $\mult(f'_0)=1$, $\mult(f'_\ell)=\mult(f_{\ell+i})$ for all $\ell\in[-x-i,-1]$, and $\mult(f'_\ell)=\mult(f_{\ell+i-1})$ for all $\ell\in[1,y-i+1]$.

		\item \textbf{Case 2:} $W_{m+1}-W_m=(-i,1)$ for some $i\in \{0\}\cup[x]$ (see the left-hand side of Figure~\ref{fig:example_for_proof}). 
		
		\noindent  As explained in Section~\ref{sect:semi_bac_obj}, in this case $\pi'=\pi^{*s_{-i}}$ and its active sites are 
		$$\{s'_{-x+i}<\dots <s'_0\}\cup\{s'_1<\dots<s'_{y+2}\},$$
		with $s'_{\ell}=s_{\ell-i}$ for all $\ell\in[-x+i,0]$, $s'_{1}=s_{-i}+1$, and $s'_{\ell}=s_{\ell-1}+1$ for all $\ell\in[2,y+2]$.
		
		On the other hand, looking at Case 2 in \eqref{eq:distrib_incr_coal_2}, we immediately have that 
		$$\FV(Z')=\{f'_{-x+i}<\dots<f'_{-1}<f'_{0}=0<f'_1<\dots<f'_{y+1}\}=[-x+i,y+1],$$
		and $\mult(f'_0)=1$, $\mult(f'_1)=\sum_{\ell=-i}^{0}\mult(f_\ell)$, $\mult(f'_\ell)=\mult(f_{\ell-i})$ for all $\ell\in[-x+i,-1]$, and $\mult(f'_\ell)=\mult(f_{\ell-1})$ for all $\ell\in[2,y+1]$.
	\end{itemize} 
	With a straightforward computation, based on the expressions of the $s'_\ell$ and $\mult(f'_\ell)$ in terms of $s_\ell$ and $\mult(f_\ell)$, it can be checked that \eqref{eq:rel_active_sites} holds.
\end{proof}

\begin{figure}[htbp]
	\centering
	\includegraphics[width=.8\linewidth]{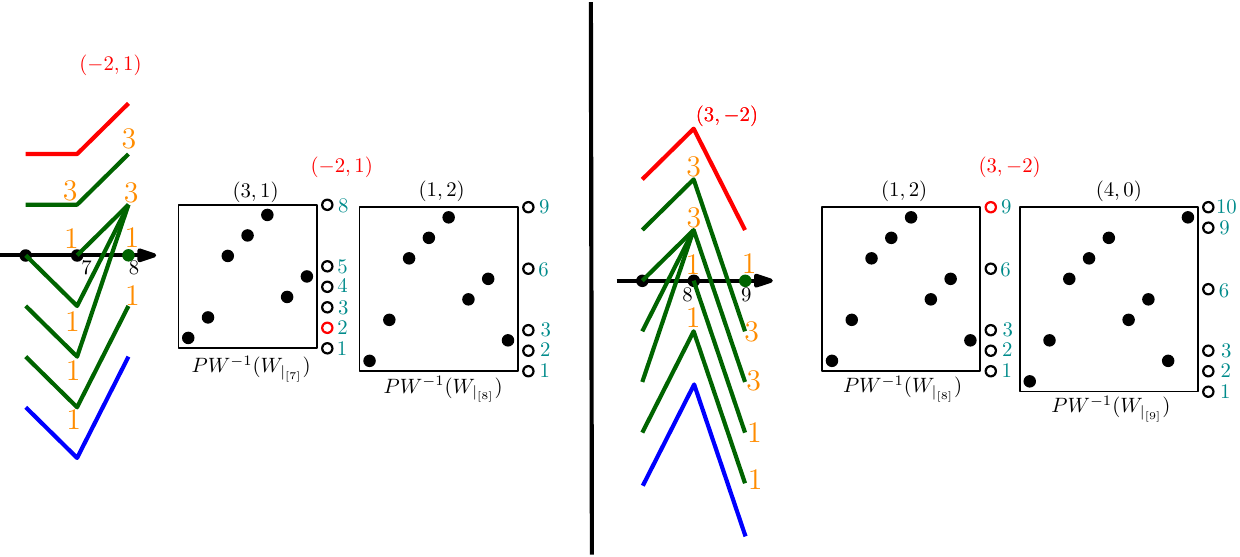}\\
	\caption{\textbf{Left:} The final steps of the coalescent-walk process $\rews(W_{|_{[8]}})$ in Figure~\ref{fig:ievdiuwbduwob} and the corresponding permutations $\pw^{-1}(W_{|_{[7]}})$ and $\pw^{-1}(W_{|_{[8]}})$ from Figure~\ref{fig:schema_perm} with the values of the active sites highlighted in cyan. We have that $W_8-W_7=(-2,1)$. Note that $\FV(\rews(W_{|_{[7]}}))=[-3,1]$ and from Figure~\ref{fig:ievdiuwbduwob}, we can determine that $\mult(-3)=\mult(-2)=\mult(-1)=\mult(0)=1$ and $\mult(1)=3$ (these numbers are plotted in orange close to the final values of the various walks). Note also that $\FV(\rews(W_{|_{[8]}}))=[-1,2]$ and we have that $\mult(-1)=\mult(0)=1$, $\mult(1)=3$ and $\mult(2)=3$.
		\textbf{Right: }The final steps of the coalescent-walk process $\rews(W_{|_{[9]}})$ in Figure~\ref{fig:ievdiuwbduwob} and the corresponding permutations $\pw^{-1}(W_{|_{[8]}})$ and $\pw^{-1}(W_{|_{[9]}})$ from Figure~\ref{fig:schema_perm}. We have that $W_9-W_8=(3,-2)$. As above, $\FV(\rews(W_{|_{[8]}}))=[-1,2]$ and $\mult(-1)=\mult(0)=1$, and  $\mult(1)=\mult(2)=3$. Note also that $\FV(\rews(W_{|_{[9]}}))=[-4,0]$ and we have that $\mult(-4)=\mult(-3)=1$, $\mult(-2)=\mult(-1)=3$, and $\mult(0)=1$. One can check, comparing the orange and cyan numbers, that in both cases \eqref{eq:rel_active_sites} holds.
		\label{fig:example_for_proof}}
\end{figure}

\begin{theorem}\label{thm:The_diagram_commutes}
	The diagram in \eqref{eq:diagrm_semi_baxter} commutes.
\end{theorem}

The proof is identical to the one of Theorem~\ref{thm:The_diagram_commutes_strong}, replacing Lemma~\ref{lem:techn_for_comm_diagram_strong} with Lemma~\ref{lem:techn_for_comm_diagram}.

\subsection{Probabilistic results for semi-Baxter permutations}\label{sect:semi_baxter_prob_res}

Here we follow the same steps as in Section~\ref{sect:prob_part} specializing to the case of semi-Baxter permutations. We will only explain the key-differences between the case of strong-Baxter permutations and semi-Baxter permutations.

\subsubsection{Sampling a uniform semi-Baxter permutation as a conditioned two-dimensional walk}

We proceed as in Section~\ref{sect:strong_bax_walks}.
Consider the following probability measure on  $I_{sb}$ (see also the left-hand side of Figure~\ref{fig:jumps_semi_baxter}):
\begin{equation}\label{eq:distristep_semibaxter}
	\mu_{sb}=\sum_{i=0}^{\infty}\alpha\gamma^{i}\cdot \delta_{(-i,1)}
	+\sum_{i=1}^{\infty}\alpha\gamma^{i}\cdot\delta_{(i,-i+1)},
\end{equation}
where $\alpha=\sqrt{5}-2$ and $\gamma=\frac{\sqrt{5}-1}{2}$, and $\delta$ denotes the delta-Dirac measure. Let $(\bm X,\bm Y)$ be a random variable such that $\mathcal{L}aw(\bm X,\bm Y)=\mu_{sb}$.  With standard computations, one can verify that:
\begin{equation}
	\E[\bm X]=\E[\bm Y]=0, \quad\quad \E[\bm X^2]=2(2+\sqrt 5),\quad\quad \E[\bm Y^2]=1+\sqrt 5,\quad\quad\E[\bm X\bm Y]=-(2+\sqrt 5).
\end{equation}
Therefore
\begin{equation}\label{eq:cov_semi-baxter}
	\Var((\bm X,\bm Y))=
	\begin{pmatrix}
		2(2+\sqrt 5) & -(2+\sqrt 5) \\
		-(2+\sqrt 5) & 1+\sqrt 5 
	\end{pmatrix}
\end{equation}
and so $\rho=\Cor((\bm X,\bm Y))=-\frac{1+\sqrt 5}{4}$.
We now denote by
\begin{equation}
	\stackrel{\leftarrow}{I_{sb}}\coloneqq\{(i,-1): i\geq 0\}\cup\{(-i,i-1): i\geq 1\},
\end{equation}
i.e.\ the set of ``reversed" increments (recall the definition of the set $I_{sb}$ in \eqref{eq:semi_bax_increm}). We further denote by 
\begin{equation}
	\stackrel{\leftarrow}{\mu_{sb}}=\sum_{i=0}^{\infty}\alpha\gamma^{i}\cdot \delta_{(i,-1)}
	+\sum_{i=1}^{\infty}\alpha\gamma^{i}\cdot\delta_{(-i,i-1)},
\end{equation}
the ``reversed" distribution on $\stackrel{\leftarrow}{I_{sb}}$ induced by $\mu_{sb}$ (see also the right-hand side of Figure~\ref{fig:jumps_semi_baxter}). 
\begin{figure}[htbp]
	\centering
	\hspace{3.1cm}\includegraphics[scale=0.7]{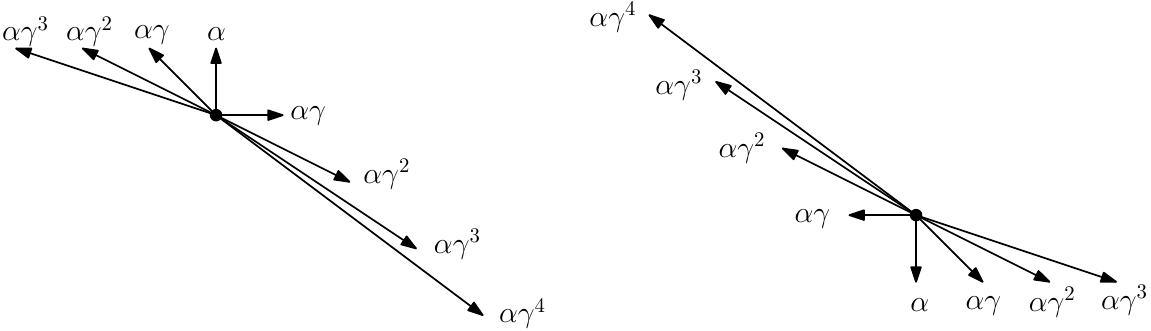}
	\caption{\textbf{Left:} Some of the increments in the set $I_{sb}$ are plotted together with the corresponding probability weights given by $\mu_{sb}$. \textbf{Right:} Some increments in the set $\stackrel{\leftarrow}{I_{sb}}$ are plotted together with the corresponding probability weights given by $\stackrel{\leftarrow}{\mu_{sb}}$.\label{fig:jumps_semi_baxter}}
\end{figure}

For all $n\in\zgz$, we define the following additional probability measure
\begin{equation}
	\nu_{sb}^n=\frac{1}{Z_n}\sum_{(h,\ell)\in\mathcal{L}^n_{sb}}\gamma^{2\ell+h} \delta_{(h,\ell)},
\end{equation}
where 
$
\mathcal{L}_{sb}^n\coloneqq\{\text{Labels at level $n$ in the generating tree for semi-Baxter permutations}\}
$
and the normalizing constant satisfies
$
Z_n=\sum_{(h,\ell)\in\mathcal{L}^n_{sb}}\gamma^{2\ell+h}.
$

Let $(\stackrel{\leftarrow}{\bm W_n}(i))_{i\geq 1}$ be a two-dimensional random walk with increments distributed as $\stackrel{\leftarrow}{\mu_{sb}}$ and starting probability $\nu_{sb}^n(h,\ell)=\P(\stackrel{\leftarrow}{\bm W_n}(1)=(h,\ell))$.
Denote by $\stackrel{\leftarrow}{\mathcal{W}_{sb}^n}$ the set of two-dimensional walks $(x_i)_{i\in[n]}$ in the non-negative quadrant with increments in $\stackrel{\leftarrow}{I_{sb}}$ and such that $x_n=(0,0)$.
\begin{proposition}\label{prop:the_walk_is_uniform}
	Conditioning on the event $\left\{(\stackrel{\leftarrow}{\bm W_n}(i))_{i\in [n]}\in\stackrel{\leftarrow}{\mathcal{W}_{sb}^n}\right\}$, the walk $(\stackrel{\leftarrow}{\bm W_n}(i))_{i\in [n]}$ is a uniform walk in $\stackrel{\leftarrow}{\mathcal{W}_{sb}^n}$.
\end{proposition}

\begin{proof}
	Fix $(x_i)_{i\in[n]}\in\stackrel{\leftarrow}{\mathcal{W}_{sb}^n}$. It is enough to show that $\P\left((\stackrel{\leftarrow}{\bm W_n}(i))_{i\in [n]}=(x_i)_{i\in[n]}\right)$ is independent of the choice of $(x_i)_{i\in[n]}$. To do that, assume that the walk $(x_i)_{i\in[n]}$ is formed by $k$ increments of the form $(-i,i-1)$ and $n-k-1$ increments of the form $(i,-1)$. Moreover, assume that the sum of the $k$ increments of the form $(-i,i-1)$ is equal to $(-s,s-k)$ and the sum of the $n-k-1$ increments of the form $(i,-1)$ is equal to $(q,-n+k+1)$. Under these assumptions we have that if $x_1=(h,\ell)$ then $x_n=(h-s+q,\ell+s-n+1)$. Since $(x_i)_{i\in[n]}\in\stackrel{\leftarrow}{\mathcal{W}_{sb}^n}$, it must hold that $(h-s+q,\ell+s-n+1)=(0,0)$ and so $s=n-1-\ell$ and $q=n-1-\ell-h$, in particular $s+q=2n-2\ell-h-2$.
	
	Under these assumptions,
	\begin{multline}
		\P\left((\stackrel{\leftarrow}{\bm W_n}(i))_{i\in [n]}=(x_i)_{i\in[n]}\right)=\nu_{sb}^n(h,\ell)\cdot\alpha^{n-1}\cdot\gamma^{s+q}\\
		=\frac{\gamma^{2\ell+h}}{Z_n}\cdot\alpha^{n-1}\cdot\gamma^{2n-2\ell-h-2}=\frac{\alpha^{n-1}\cdot\gamma^{2n-2}}{Z_n},
	\end{multline}
	and this concludes the proof.
\end{proof}

Let $(\bm W_n(i))_{i\geq 1}$ be the reversed walk obtained from $(\stackrel{\leftarrow}{\bm W_n}(i))_{i\geq 1}$. A consequence of Proposition~\ref{prop:the_walk_is_uniform} is the following.

\begin{corollary}\label{cor:the_walk_is_uniform}
	Conditioning on the event $\left\{(\bm W_n(i))_{i\in[n]}\in\mathcal{W}_{sb}^n\right\}$, the walk $(\bm W_n(i))_{i\in [n]}$ is a uniform walk in $\mathcal{W}_{sb}^n$.
\end{corollary}

\subsubsection{Scaling limit of the conditioned two-dimensional walks for semi-Baxter permutations}

We define a rescaled version of the walk $(\bm W_n(i))_{i\geq 1}=(\bm X_{n}(i),\bm Y_{n}(i))_{i\geq 1}$: for all $n\geq 1,$ let $\conti W_n:[0,1]\to \R^2$ be the continuous process that linearly interpolates the following points
$$
\conti W_n\left(\frac kn\right) =  \left(\frac {\bm X_n(k)} {\sqrt {2(2+\sqrt 5) n}},\frac {\bm Y_n(k)} {\sqrt {(1+\sqrt 5) n}}\right),\quad\text{ for all } k\in [n].
$$

\begin{proposition}\label{prop:weiugfweougfwoeu}
	Conditioning on the event $\left\{(\bm W_n(i))_{i\in[n]}\in\mathcal{W}_{sb}^n\right\}$, the following convergence in the space $\czord$ holds 
	\begin{equation}\label{eq:conv_to_brow_excurs}
		\conti W_n
		\xraninf{d}
		\conti E_{\rho},
	\end{equation}
	where $\rho=-\frac{1+\sqrt 5}{4}$.
\end{proposition}

The proof of this proposition is identical to the proof of Proposition~\ref{prop:scaling_strong_walk} using the expression for $\rho$ given below \eqref{eq:cov_semi-baxter}.

\subsubsection{Scaling limit of coalescent-walk processes and semi-Baxter permutations}

Let $\overline{\bm W} = (\obmx,\obmy) =(\obmx(k),\obmy(k))_{k\in \Z}$ be a random bi-infinite two-dimensional walk with step distribution $\mu_{sb}$ defined in \eqref{eq:distristep_semibaxter}, and let $\ovbmZ = \wcp(\overline{\bm W})$ be the corresponding discrete coalescent-walk process. For convenience, we set $\ovbmZ^{(j)}_i = 0$ for $i,j \in \Z$, $i<j$.

We further define the following rescaled processes: for all $n\geq 1, u\in \R$, let $\overline {\conti W}_n:\R\to \R^2$, and $\contzu_n:\R\to\R$ be the continuous processes that interpolate the following points:
\begin{equation}\label{eq:rescaled_version}
	\ovconw_n\left(\frac kn\right) = \left(\frac {\obmx(k)} {\sqrt {2(2+\sqrt 5)n}},\frac {\obmy(k)} {\sqrt {(1+\sqrt 5)n}}\right), \quad\text{for all}\quad k\in \Z,
\end{equation}
and
\begin{equation}\label{eq:rescaled_version2}
	\contzu_{n}\left(\frac kn\right) =
	\begin{cases}
		\frac {\ovbmZ^{(\lceil nu\rceil)}_{k}} { \sqrt {(1+\sqrt 5)n}},\quad&\text{when}\quad\ovbmZ^{(\lceil nu\rceil)}_{k}\geq 0,\\
		\frac {\ovbmZ^{(\lceil nu\rceil)}_{k}} {\sqrt {2(2+\sqrt 5) n}},\quad&\text{when}\quad\ovbmZ^{(\lceil nu\rceil)}_{k}\leq 0,
	\end{cases}
	\quad\text{for all}\quad k\in \Z.
\end{equation}

Recall that $\ovconw_{\rho} = (\overline{\conti X}_{\rho},\overline{\conti Y}_{\rho})$ denotes a two-dimensional Brownian motion of correlation $\rho$.
\begin{theorem}\label{thm:fvwuofbgqfipqhfqfpq_semib}
	Fix $u\in \R$. 
	The following joint convergence in the space $\mathcal C(\R,\R)^{3}$ holds: 
	\begin{equation}
		\label{eq:fvwuofbgqfipqhfqfpq_semib}
		\left(\ovconw_n,\contzu_n\right) 
		\xraninf{d}
		\left(\ovconw_{\rho},\ovconz_{\rho,q}^{(u)}\right),
	\end{equation}
	where
	\begin{equation}
		\rho=-\frac{1+\sqrt 5}{4}
		\quad\text{and}\quad
		q=\frac{1}{2},
	\end{equation} 
	and $\ovconz_{\rho,q}^{(u)}$ is the solution of the SDE in \eqref{eq:flow_SDE_inf_vol_strong} driven by $\ovconw_{\rho}$.
\end{theorem}

Theorem~\ref{thm:fvwuofbgqfipqhfqfpq_semib} follows from the following lemma, as Theorem~\ref{thm:fvwuofbgqfipqhfqfpq_strongb} follows from Proposition~\ref{prop:skew_part_strong}. Therefore we just give the proof of the following result.

\begin{proposition}\label{prop:skew_part}
	Fix $u\in \R$. 
	We have the following convergence in $\mathcal C(\R,\R)$: 
	\begin{equation}
		\contzu_n 
		\xraninf{d} \overline{\conti B}^{(u)},
	\end{equation}
	where $\overline{\conti B}^{(u)}(t)=0$ for $t<u$ and $\overline{\conti B}_{q}^{(u)}(t)$ is a  Brownian motion for $t\geq u$.
\end{proposition}

\begin{remark}\label{rem:traj}
	We highlight a remarkable difference with the case of strong-Baxter permutations: the walks of the continuous coalescent-walk process for semi-Baxter permutations are classical Brownian motions. This was also the case of Baxter permutations (see \cite[Theorem 4.5 and Remark 4.3]{borga2020scaling}), but there is a substantial difference between the walks of the  discrete coalescent-walk process associated with Baxter permutations and the ones associated with semi-Baxter permutations: in \cite[Proposition 3.3.]{borga2020scaling} we proved that the walks of the discrete coalescent-walk process associated with Baxter permutations are simple random walks with a specific step distribution that is centered. On the contrary the walks of the discrete coalescent-walk process associated with semi-Baxter permutations are not even martingales (this can be checked with a standard computation using \eqref{eq:distristep_semibaxter} and \eqref{eq:distrib_incr_coal_2}). This will force us to first show that $\contzu_n$ converges in distribution to a skew Brownian motion (using some arguments similar to the ones already used for strong-Baxter permutations) and then to deduce that its parameter is $1/2$, concluding that it is actually a classical Brownian motion.
	
	In simple words, we can summarize the discussion above as follows:
	\begin{itemize}
		\item walks of the coalescent-walk process for Baxter permutations are symmetric both in the discrete and in the continuum limit;
		\item walks of the coalescent-walk process for semi-Baxter permutations are not symmetric in the discrete case but they become symmetric in the continuum limit;
		\item walks of the coalescent-walk process for strong-Baxter permutations are not symmetric both in the discrete and in the continuum limit.
	\end{itemize} 
\end{remark}

\begin{proof}[Proof of Proposition~\ref{prop:skew_part}]
	As explained in Remark~\ref{rem:traj}, we prove that $\contzu_n 
	\xraninf{d} \overline{\conti B}_{1/2}^{(u)}$, where $\overline{\conti B}_{1/2}^{(u)}(t)=0$ for $t<u$ and $\overline{\conti B}_{1/2}^{(u)}(t)$ is a  skew Brownian motion for $t\geq u$ of parameter $1/2$.
	
	We only consider the case $u=0$, the general proof being similar. Again we just prove convergence of one-dimensional marginal distributions. We recall that $\overline {\conti Z}_n\coloneqq\overline {\conti Z}^{(0)}_n:\R\to\R$ is the continuous process defined by linearly interpolating the following points:
	\begin{equation}
		\ovconz_{n}\left(\frac kn\right) =
		\begin{cases}
			\frac {\ovbmZ_{k}} {\sigma'\sqrt {n}},\quad&\text{when}\quad\ovbmZ_{k}\geq 0,\\
			\frac {\ovbmZ_{k}} {\sigma\sqrt {n}},\quad&\text{when}\quad\ovbmZ_{k}\leq 0,
		\end{cases}
		\quad\text{for all}\quad k\in \Z.
	\end{equation}
	where ${\ovbmZ_{k}}\coloneqq {\ovbmZ^{(0)}_{k}}$, and $\sigma'=\sqrt{1+\sqrt 5}$ and $\sigma=\sqrt{2(2+\sqrt 5)}$.
	For the rest of the proof we fix a non-negative and Lipschitz continuous function $\varphi$ with compact support.
	We want to show that for all $t\geq 0$,
	\begin{equation}
		\E\left[\varphi\left(\overline {\conti Z}_n(t)\right)\right]\to
		1/2\int_{0}^{+\infty}\varphi(y)\frac{2e^{-y^2/2t}}{\sqrt{2\pi t}}dy+1/2\int_{-\infty}^{0}\varphi(y)\frac{2e^{-y^2/2t}}{\sqrt{2\pi t}}dy.
	\end{equation}
	Since $\varphi$ is Lipschitz continuous it is enough to study
	\begin{equation}\label{eq:jbfgweiufbowfwfnoiweb}
		\E\left[\varphi\left(\frac {\ovbmZ_{\lfloor nt\rfloor}} {\sigma'\sqrt  n}\right);\ovbmZ_{\lfloor nt\rfloor}>0\right]
		\quad\text{and}\quad
		\E\left[\varphi\left(\frac {\ovbmZ_{\lfloor nt\rfloor}} {\sigma\sqrt  n}\right);\ovbmZ_{\lfloor nt\rfloor}<0\right].
	\end{equation}
	As before, we just explicitly detail the computations for the first expectation. It can be decomposed as (recall Remark~\ref{obs:alternating_excursions})
	\begin{equation}
		\E\left[\varphi\left(\frac {\ovbmZ_{\lfloor nt\rfloor}} {\sigma'\sqrt  n}\right);\ovbmZ_{\lfloor nt\rfloor}>0\right]=
		\sum_{k=0}^{\lfloor nt\rfloor}\sum_{\ell=0}^{+\infty}\E\left[\varphi\left(\frac {\ovbmZ_{\lfloor nt\rfloor}} {\sigma'\sqrt  n}\right);\bm \tau^{+}_{\ell}=k,(\ovbmZ_{i})_{i\in[k+1,\lfloor nt\rfloor]}>0\right],
	\end{equation}
	where $\bm \tau^{+}_{0}=0$ and $\bm \tau^{+}_{\ell+1}=\inf\{i\geq\bm \tau^{+}_{\ell}:\ovbmZ_{i-1}<0,\ovbmZ_{i}=1\}$ for all $\ell\in\Z_{\geq 0}$. Note that for all $\ell\in\Z_{\geq 0}$, using the definition in \eqref{eq:refrom_coal_semib} we have that
	$$(\ovbmZ_{\lfloor nt\rfloor};\bm \tau^{+}_{\ell}=k,(\ovbmZ_{i})_{i\in[k+1,\lfloor nt\rfloor]}>0)\stackrel{d}{=}(1+\bm S_{\lfloor nt\rfloor-k};\tau^{+}_{\ell}=k,(1+\bm S_{i})_{i\in[\lfloor nt\rfloor-k]}> 0),$$ 
	where $(\bm S_i)_{i\geq 0}$ denotes a random walk started at zero at time zero, with step distribution equal to the distribution of $\obmy(1)-\obmy(0)$, and independent of $\bm \tau^{+}_{\ell}$.Therefore we can write
	\begin{align}\label{eq:webhfyuwheribfiuewbfoweuf}
		\E&\left[\varphi\left(\frac {\ovbmZ_{\lfloor nt\rfloor}} {\sigma'\sqrt  n}\right);\ovbmZ_{\lfloor nt\rfloor}>0\right]\notag\\
		=&
		\sum_{k=0}^{\lfloor nt\rfloor}\left(\sum_{\ell=0}^{+\infty}\P(\bm \tau^{+}_{\ell}=k)\right)\E\left[\varphi\left(\frac {1+\bm S_{\lfloor nt\rfloor-k}} {\sigma'\sqrt  n}\right);(1+\bm S_{i})_{i\in[\lfloor nt\rfloor-k]}> 0\right]\notag\\
		=&\sum_{k=0}^{\lfloor nt\rfloor}\left(\sum_{\ell=0}^{+\infty}\P(\bm \tau^{+}_{\ell}=k)\right)\E\left[\varphi\left(\frac {1+\bm S_{\lfloor nt\rfloor-k}} {\sigma'\sqrt  n}\right)\middle |(1+\bm S_{i})_{i\in[\lfloor nt\rfloor-k]}> 0\right]\notag\\
		&\cdot\P\left((1+\bm S_{i})_{i\in[\lfloor nt\rfloor-k]}> 0\right).
	\end{align}
	We now focus on studying $\P(\bm \tau^{+}_{\ell}=k)$. Note that $(\bm \tau^{+}_{\ell})_{\ell\in\Z_{\geq 2}}$ are all equidistributed, but $\bm \tau^{+}_{1}$ has a slightly different distribution (because $\ovbmZ_{\bm \tau^{+}_{0}}=\ovbmZ_{0}=0$ but $\ovbmZ_{\bm \tau^{+}_{\ell}}=1$ for all $\ell\in\Z_{\geq 1}$).  We have the following result.
	
	\begin{lemma}\label{lem:comput_param_semi}
		As $k\to\infty$,
		\begin{equation}
			\P(\bm \tau^{+}_{2}=k)\sim\frac{\beta}{k^{3/2}},
		\end{equation}
		where
		\begin{equation}\label{eq:efwbjue8wifbew9nf}
			\beta=\frac{1}{\sqrt{2\pi}}\frac{1}{1-\gamma}\left(\frac{1}{\sigma'}
			+
			\frac{1+\gamma}{\sigma}
			\right).
		\end{equation}
	\end{lemma}
	\begin{proof}
		Recall that $\ovbmZ_{\bm \tau^{+}_{1}}=1$. Using Observations \ref{obs:never_zero} and \ref{obs:alternating_excursions}, we can write (see also Figure~\ref{fig:schema_prob_event})
		\begin{multline}\label{eq:wfbweiubfwepifbwepfnewf}
			\P(\bm \tau^{+}_{2}=k)=
			\sum_{s=0}^{k-2}\sum_{\substack{y\in\zgz\\
					y',y''\in\Z_{< 0}}}\P(\ovbmZ_{s}=y,(\ovbmZ_{i})_{i\in[s]}> 0)
			\P(\ovbmZ_{s+1}-\ovbmZ_{s}=y'-y)\\
			\P(\ovbmZ_{k-1}=y'',(\ovbmZ_{i})_{i\in[s+1,k-1]}<0|\ovbmZ_{s+1}=y')
			\P(\ovbmZ_{k}-\ovbmZ_{k-1}=1-y'').
		\end{multline}
		\begin{figure}[htbp]
			\centering
			\includegraphics[scale=1]{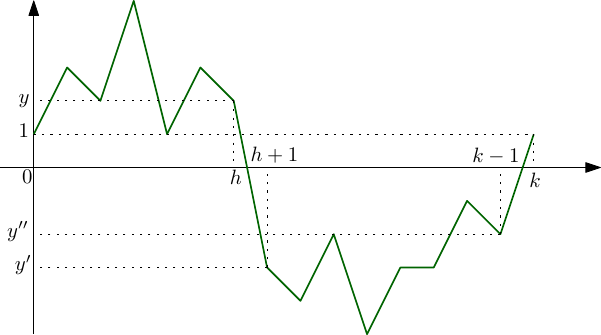}
			\caption{A schema for the event $\left\{\bm \tau^{+}_{2}=k\right\}$.\label{fig:schema_prob_event}}
		\end{figure}

		Using the definition in \eqref{eq:refrom_coal_semib} and letting:
		\begin{itemize}
			\item  $(\bm S_i)_{i\geq 0}$ be a random walk started at zero at time zero, with step distribution equal to the distribution of $\obmy(1)-\obmy(0)$,
			\item  $(\bm S'_i)_{i\geq 0}$ be a random walk started at zero at time zero, with step distribution equal to the distribution of $-(\obmx(1)-\obmx(0))$,
			\item $(\bm S_i)_{i\geq 0}$ is independent of $(\bm S'_i)_{i\geq 0}$,
		\end{itemize}
		we can rewrite the previous expression as
		\begin{multline}\label{eq:wefiuwebfewbfqewbofq}
			\P(\bm \tau^{+}_{2}=k)=\sum_{s=0}^{k-2}\sum_{\substack{y\in\zgz\\
					y',y''\in\Z_{< 0}}}\P(1+\bm S_{h}=y,(1+\bm S_{i})_{i\in[s]}> 0)
			\P(\obmy(1)-\obmy(0)=y'-y+1)\\
			\P(y'+\bm S'_{k-s-2}=y'',(y'+\bm S'_{i})_{i\in[k-s-2]}< 0)
			\P(-(\obmx(1)-\obmx(0))\geq -y'').
		\end{multline}
		We now focus on
		$\sum_{s=0}^{k-2}\P(1+\bm S_{s}=y,(1+\bm S_{i})_{i\in[s]}> 0)
		\P(y'+\bm S'_{k-s-2}=y'',(y'+\bm S'_{i})_{i\in[k-s-2]}< 0)$. From \cite[Theorem 6]{MR3342657} we know that as $m\to \infty$,
		\begin{align}\label{eq:kfjbouefboewnfewiopknf}
			&\P(1+\bm S_{m}=y,(1+\bm S_{i})_{i\in[m]}> 0)\sim\frac{c_1}{\sigma'\sqrt{2\pi}}\frac{\tilde{h}(y)}{m^{3/2}},\\
			&\P(y'+\bm S'_{m}=y'',(y'+\bm S'_{i})_{i\in[m]}< 0)\sim\frac{c_2}{\sigma\sqrt{2\pi}}\frac{h'(y')\tilde{h'}(y'')}{m^{3/2}},\notag\\
			&\P(1+\bm S_{m}=y,(1+\bm S_{i})_{i\in[m]}> 0)\leq C\frac{\tilde{h}(y)}{m^{3/2}},\quad\text{for all}\quad y\in\Z_{> 0},\notag\\
			&\P(y'+\bm S'_{m}=y'',(y'+\bm S'_{i})_{i\in[m]}< 0)\leq C\frac{h'(y')\tilde{h'}(y'')}{m^{3/2}},\quad\text{for all}\quad y',y''\in\Z_{< 0},\notag
		\end{align}
		where $\tilde{h}$, $h'$ and $\tilde{h'}$ are the functions defined in \eqref{eq:efdbuebfwebfowe} for the walks $-\bm S_{m}$, $\bm S'_{m}$ and $-\bm S'_{m}$ respectively, and
		\begin{equation}\label{eq:corrconstants_semi}
			c_1=\frac{\E[-\bm{S}_{\bm N}]}{\sum_{z\in\zgz} \tilde{h}(z)\P(\bm{S}_1\leq -z)}\quad\text{ and }\quad c_2=\frac{\E[\bm{S}'_{\bm N'}]}{\sum_{z\in\Z_{<0}} \tilde{h}'(z)\P(\bm{S}'_1\geq -z)},
		\end{equation}
		with $\bm N\coloneqq\inf\{n>0|\bm{S}_n< 0\}$ and $\bm N'\coloneqq\inf\{n>0|\bm{S'}_n> 0\}$.
		
		Using Proposition~\ref{prop:asympt_estim_sum} and the estimates above, we have, as $k\to \infty$,
		\begin{multline}
			\sum_{s=0}^{k-2}\P(1+\bm S_{s}=y,(1+\bm S_{i})_{i\in[s]}> 0)
			\P(y'+\bm S'_{k-s-2}=y'',(y'+\bm S'_{i})_{i\in[k-s-2]}< 0)\\
			\sim\frac{1}{\sqrt{2\pi}k^{3/2}}C(y,y',y''),
		\end{multline}
		where $C(y,y',y'')$ is equal to
		\begin{multline}
			\frac{c_1\cdot\tilde{h}(y)}{\sigma'}\left(\sum_{s=0}^{\infty}\P(y'+\bm S'_{s}=y'',(y'+\bm S'_{i})_{i\in[s]}< 0)\right)\\
			+\frac{c_2\cdot h'(y')\tilde{h'}(y'')}{\sigma}\left(\sum_{s=0}^{\infty}\P(1+\bm S_{s}=y,(1+\bm S_{i})_{i\in[s]}> 0)\right).
		\end{multline}
		Therefore, using the last two uniform bounds in \eqref{eq:kfjbouefboewnfewiopknf}, we can conclude that
		\begin{multline}
			\P(\bm \tau^{+}_{2}=k)
			\sim\frac{1}{k^{3/2}}
			\frac{1}{\sqrt{2\pi}}\\
			\cdot\sum_{\substack{y\in\zgz\\
					y',y''\in\Z_{< 0}}}
			\P(\obmy(1)-\obmy(0)=y'-y+1)
			\P(-(\obmx(1)-\obmx(0))\geq -y'')C(y,y',y'')=\frac{\beta}{k^{3/2}}.
		\end{multline}
		We finally simplify the expression for $\beta$ in order to obtain the  expression in \eqref{eq:efwbjue8wifbew9nf}. Note that
		\begin{multline}
			\sum_{s=0}^{\infty}\sum_{y''\in\Z_{< 0}}\P(-(\obmx(1)-\obmx(0))\geq -y'')\P(y'+\bm S'_{s}=y'',(y'+\bm S'_{i})_{i\in[s]}< 0)\\
			=
			\sum_{s=0}^{\infty}
			\P(y'+\bm S'_{s+1}\geq0,(y'+\bm S'_{i})_{i\in[s]}< 0)
			=1,
		\end{multline}
		because a symmetric random walk started at $y'$ becomes eventually non-negative almost surely.
		Using the latter identity and substituting the expressions of $c_1$ and $c_2$ given in \eqref{eq:corrconstants_semi}, we obtain that
		\begin{multline}
			\beta=\frac{1}{\sqrt{2\pi}}
			\Bigg(\frac{\E\left[\bm S_{\bm N}\right]}{\sigma'}\\
			+\frac{\E\left[\bm S'_{\bm N'}\right]}{\sigma}
			\sum_{\substack{y\in\zgz\\
					y'\in\Z_{< 0}}}
			h'(y')\P(\obmy(1)-\obmy(0)=y'-y+1)
			\sum_{s=0}^{\infty}\P(1+\bm S_{s}=y,(1+\bm S_{i})_{i\in[s]}> 0)\Bigg).
		\end{multline}
		Noting that
		\begin{equation}
			\sum_{y\in\zgz}\P(\obmy(1)-\obmy(0)=y'-y+1)
			\P(1+\bm S_{s}=y,(1+\bm S_{i})_{i\in[s]}> 0)=\P(\bm S_{s+1}=y',(\bm S_{i})_{i\in[s]}\geq 0),
		\end{equation} 
		we obtain the new simplified expression
		\begin{equation}
			\beta=\frac{1}{\sqrt{2\pi}}
			\left(\frac{\E\left[\bm S_{\bm N}\right]}{\sigma'}
			+\frac{\E\left[\bm S'_{\bm N'}\right]}{\sigma}
			\sum_{s=0}^{\infty}\sum_{y'\in\Z_{< 0}}
			h'(y')
			\P(\bm S_{s+1}=y',(\bm S_{i})_{i\in[s]}\geq 0)\right).
		\end{equation}
		We now compute an expression for the function $h'(x)$ defined in \eqref{eq:efdbuebfwebfowe} for the walk $\bm S'_{m}$,
		\begin{equation}
			h'(x)=
			\begin{cases}
				1+\sum_{j=1}^{+\infty}\P(\bm {S'}^1_{\bm N'}+\dots +\bm {S'}^j_{\bm N'}< -x) 			\quad&\text{if}\quad x< 0,\\
				0 \quad&\text{otherwise},
			\end{cases}
		\end{equation} 
		where the $(\bm {S'}^j_{\bm N'})_{j\in\Z_{\geq 1}}$ are i.i.d.\ copies of $\bm {S'}_{\bm N'}$.
		Recall that 
		\begin{equation}
			\mu_{sb}=\sum_{i=0}^{\infty}\alpha\gamma^{i}\cdot \delta_{(-i,1)}
			+\sum_{i=1}^{\infty}\alpha\gamma^{i}\cdot\delta_{(i,-i+1)},
		\end{equation}
		where $\alpha=\sqrt{5}-2$ and $\gamma=\frac{\sqrt{5}-1}{2}$. 			Therefore 
		\begin{equation}\label{eq:weiuvgfbfowehnfeiwbfuewobf}
			\P(\obmx(1)-\overline{\bm 	X}_{0}=z)=\alpha\gamma^{|z|}\quad \text{for all} \quad z\in\Z,
		\end{equation}
		and
		\begin{align}\label{eq:fewviiuqewbfopeiqwbfiwpqehbf}
			&\P(\obmy(1)-\obmy(0)=z)=\alpha\gamma^{|z|+1}\quad \text{for all} \quad z\in\Z_{\leq 0},\notag\\
			&\P(\obmy(1)-\obmy(0)=1)=\frac{\alpha}{1-\gamma}.
		\end{align}
		From Lemma~\ref{lem:local_est_one_dim_expect} we have that $\P(\bm{S'}_{\bm N'}=x)=(1-\gamma)\gamma^{x-1}$, for all $x\in\zgz$. Therefore, for all $j,z\in\zgz$,
		\begin{multline}
			\P(\bm {S'}^1_{\bm N'}+\dots +\bm {S'}^j_{\bm N'}= z)=\sum_{\substack{z_1,\dots,z_j>0\\
					z_1+\dots+z_j=z}}\P(\bm {S'}^1_{\bm N'}=z_1,\dots,\bm {S'}^j_{\bm N'}=z_j)\\
			=\sum_{\substack{z_1,\dots,z_j>0\\
					z_1+\dots+z_j=z}}(1-\gamma)^j\gamma^{z-j}=\mathds{1}_{\{j\leq z\}}(1-\gamma)^j\gamma^{z-j}\binom{z-1}{j-1}
			.
		\end{multline}
		Therefore for all $x\in\Z_{<0}$,
		\begin{multline}
			\sum_{j=1}^{+\infty}\P(\bm {S'}^1_{\bm N'}+\dots +\bm {S'}^j_{\bm N'}< -x)=\sum_{j=1}^{+\infty}\sum_{z=1}^{-x-1}\P(\bm {S'}^1_{\bm N'}+\dots +\bm {S'}^j_{\bm N'}=z)\\
			=\sum_{z=1}^{-x-1}\gamma^{z}\sum_{j=1}^{z}(1-\gamma)^j\gamma^{-j}\binom{z-1}{j-1}=(1+x)(\gamma-1),
		\end{multline}
		and we can conclude that 
		\begin{equation}\label{eq:euidfweuibdfweopbdfiopwe}
			h'(x)=
			\begin{cases}
				\gamma + (\gamma-1) x \quad&\text{if}\quad x< 0,\\
				0 \quad&\text{otherwise}.
			\end{cases}
		\end{equation} 
		Therefore
		\begin{multline}\label{eq:ewbjfeiwufbnoewbnfeiwpo0}
			\beta
			=\frac{1}{\sqrt{2\pi}}\left(\frac{\E[-\bm{S}_{\bm N}]}{\sigma'}
			+
			\frac{\E[\bm{S}'_{\bm N'}]}{\sigma}
			\sum_{y'\in\Z_{< 0}}
			(\gamma + (\gamma-1) y')\sum_{s=1}^{\infty}\P(\bm S_{s}=y',(\bm S_{i})_{i\in[s-1]}\geq 0)\right)\\
			=\frac{\E[-\bm{S}_{\bm N}]}{\sqrt{2\pi}}\left(\frac{1}{\sigma'}
			+
			\frac{(1-\gamma)}{\sigma}
			\sum_{y'\in\Z_{< 0}}
			(\gamma + (\gamma-1) y')\gamma^{-y'-1}\right)\\
			=\frac{\E[-\bm{S}_{\bm N}]}{\sqrt{2\pi}}\left(\frac{1}{\sigma'}
			+
			\frac{1+\gamma}{\sigma}
			\right)
			=\frac{1}{\sqrt{2\pi}}\frac{1}{1-\gamma}\left(\frac{1}{\sigma'}
			+
			\frac{1+\gamma}{\sigma}
			\right),
		\end{multline}
		where we also used that $\P(\bm{S}_{\bm N}=-x)=(1-\gamma)\gamma^{x-1}=\P(\bm{S}'_{\bm N'}=x)$ and that $\E[-\bm{S}_{\bm N}]=\frac{1}{1-\gamma}$ as a consequence of Lemma~\ref{lem:local_est_one_dim_expect} together with the expressions in \eqref{eq:weiuvgfbfowehnfeiwbfuewobf} and \eqref{eq:fewviiuqewbfopeiqwbfiwpqehbf}.
	\end{proof}

	\begin{corollary}\label{cor:asympt_estim_rew}
		As $k\to+\infty$,
		\begin{equation}
			\sum_{\ell=0}^{+\infty}\P(\bm \tau^{+}_{\ell}=k)\sim\frac{1}{\beta}\frac{1}{2\pi\sqrt k}\;.
		\end{equation}
	\end{corollary}
	\begin{proof}
		With some simple modifications of the computations in the proof of Lemma~\ref{lem:comput_param_semi}, one can show that $\sqrt k\cdot \P(\bm \tau^{+}_{1}=k)\to 0$. Then, it is enough to use the same proof used for Corollary~\ref{cor:asympt_estim_rew_strong}.
	\end{proof}
	
	With this result in our hands we can now continue to estimate \eqref{eq:webhfyuwheribfiuewbfoweuf}. With the same argument used for \eqref{eq:fbwebfowenfpwenmfpew}, we have that
	\begin{multline}
		\E\left[\varphi\left(\frac {\ovbmZ_{\lfloor nt\rfloor}} {\sigma'\sqrt  n}\right);\ovbmZ_{\lfloor nt\rfloor}>0\right]\\
		=
		\sum_{k=0}^{\lfloor nt\rfloor}\left(\sum_{\ell=0}^{+\infty}\P(\bm \tau^{+}_{\ell}=k)\right)\E\left[\varphi\left(\frac {1+\bm S_{\lfloor nt\rfloor-k}} {\sigma'\sqrt  n}\right)\middle|(1+\bm S_{i})_{i\in[\lfloor nt\rfloor-k]}> 0\right]\\
		\cdot\P\left((1+\bm S_{i})_{i\in[\lfloor nt\rfloor-k]}> 0\right)
		=\int_0^tf_n^+(s) ds+O(1/\sqrt n),
	\end{multline}
	where for $1\leq k\leq \lfloor nt\rfloor-1$ and $s\in[\frac k n,\frac{k+1}{n})$,
	\begin{multline}
		f^+_n(s)\coloneqq n\left(\sum_{\ell=0}^{+\infty}\P(\bm \tau^{+}_{\ell}=k)\right)\E\left[\varphi\left(\frac {1+\bm S_{\lfloor nt\rfloor-k}} {\sigma'\sqrt  n}\right)\middle|(1+\bm S_{i})_{i\in[\lfloor nt\rfloor-k]}> 0\right]\\
		\cdot\P\left((1+\bm S_{i})_{i\in[\lfloor nt\rfloor-k]}> 0\right).
	\end{multline}
	and $f^+_n(s)\coloneqq 0$ on $(0,\frac 1 n)\cup[\frac{\lfloor nt\rfloor-1}{n},t)$.
	Using that (as in \eqref{eq:buibdqweobdpqwdbnwq})
	\begin{equation}
		\E\left[\varphi\left(\frac {1+\bm S_{\lfloor nt\rfloor}} {\sigma'\sqrt  n}\right)\middle|(1+\bm S_{i})_{i\in[\lfloor nt\rfloor]}> 0\right]\sim\frac{1}{t}\int_0^{+\infty}\varphi(u)ue^{-u^2/2t}du,
	\end{equation}
	and that from Lemma~\ref{lem:local_est_one_dim},
	\begin{equation}
		\P\left((1+\bm S_{i})_{i\in[n]}> 0\right)\sim 2\frac{\E[-\bm{S}_{\bm N}]}{\sigma'\sqrt{2\pi}} \frac{1}{\sqrt{n}}, 
	\end{equation}
	together with Corollary~\ref{cor:asympt_estim_rew}, we obtain that for all $s\in (0,t)$,
	\begin{multline}
		f^+_n(s)\sim n\frac{1}{\beta}\frac{1}{2\pi\sqrt n\sqrt s}\frac{1}{t-s}\int_0^{+\infty}\varphi(u)ue^{-u^2/2(t-s)}du\cdot 2 \frac{\E[-\bm{S}_{\bm N}]}{\sigma'\sqrt{2\pi}} \frac{h(1)}{\sqrt n\sqrt{t-s}}\\
		=\frac{\E[-\bm{S}_{\bm N}]}{\sigma'\sqrt{2\pi}\beta}\frac{1}{\pi\sqrt{s(t-s)}(t-s)}\int_0^{+\infty}\varphi(u)ue^{-u^2/2(t-s)}du.
	\end{multline}
	Therefore, using the same arguments used for \eqref{eq:weijbfuiwefbewonfpewinf}, we can conclude that
	\begin{equation}
		\E\left[\varphi\left(\frac {\ovbmZ_{\lfloor nt\rfloor}} {\sigma'\sqrt  n}\right);\ovbmZ_{\lfloor nt\rfloor}>0\right]
		\xrightarrow{n\to\infty}
		\frac{\E[-\bm{S}_{\bm N}]}{\sigma'\sqrt{2\pi}\beta}\int_{0}^{+\infty}\varphi(y)\frac{2e^{-y^2/2t}}{\sqrt{2\pi t}}dy.
	\end{equation}
	Using the expression $\beta=\frac{\E[-\bm{S}_{\bm N}]}{\sqrt{2\pi}}\left(\frac{1}{\sigma'}
	+
	\frac{1+\gamma}{\sigma}
	\right)$ in \eqref{eq:ewbjfeiwufbnoewbnfeiwpo0}, we can conclude that
	\begin{equation}
		q=\frac{\E[-\bm{S}_{\bm N}]}{\sigma'\sqrt{2\pi}\beta}=\frac{\frac{1}{\sigma'}}{\frac{1}{\sigma'}
			+
			\frac{1+\gamma}{\sigma}}
		=\frac{1}{2},
	\end{equation}
	where in the last equality we used that $\sigma'=\sqrt{1+\sqrt 5}$, $\sigma=\sqrt{2(2+\sqrt 5)}$, and $\gamma=\frac{\sqrt{5}-1}{2}$.	
\end{proof}

The proof of Theorem~\ref{thm:semi-baxter} follows from Theorem~\ref{thm:fvwuofbgqfipqhfqfpq_semib} using exactly the same strategy explained in Sections~\ref{sec:cond_conv}~and~\ref{sect:permlim}.

\appendix

\section{An upper-bound for some local probabilities for random walks in cones}\label{sect:prob_walk_cones}

We fix some notation being consistent with the notation used in \cite{MR3342657} in order to help the reader to follow the proof of Proposition~\ref{prop:upper_bound_walks}. Let $\bm S(n)$ be a random walk on $\R^d$, $d\geq 1$, where
$$\bm S(n)\coloneqq\bm X_1+\dots+\bm X_n,$$
and $(\bm X_i)_i$ is a family of i.i.d.\ copies of a random variable $\bm X=(\bm X^{(1)},\dots,\bm X^{(d)})$ with values in $\R^d$. Denote by $\mathbb{S}^{d-1}$ the unit sphere of $\R^d$ and $\Sigma$ an open and connected subset of $\mathbb{S}^{d-1}$. Let $K$ be the cone generated by the rays emanating from the origin and passing through $\Sigma$, that is, $\Sigma = K \cap \mathbb{S}^{d-1}$.
Let $\bm \tau_x$ be the exit time from $K$ of the random walk with starting point $x\in K$,
that is,
$$\bm \tau_x\coloneqq\inf\{n\geq 1 :x+\bm S(n)\notin K\}.$$

We make the same assumptions on $K$ as in \cite{MR3342657} (see in particular bottom of page 994). We also recall that the choice of the cone $K$ determines a corresponding parameter $p$ (see the equations from (2) to (3) in \cite{MR3342657} for a precise definition) that plays an important role in many formulas in \cite{MR3342657}.

We also impose, as in \cite{MR3342657}, the following assumptions on the increments of the random walk:
\begin{itemize}
	\item $\E[\bm X^{(j)}] = 0$, $\E[(\bm X^{(j)})^2] = 1$, $j=1,\dots, d$. In addition, we assume that $\cov(\bm X_i, \bm X_j ) = 0$.\footnote{See the bottom part of  page 996 in \cite{MR3342657} for a discussion on the case $\cov(\bm X_i, \bm X_j ) \neq 0$ and the fact that assuming $\cov(\bm X_i, \bm X_j ) = 0$ does not restrict the generality of the results.}
	\item We assume that $\E[\bm X^\alpha] = 0$ with $\alpha = p$ if $p > 2$ and some
	$\alpha > 2$ if $ p \leq 2$. 
\end{itemize}

\medskip

In \cite[Lemma 28]{MR3342657} it was shown that
$$\P(x+\bm S(n)=y,\bm \tau_x>n)\leq C(x,y) n^{-p-d/2},\quad\text{for all}\quad n\in \Z_{> 0} \text{ and } x,y\in K.$$
In this section, we want to make explicit the dependence in $x$ and $y$ of the constant $C(x,y)$.

\begin{proposition}\label{prop:upper_bound_walks}
	There exists a constant $C>0$ independent of $x,y,n$ such that
	$$\P(x+\bm S(n)=y,\bm \tau_x>n)\leq C (1+|x|^p)(1+|y|^p) n^{-p-d/2},\quad\text{for all}\quad n\in \Z_{> 0} \text{ and } x,y\in K.$$
\end{proposition}
In what follows $C$ denotes any constant (possibly different from place to place) independent of $x,y,n$.
\begin{proof}
	From \cite[Lemma 27]{MR3342657}, 
	$$\P(x+\bm S(n)=y,\bm \tau_x>n)\leq C \cdot n^{-d/2} \cdot \P(\bm \tau_x >n/2).$$
	If we show that
	\begin{equation}\label{eq:bgrebfuirbfo}
		\P(\bm \tau_x >n)\leq C (1+|x|^p)n^{-p/2},
	\end{equation}
	then using exactly the same proof as in \cite[Lemma 28]{MR3342657} we can conclude that 
	$$\P(x+\bm S(n)=y,\bm \tau_x>n)\leq C (1+|x|^p)(1+|y|^p) n^{-p-d/2},\quad\text{for all}\quad n\in \Z_{> 0} \text{ and } x,y\in K.$$	
	Therefore it is enough to prove \eqref{eq:bgrebfuirbfo}. Note that for all $\varepsilon>0$,
	\begin{equation}
		\P(\bm \tau_x >n)=\P(\bm \tau_x >n,\bm \nu_n\leq n^{1-\varepsilon})+\P(\bm \tau_x >n,\bm \nu_n> n^{1-\varepsilon}),
	\end{equation}
	where
	$$\bm \nu_n\coloneqq\min\{k\geq 1: x+\bm S(k)\in K_{n,\varepsilon}\},$$
	and
	$$K_{n,\varepsilon}\coloneqq \{x\in K: \dist(x,\partial K)\geq n^{1/2-\varepsilon}\}.$$
	By \cite[Lemma 14]{MR3342657}
	\begin{equation}
		\P(\bm \tau_x >n,\bm \nu_n> n^{1-\varepsilon})\leq \exp\{-C n^{\varepsilon}\}.
	\end{equation}
	Therefore, in order to conclude the proof, we have to show that
	\begin{equation}\label{eq:fvgwry8bfvruqbfrqp}
		\P(\bm \tau_x >n,\bm \nu_n\leq n^{1-\varepsilon})\leq C (1+|x|^p) n^{-p/2}.
	\end{equation}
	By  \cite[Equation (50)]{MR3342657},
	\begin{multline}\label{eq:webfbuowebfoeqbfiq}
		\P(\bm \tau_x >n,\bm \nu_n\leq n^{1-\varepsilon})
		\leq 
		C \cdot n^{-p/2} \Big(\E\left[u(x+\bm S(\bm \nu_n));\bm\tau_x>\bm\nu_n,\bm \nu_n\leq n^{1-\varepsilon}\right]\\
		+\E\left[|x+\bm S(\bm \nu_n)|^p;\bm\tau_x>\bm\nu_n,|x+\bm S(\bm \nu_n)|>n^{-\varepsilon/8}\sqrt{n},\bm \nu_n\leq n^{1-\varepsilon}\right]\Big),
	\end{multline}
	where $u(x)$ is a function defined in \cite[Equation (3)]{MR3342657} that satisfies $0\leq u(x)\leq C |x|^p$.
	We have the following results (whose proof is postponed after the end of the proof of this proposition).
	\begin{claim}\label{clm:bound_on_first_expect}  It holds that
		$$\E\left[u(x+\bm S(\bm \nu_n));\bm\tau_x>\bm\nu_n,\bm \nu_n\leq n^{1-\varepsilon}\right]\leq C (1+|x|^p).$$
	\end{claim}
	\begin{claim}\label{clm:bound_on_second_expect} It holds that
		$$\E\left[|x+\bm S(\bm \nu_n)|^p;\bm\tau_x>\bm\nu_n,|x+\bm S(\bm \nu_n)|>n^{-\varepsilon/8}\sqrt{n},\bm \nu_n\leq n^{1-\varepsilon}\right]\leq C (1+|x|^p).$$
	\end{claim}
	Note that \eqref{eq:fvgwry8bfvruqbfrqp} then follows from \eqref{eq:webfbuowebfoeqbfiq} and the two claims above. This concludes the proof of Proposition~\ref{prop:upper_bound_walks}.
\end{proof}

It remains to prove Claims \ref{clm:bound_on_first_expect} and \ref{clm:bound_on_second_expect}.

\begin{proof}[Proof of Claim~\ref{clm:bound_on_first_expect}]
	From the proof of \cite[Lemma 21]{MR3342657} we have that
	\begin{multline}\label{eq:dvfuwbfcibfpwernf}
		\E\left[u(x+\bm S(\bm \nu_n));\bm\tau_x>\bm\nu_n,\bm \nu_n\leq n^{1-\varepsilon}\right]=
		\E\left[\bm Y_{\bm\nu_n};\bm\tau_x>\bm\nu_n,\bm \nu_n\leq n^{1-\varepsilon}\right]\\
		+
		\E\left[\sum_{k=0}^{\bm\nu_n-1}f(x+\bm S(k));\bm\tau_x>\bm\nu_n,\bm \nu_n\leq n^{1-\varepsilon}\right],
	\end{multline}
	where $\bm Y_{n}$ is a martingale defined by (see \cite[Equation (20)]{MR3342657})
	\begin{align}\label{eq:rfgewbfhfe}
		&\bm Y_{0}=v(x),\\
		&\bm Y_{n}=v(x+\bm S(n))-\sum_{k=0}^{n-1}f(x+\bm S(k)), \quad n\geq 0,
	\end{align}
	with $v(\cdot)$ and $f(\cdot)$ two functions defined in \cite[Section 1.3]{MR3342657}.
	We have that the second expectation in the right-hand side of \eqref{eq:dvfuwbfcibfpwernf} is bounded by
	\begin{equation}\label{eq:ewfbbfiewnhfpiewfpe}
		\E\left[\sum_{k=0}^{\bm\nu_n-1}f(x+\bm S(k));\bm\tau_x>\bm\nu_n,\bm \nu_n\leq n^{1-\varepsilon}\right]
		\leq 
		\E\left[\sum_{k=0}^{\bm\tau_x-1}|f(x+\bm S(k))|;\bm\tau_x>\bm\nu_n\right]\leq C(1+|x|^p),
	\end{equation}
	where the last bound follows from \cite[Equation (24)]{MR3342657}. For the first expectation in the right-hand side of \eqref{eq:dvfuwbfcibfpwernf}, we can write using the computations after \cite[Equation (52)]{MR3342657},
	\begin{multline}\label{eq:errifbibfofohfewf}
		\E\left[\bm Y_{\bm\nu_n};\bm\tau_x>\bm\nu_n,\bm \nu_n\leq n^{1-\varepsilon}\right]=
		u(x)
		-
		\E\left[\bm Y_{\bm\tau_x};\bm\tau_x\leq\bm\nu_n\wedge n^{1-\varepsilon}\right]\\
		-
		\E\left[u(x+\bm S(n^{1-\varepsilon}));\bm\tau_x>n^{1-\varepsilon}, \bm\nu_n>n^{1-\varepsilon}\right]\\
		+
		\E\left[\sum_{k=0}^{n^{1-\varepsilon}-1}f(x+\bm S(k));\bm\tau_x>n^{1-\varepsilon},\bm \nu_n>n^{1-\varepsilon}\right],
	\end{multline}
	where $x\wedge y=\min\{x,y\}$.
	It remains to bound the four terms above. As we already mentioned $0\leq u(x)\leq C |x|^p$, so it is enough to find an appropriate bound for the second and fourth expectations . Recalling the definition of $\bm Y_n$ in \eqref{eq:rfgewbfhfe} and using the bounds in \cite[Equations (22) and (24)]{MR3342657} we have that
	\begin{equation}
		\left|\E\left[\bm Y_{\bm\tau_x};\bm\tau_x\leq\bm\nu_n\wedge n^{1-\varepsilon}\right]\right|\leq C(1+|x|^p).
	\end{equation}
	Finally, for the fourth expectation in the left-hand side of \eqref{eq:errifbibfofohfewf} we can use the same arguments already used for \eqref{eq:ewfbbfiewnhfpiewfpe} to obtain that 
	\begin{equation}
		\E\left[\sum_{k=0}^{n^{1-\varepsilon}-1}f(x+\bm S(k));\bm\tau_x>n^{1-\varepsilon},\bm \nu_n>n^{1-\varepsilon}\right]\leq C(1+|x|^p).
	\end{equation}
	This concludes the proof.
\end{proof}

\begin{proof}[Proof of Claim~\ref{clm:bound_on_second_expect}]
	As in the proof of \cite[Lemma 24]{MR3342657} we set
	\begin{equation}
		\bm \mu_n\coloneqq \min\{j\geq 1: |\bm X_j|>n^{1/2-\varepsilon/4}  \},
	\end{equation}
	and we write
	\begin{align}
		\E\Big[|x+\bm S(\bm \nu_n)|^p;&\bm\tau_x>\bm\nu_n,|x+\bm S(\bm \nu_n)|> n^{-\varepsilon/8}\sqrt{n},\bm \nu_n\leq n^{1-\varepsilon}\Big]\\
		=
		&\E\left[|x+\bm S(\bm \nu_n)|^p;\bm\tau_x>\bm\nu_n,|x+\bm S(\bm \nu_n)|> n^{-\varepsilon/8}\sqrt{n},\bm \nu_n\leq n^{1-\varepsilon},\bm\mu_n\leq\bm\nu_n\right]\\
		+
		&\E\left[|x+\bm S(\bm \nu_n)|^p;\bm\tau_x>\bm\nu_n,|x+\bm S(\bm \nu_n)|> n^{-\varepsilon/8}\sqrt{n},\bm \nu_n\leq n^{1-\varepsilon},\bm\mu_n>\bm\nu_n\right].
	\end{align}
	From the beginning of the proof of \cite[Lemma 24]{MR3342657} we have that the second expectation in the right-hand side of the equation above is bounded by $\exp\{-Cn^{\varepsilon/8}\}$, therefore we focus on the first expectation above. Using the first displayed equation after \cite[Equation (59)]{MR3342657} we have that
	\begin{multline}
		\E\left[|x+\bm S(\bm \nu_n)|^p;\bm\tau_x>\bm\nu_n,|x+\bm S(\bm \nu_n)|> n^{-\varepsilon/8}\sqrt{n},\bm \nu_n\leq n^{1-\varepsilon},\bm\mu_n\leq\bm\nu_n\right]\\
		\leq 
		\sum_{j=1}^{n^{1-\varepsilon}}\E\left[|x+\bm S(\bm \nu_n)|^p;\bm\tau_x>j,\bm \nu_n\leq n^{1-\varepsilon},j\leq \bm \nu_n,\bm\mu_n=j\right].
	\end{multline}
	We split the sum above in three parts using the bound
	\begin{equation}
		|x+\bm S(\bm \nu_n)|^p\leq C \left(|x+\bm S(j-1)|^p+|\bm X_j|^p+|\bm S(\bm \nu_n)-\bm S(j)|^p\right).
	\end{equation}
	
	For the third term we have that
	\begin{multline}\label{eq:ewfgbwgfo2hbfeqhfoe}
		\sum_{j=1}^{n^{1-\varepsilon}}\E\left[|\bm S(\bm \nu_n)-\bm S(j)|^p;\bm\tau_x>j,\bm \nu_n\leq n^{1-\varepsilon},j\leq \bm \nu_n,\bm\mu_n=j\right]\\
		\leq
		\begin{cases}
			C n^{-p\varepsilon/4}\E[|\bm X|^p]\E[\bm \tau_x]\leq C n^{-p \varepsilon/4} (1+|x|^2),\quad&\text{if}\quad p>2,\\
			C (1+|x|^p)n^{-\delta},\text{ with }\delta>0,\quad&\text{if}\quad p\leq2,
		\end{cases}
	\end{multline}
	where the bound for $p>2$ is obtained using the first and third equations at page 1026 of \cite{MR3342657} and \cite[Lemma 10]{MR3342657}, while the bound for $p\leq 2$ is obtained using the first, fourth and fifth equation at page 1026 of \cite{MR3342657} and \cite[Lemma 10]{MR3342657}.
	
	For the second term we have that
	\begin{equation}
		\sum_{j=1}^{n^{1-\varepsilon}}\E\left[|\bm X_j|^p;\bm\tau_x>j,\bm \nu_n\leq n^{1-\varepsilon},j\leq \bm \nu_n,\bm\mu_n=j\right]
		\leq
		\begin{cases}
			C (1+|x|^2)\cdot o(1),&\text{ if } p>2,\\
			C (1+|x|^p)n^{-\delta},\text{ with }\delta>0,&\text{ if } p\leq2,
		\end{cases}
	\end{equation}
	where the bounds are obtained from the equations on top of page 1027 in \cite{MR3342657}.
	
	For the first term we have that
	\begin{multline}
		\sum_{j=1}^{n^{1-\varepsilon}}\E\left[|x+\bm S(j-1)|^p;\bm\tau_x>j,\bm \nu_n\leq n^{1-\varepsilon},j\leq \bm \nu_n,\bm\mu_n=j\right]\\
		\leq
		\begin{cases}
			2^{p}(|x|^pdn^{-\varepsilon/2}+Cn^{3p/2+1}\exp\{-Cn^{\varepsilon/8}\}+Cn^{-p\varepsilon/4}(1+|x|^2)),\quad&\text{if}\quad p>2,\\
			2^{p}(|x|^pdn^{-\varepsilon/2}+Cn^{3p/2+1}\exp\{-Cn^{\varepsilon/8}\}+C(1+|x|^p)n^{-\delta}),\quad&\text{if}\quad p\leq2,
		\end{cases}
	\end{multline}
	where the bounds are obtained from Equation (63) and the three consecutive equations in \cite{MR3342657} and the same bounds used for \eqref{eq:ewfgbwgfo2hbfeqhfoe}. These three bounds conclude the proof.
\end{proof}

\section{A toolbox for strong-Baxter walks}

In this section we rewrite Proposition~\ref{prop:upper_bound_walks} using the notation and the specific random walks considered in Section~\ref{sect:strong_bax_walks}.

Recall that $(\bm W_n(i))_{i\geq 1}$ is a random walk with increments in 
\begin{equation}
	I_{Sb}\coloneqq\{(-i,0): i\geq 1\}\cup\{(0,1)\}\cup\{(1,-i): i\geq 0\}.
\end{equation}
and with step distribution 
\begin{equation}
	\mu_{Sb}=\sum_{i=1}^{\infty}\alpha\gamma^{i}\cdot \delta_{(-i,0)}
	+\alpha\theta^{-1}\cdot \delta_{(0,1)}+\sum_{i=0}^{\infty}\alpha\gamma^{-1}\theta^{i}\cdot\delta_{(1,-i)},
\end{equation}
where the various parameters are given below \eqref{eq:system_strong}.
Recall also that $\mathcal{W}^n_{Sb}$ denotes the set of two-dimensional walks in the non-negative quadrant, starting at $(0,0)$, with increments in $I_{Sb}$.
Therefore, setting $K=\{(x_1,x_2)\in\R^{2}:x_1\geq 0, x_2\geq 0\}$, we have that
\begin{multline}\label{eq:fieuywfgvdweqbgf_strong}
	\P\left((\bm W_n(i))_{i\in[n]}\in\mathcal{W}_{Sb}^n\middle|\bm W_n(n)=(h,\ell)\right)\\
	=\P\left(\stackrel{\leftarrow}{\bm W_n}(n)=(0,0),(\stackrel{\leftarrow}{\bm W_n}(i))_{i\in [n]}\in K\middle|\stackrel{\leftarrow}{\bm W_n}(1)=(h,\ell)\right),
\end{multline}
where we recall that $(\stackrel{\leftarrow}{\bm W_n}(i))_{i\in [n]}$ is the time-reversed walk obtained from $(\bm W_n(i))_{i\in[n]}$.
Therefore we have the following reinterpretation of Proposition~\ref{prop:upper_bound_walks}.

Let $p>0$ be the parameter previously described in Section~\ref{sect:prob_walk_cones} associated with the walk $(\bm W_n(i))_{i\geq 1}$ and the cone $K$ (we do not need its explicit expression).

\begin{corollary}\label{cor:upper_bound_walks_strong}
	There exists a constant $C>0$ independent of $h,\ell,n$ such that
	\begin{equation}
		\P\left((\bm W_n(i))_{i\in[n]}\in\mathcal{W}_{Sb}^n\middle|\bm W_n(n)=(h,\ell)\right)\leq  C(1+ |(h,\ell)|^p) n^{-p-1}, \quad\text{for all}\quad n,h,\ell\in \Z_{\geq 0},
	\end{equation}
	where $|(h,\ell)|$ denotes the Euclidean norm of the vector $(h,\ell)$.
\end{corollary}

We also state the following consequence of \cite[Theorem 6]{MR3342657}.

\begin{corollary}\label{cor:upper_bound_walks_strong2}
	Fix $h,\ell\in \Z_{\geq 0}$. There exists a constant $C>0$ independent of $n$ such that
	\begin{equation}
		\P\left((\bm W_n(i))_{i\in[n]}\in\mathcal{W}_{Sb}^n\middle|\bm W_n(n)=(h,\ell)\right)\geq  C n^{-p-1}, \quad\text{for all}\quad n\in \Z_{\geq 0}.
	\end{equation}
\end{corollary}

\section{Asymptotic estimates}

\begin{proposition}\label{prop:asympt_estim_sum}
	Let $(a_n)_{n\in\Z_{\geq 0}}$ and $(b_n)_{n\in\Z_{\geq 0}}$ be two sequences of positive numbers such that for some $\alpha>1$:
	\begin{itemize}
		\item $a_n\sim a\cdot n^{-\alpha}$, for some $a>0$;
		\item $b_n\sim b\cdot n^{-\alpha}$, for some $b>0$;
	\end{itemize}
	Then as $n\to \infty$,
	$$\sum_{k=0}^n a_k\cdot b_{n-k}\sim n^{-\alpha}\left(a\cdot \sum_{k=0}^{\infty}b_k+b\cdot\sum_{k=0}^{\infty}a_k\right).$$
\end{proposition}

\begin{proof}
	We split the sum in four parts as follows:
	$\sum_{k=0}^na_kb_{n-k} =t_n+u_n+v_n+w_n$, where
	\begin{multline}
		t_n= \sum_{k=0}^{\lfloor n^{\alpha}\rfloor}a_kb_{n-k}, \quad u_n=  \sum_{k=\lfloor n^{\alpha}\rfloor+1}^{\lfloor n/2\rfloor }a_kb_{n-k},\\
		 v_n = \sum_{k=\lfloor n/2\rfloor +1}^{n-\lfloor n^{\alpha}\rfloor} a_kb_{n-k},\quad w_n=\sum_{k=n-\lfloor n^{\alpha}\rfloor +1}^n a_kb_{n-k}.
	\end{multline}
	By assumption, there exists $B>0$ such that for all $n\geq 1$, $b_n\leq B\cdot n^{-\alpha}$. Then 
	$$ u_n\leq B \sum_{k= \lfloor n^{\alpha}\rfloor+1}^{\lfloor n/2\rfloor }  a_k\cdot    (n-\lfloor n/2\rfloor)^{-\alpha}=o(n^{-\alpha}).$$
	Similarly,  $v_n=o\left( n^{-\alpha}\right)$.
	It remains to show that
	\begin{equation}
		t_n \sim n^{-\alpha} \cdot b\sum_{k=0}^{+\infty} a_k,\qquad
		w_n \sim n^{-\alpha} \cdot a\sum_{k=0}^{+\infty} b_k.
	\end{equation}
	We can write 
	$t_n- n^{-\alpha} \cdot b\sum_{k=0}^{+\infty} a_k = x_n+y_n+z_n$, where
	\begin{align}
		&x_n= b\sum_{k=0}^{\lfloor n^{\alpha}\rfloor} a_k \left((n-k)^{-\alpha}-n^{-\alpha} \right)  , \qquad y_n= \sum_{k=0}^{\lfloor n^{\alpha}\rfloor} a_k\left(b_{n-k}- b(n-k)^{-\alpha}\right), \notag\\
		 &z_n=-n^{-\alpha}\cdot b\sum_{k=\lfloor n^{\alpha} \rfloor +1}^{+\infty} a_k.
	\end{align}
	Obviously $\ z_n=o\left(n^{-\alpha}\right)$. For $x_n$ note that
	$$ 0\leq x_n\leq b\sum_{k=0}^{\lfloor n^{\alpha}\rfloor} a_k \left((n-\lfloor n^{\alpha}\rfloor)^{-\alpha}-n^{-\alpha} \right) \leq n^{-\alpha} \left( \left( 1- \frac{\lfloor n^{\alpha}\rfloor }{n}\right)^{-\alpha}-1\right)b\sum_{k=0}^{+\infty}a_k .$$ 
	Hence $\ x_n=o\left(n^{-\alpha}\right)$. Finally for $y_n$ we have that
	$$|y_n| \leq  \sum_{k=0}^{\lfloor n^{\alpha}\rfloor} a_k(n-k)^{-\alpha} \left| (n-k)^{\alpha}b_{n-k}-b\right| \leq (n-n^{\alpha})^{-\alpha}\left(\sup_{\ell\geq n-\lfloor n^{\alpha}\rfloor} |\ell^{\alpha}b_\ell-b|\right)\sum_{k=0}^{+\infty}a_k.$$
	Therefore  $\ y_n=o\left(n^{-\alpha}\right)$. This proves that
	$t_n \sim n^{-\alpha} \cdot b\sum_{k=0}^{+\infty} a_k.$ Similarly, we have that $w_n \sim n^{-\alpha} \cdot a\sum_{k=0}^{+\infty} b_k$.
\end{proof}

\section{Local estimates for one-dimensional random walks conditioned to be positive}

Let $(\bm S_{n})_{n\in\Z_{\geq 0}}$ be a one-dimensional random walk started at zero at time zero. Assume that the increments of $(\bm S_{n})_{n\in\Z_{\geq 0}}$ are i.i.d., centered and with finite variance $\sigma^2$. Let 
$$\bm \tau=\inf\{n>0|\bm{S}_n< 0\},$$ 
and $h$ be the function defined by
\begin{equation}\label{eq:efdbuebfwebfowe}
	h(x)=
	\begin{cases}
		1+\sum_{j=1}^{+\infty}\P(\bm S^1_{\bm \tau}+\dots +\bm S^j_{\bm \tau}> -x) \quad&\text{if}\quad x> 0,\\
		0 \quad&\text{otherwise},
	\end{cases}
\end{equation} 
where the random variables $\bm S^i_{\bm \tau}$ are independent copies of $\bm S_{\bm \tau}$.
We have the following local estimates (see for instance\footnote{We remark that the expression in Eq.(5) in \cite{MR4105264} is wrong. For a correct expression see \cite[Eq. (7)]{MR2449127}.} \cite[Lemma 2.2]{MR4105264} or \cite[Theorem 4]{MR2449127}):
\begin{lemma}\label{lem:local_est_one_dim}
	The following estimates hold
	\begin{align}
		&\P\left((x+\bm S_{i})_{i\in[n]}> 0\right)\sim 2\frac{\E[-\bm{S}_{\bm \tau}]}{\sigma\sqrt{2\pi}} \frac{h(x)}{\sqrt{n}},\quad\text{for all}\quad x\in\Z_{> 0},\\
		&\P(x+\bm S_{n}=y,(x+\bm S_{i})_{i\in[n]}> 0)\sim\frac{c}{\sigma\sqrt{2\pi}}\frac{h(x)\tilde{h}(y)}{n^{3/2}},\quad\text{for all}\quad x,y\in\Z_{> 0},\\
		&\P(x+\bm S_{n}=y,(x+\bm S_{i})_{i\in[n]}> 0)\leq C\frac{h(x)\tilde{h}(y)}{n^{3/2}},\quad\text{for all}\quad x,y\in\Z_{> 0}.
	\end{align}
	where $c=\frac{\E[-\bm{S}_{\bm \tau}]}{\sum_{z\in\zgz} \tilde{h}(z)\P(\bm{S}_1\leq -z)}$, $C$ is another constant independent of $n,x,y$, and $\tilde{h}$ is the function defined in \eqref{eq:efdbuebfwebfowe} for the walk $(-\bm S_{n})_{n\in\Z_{\geq 0}}$.
\end{lemma}

We also explicitly compute $\E[-\bm S_{\bm \tau}]$ in a specific case.

\begin{lemma}\label{lem:local_est_one_dim_expect}
	Assume that\footnote{Note that we are not assuming any restriction on the expression for $\P(\bm S_{1}= y)$ when $y\in \Z_{\geq0}$.} there exists two constants $\alpha>0$ and $\gamma>0$ such that $\P(\bm S_{1}= y)=\alpha\gamma^{-y}$ for all $y\in \Z_{<0}$. Then
	\begin{equation}
		\P(\bm{S}_{\bm \tau}=x)=(1-\gamma)\gamma^{-x-1}, \quad\text{for all}\quad x\in\Z_{<0},\qquad\text{and}\qquad\E[-\bm S_{\bm \tau}]=\frac{1}{1-\gamma}.
	\end{equation}
\end{lemma}
\begin{proof}
	Since $\bm \tau$ is a.s.\ finite, for every $x\in\Z_{<0}$,
	\begin{multline}
		\P(\bm{S}_{\bm \tau}=x)=\sum_{h=1}^{\infty}\P(\bm S_{h}=x,(\bm S_{i})_{i\in[h-1]}\geq 0)\\
		=\sum_{h=1}^{\infty}\sum_{x'\in\Z_{\geq 0}}\P(\bm S_{h-1}=x',(\bm S_{i})_{i\in[h-1]}\geq 0)\P(\bm S_{h}=x|\bm S_{h-1}=x').
	\end{multline}
	Using that by assumption  $\P(\bm S_{h}=x|\bm S_{h-1}=x')=\alpha\gamma^{x'-x}$ and setting 
	$$C=\sum_{h=1}^{\infty}\sum_{x'\in\Z_{\geq 0}}\P(\bm S_{h-1}=x',(\bm S_{i})_{i\in[h-1]}\geq 0)\alpha\gamma^{x'},$$ 
	we get that $\P(\bm{S}_{\bm \tau}=x)=C\cdot\gamma^{-x}$. Using that $1=\sum_{x\in\Z_{<0}}\P(\bm{S}_{\bm \tau}=x)=\sum_{x\in\Z_{<0}}C\cdot\gamma^{-x}=\frac{C\gamma}{1-\gamma}$. We conclude that $\P(\bm{S}_{\bm \tau}=x)=(1-\gamma)\gamma^{-x-1}$. The statement for the expectation follows with a basic computation.
\end{proof}

\section{Technical proofs}\label{sect:techproofs}

We give here the details of the proofs of Propositions \ref{prop:jfwvbvouvwefuow} and \ref{prop:jfbeouwfboufwe} that we skipped in Section~\ref{sec:cond_conv}.

\begin{proof}[Proof of Proposition~\ref{prop:jfwvbvouvwefuow}]
	The convergence in distribution ${\conti W}_n\xraninf{d}\conti E_{\rho}$ follows from Proposition~\ref{prop:scaling_strong_walk} together with Corollary~\ref{cor:the_walk_is_uniform_strong}.
	
	Now let $0<\varepsilon<\min\{u,(1-u)\}$. Note that $\left(({\conti W}_n-{\conti W}_n(u))|_{[u,1-\varepsilon]},\conti{Z}^{(u)}_n|_{[u,1-\varepsilon]}\right)$
	is a measurable functional of $(\bm W_{n}(k) - \bm W_{n}(\lfloor \varepsilon n \rfloor))_{\lfloor \varepsilon n \rfloor\leq k \leq \lfloor (1-\varepsilon)n \rfloor}$. Using Theorem~\ref{thm:fvwuofbgqfipqhfqfpq_strongb} together with standard absolute continuity arguments (see for instance
	\cite[Proposition A.1]{borga2021skewperm}), we obtain that 
	\begin{equation}
		\left(({\conti W}_n - \conti W_n(u))|_{[u,1-\varepsilon]},\conti{Z}^{(u)}_n|_{[u,1-\varepsilon]} \right)
		\xraninf{d}
		\left((\conti{E}_\rho - \conti{E}_\rho(u))|_{[u,1-\varepsilon]},\conti Z_{\rho,q}^{(u)}|_{[u,1-\varepsilon]}\right).
	\end{equation}
	Using two times Prokorov's theorem, as already done in Theorem~\ref{thm:fvwuofbgqfipqhfqfpq_strongb}, we have that the sequence
	\begin{equation}\label{eq:certwv_1}
		\left( \conti W_n,\,\,\left(({\conti W}_n - \conti W_n(u))|_{[u,1-\varepsilon]},{\conti Z}^{( u)}_n|_{[u,1-\varepsilon]}\right)_{\varepsilon \in \mathbb Q \cap (0,u\wedge 1-u)}
		\right)
	\end{equation}
	is tight and its limit in distribution must be
	\begin{equation}\label{eq:cafafv_2}
		\left( \conti{E}_\rho,\,\,\left((\conti{E}_\rho - \conti{E}_\rho(u))|_{[u,1-\varepsilon]},\conti Z_{\rho,q}^{(u)}|_{[u,1-\varepsilon]}\right)_{\varepsilon \in \mathbb Q \cap (0,u\wedge 1-u)}
		\right)
	\end{equation}
	because ${\bm Z}^{( u)}_n|_{[u,1-\varepsilon]}$ is a measurable functional of $(\bm W_{n}(k) - \bm W_{n}(\lfloor \varepsilon n \rfloor))_{\lfloor \varepsilon n \rfloor\leq k \leq \lfloor (1-\varepsilon)n \rfloor}$ and the restriction operation $\conti W_n\mapsto(\conti W_n - \conti W_n(u))|_{[u,1-\varepsilon]}$ is continuous. Therefore we obtain convergence in distribution of the sequence in \eqref{eq:certwv_1} to the limit in \eqref{eq:cafafv_2}. By Skorokhod's theorem we can now assume that a.s., we have uniform convergence on $[0,1]$ of ${\conti W}_n$ to $\conti{E}_\rho$, and uniform convergence on $[u,1-\varepsilon]$ of $\conti{Z}^{(u)}$ to $\conti Z_{\rho,q}^{(u)}$ for every $\varepsilon>0, \varepsilon\in\mathbb{Q}$. This is enough to conclude the proof.
\end{proof}

\begin{proof}[Proof of Proposition~\ref{prop:jfbeouwfboufwe}]	
	Fix $u_1,\ldots,u_k\in (0,1)$. The fact that $\conti Z_{\rho,q}^{(u)}$ is a measurable functions of $\conti{E}_\rho$ and tightness, give that convergence in distribution in \eqref{eq:coal_con_cond} holds jointly for $u\in \{u_1,\ldots,u_k\}$. Therefore, for every continuous bounded $\varphi: \czord\times (\mathcal C([0,1],\R)\times \mathcal C([0,1),\R))^{\zgz}\to \R$, it holds that 
	\begin{equation}
		\E\left[\varphi\left(
		{\conti W}_n,\left({\conti Z}^{( u_i)}_n\right)_{i\in [k]}
		\right)
		\right] \to 
		\E\left[\varphi
		\left( \conti{E}_\rho,\left({\conti Z}^{( u_i)}_{\rho,q}, \right)_{i\in [k]}\right)	
		\right].
	\end{equation}
	Using dominated convergence theorem, it is possible to integrate this over $u_1,\ldots,u_k\in [0,1]$. Therefore, using Fubini--Tonelli's theorem, we can conclude that
	\begin{equation}
		\E\left[\varphi\left(
		{\conti W}_n,\left({\conti Z}^{(\bm u_i)}_n\right)_{i\in [k]}
		\right)
		\right] \to 
		\E\left[\varphi
		\left( \conti{E}_\rho,\left({\conti Z}^{(\bm u_i)}_{\rho,q}\right)_{i\in [k]}\right)	
		\right].
	\end{equation}
	Noting that $k$ is arbitrary, we obtain convergence in distribution in the product topology.	
\end{proof}

%%%%%%%%%%%%%%%%%%%%%%%%%%%%%%%%%%%%%%%%%%%%%%%%%%%%%%%%%%%%%%%%%%%
%%                                                               %%
%% Supplementary Material, if any, should be provided in         %%
%% {supplement} environment  with title and short description.   %%
%%                                                               %%
%%%%%%%%%%%%%%%%%%%%%%%%%%%%%%%%%%%%%%%%%%%%%%%%%%%%%%%%%%%%%%%%%%%

%\begin{supplement}
%\stitle{Title of Supplement A.}
%\sdescription{Short description of Supplement A.}
%\end{supplement}
%\begin{supplement}
%\stitle{Title of Supplement B.}
%\sdescription{Short description of Supplement B.}
%\end{supplement}

%%%%%%%%%%%%%%%%%%%%%%%%%%%%%%%%%%%%%%%%%%%%%%%%%%%%%%%%%%%%%%%%%%%
%%                                                               %%
%% Use the two commands below for producing your bibliography    %%
%% with bibtex, then comment again the commands and include the  %%
%% content of the .bbl file in this file below the commands.     %%
%%                                                               %%
%%%%%%%%%%%%%%%%%%%%%%%%%%%%%%%%%%%%%%%%%%%%%%%%%%%%%%%%%%%%%%%%%%%

%\bibliographystyle{amsplain}
%\bibliography{mybib}

% add below the content of your .bbl file produced by bibtex.

\providecommand{\bysame}{\leavevmode\hbox to3em{\hrulefill}\thinspace}
\providecommand{\MR}{\relax\ifhmode\unskip\space\fi MR }
\providecommand{\MRhref}[2]{%
	\href{http://www.ams.org/mathscinet-getitem?mr=#1}{#2}
}
\providecommand{\href}[2]{#2}

%%%%%%%%%%%%%%%%%%%%%%%%%%%%%%%%%%%%%%%%%%%%%%%%%%%%%%%%%%%%%%%%%%%
%%                                                               %%
%% You may add acknowledgments (optional).                       %%
%%                                                               %%
%%%%%%%%%%%%%%%%%%%%%%%%%%%%%%%%%%%%%%%%%%%%%%%%%%%%%%%%%%%%%%%%%%%
\begin{acks}
The author is very grateful to Mathilde Bouvel, Valentin Féray and Grégory Miermont for some enlightening discussions.
Thanks to Mireille Bousquet-Mélou, Kilian Raschel, and Antoine Lejay for pointing out some interesting papers.
Thanks also to Denis Denisov for suggesting several pointers for the proof of Proposition~\ref{prop:upper_bound_walks}. 
\end{acks}

%%%%%%%%%%%%%%%%%%%%%%%%%%%%%%%%%%%%%%%%%%%%%%%%%%%%%%%%%%%%%%%%%%%
%%                                                               %%
%% You have reached the end of your document.                    %%
%%                                                               %%
%%%%%%%%%%%%%%%%%%%%%%%%%%%%%%%%%%%%%%%%%%%%%%%%%%%%%%%%%%%%%%%%%%%

\end{document}